\documentclass[11pt]{amsart}
\usepackage{amscd,amsxtra,amssymb,mathrsfs, bbm}
\usepackage{stmaryrd}
\usepackage{amscd,amsxtra,amssymb, bbm, url}
\usepackage[new]{old-arrows}
\usepackage{subfiles}
\usepackage{faktor}
\usepackage{enumerate}
\usepackage{float}

\usepackage{tikz,tikz-3dplot,tikz-cd,graphicx}
\usetikzlibrary{positioning}
\usetikzlibrary{decorations.markings}

\usepackage{geometry}
        {\begin{quotation}\begin{center}\begin{em}}
        {\par\end{em}\end{center}\end{quotation}}

\newtheorem{theorem}{Theorem}[section]
\newtheorem{corollary}[theorem]{Corollary}
\newtheorem{lemma}[theorem]{Lemma}
\newtheorem{proposition}[theorem]{Proposition}

\theoremstyle{definition}
\newtheorem{definition}[theorem]{Definition}
\newtheorem{remark}[theorem]{Remark}
\newtheorem{example}[theorem]{Example}

\DeclareMathOperator{\rk}{\mathsf{rk}}

\renewcommand{\ker}{\mathsf{ker}}

\newcommand{\id}{\mathrm{id}}

\DeclareMathOperator{\Coh}{\mathsf{Coh}}

\DeclareMathOperator{\Tor}{\mathsf{Tor}}

\DeclareMathOperator{\VB}{\mathsf{VB}}

\DeclareMathOperator{\Hom}{\mathsf{Hom}}

\DeclareMathOperator{\Ext}{\mathsf{Ext}}

\DeclareMathOperator{\Aut}{\mathsf{Aut}}
\DeclareMathOperator{\End}{\mathsf{End}}
\DeclareMathOperator{\Mat}{\mathsf{Mat}}

\DeclareMathOperator{\Ob}{\mathsf{Ob}}

\DeclareMathOperator{\Red}{Red}
\DeclareMathOperator{\Image}{\mathsf{Im}}
\DeclareMathOperator{\lcm}{lcm}

\newcommand{\E}{\mathbf{E}}
\newcommand{\F}{\mathbf{F}}
\newcommand{\NC}{\mathsf{NC}}
\newcommand{\Rad}{\mathsf{Rad}}
\newcommand{\diag}{\mathsf{diag}}

\newcommand{\kk}{\mathbbm{k}}
\newcommand{\dd}{\mathbbm{d}}
\newcommand{\KK}{\mathbbm{K}}
\newcommand{\CC}{\mathbb{C}}
\newcommand{\RR}{\mathbb{R}}
\newcommand{\NN}{\mathbb{N}}
\newcommand{\ZZ}{\mathbb{Z}}

\newcommand{\ff}{\mathbbm{f}}

\renewcommand{\dim}{\mathsf{dim}}

\newcommand{\twoheadarrow}{\rightarrow\mathrel{\mkern-14mu}\rightarrow}

\newcommand{\FF}{\mathbb{F}}

\newcommand{\EE}{\mathbb{E}}

\newcommand{\cD}{\mathsf{D}}

\newcommand{\cA}{\mathsf{A}}

\renewcommand{\ff}{\mathbbm{f}}
\renewcommand{\gg}{\mathbbm{g}}

\newcommand{\XX}{\mathbb{X}}

\newcommand{\PP}{\mathbbm{P}}

\newcommand{\sHom}{\mbox{\it{Hom}}}

\newcommand{\kA}{\mathcal{A}}

\newcommand{\kB}{\mathcal{B}}

\newcommand{\kF}{\mathcal{F}}
\newcommand{\kG}{\mathcal{G}}
\newcommand{\kH}{\mathcal{H}}
\newcommand{\kI}{\mathcal{I}}
\newcommand{\kJ}{\mathcal{J}}
\newcommand{\kO}{\mathcal{O}}
\newcommand{\kL}{\mathcal{L}}
\newcommand{\kP}{\mathcal{P}}

\newcommand{\kR}{\mathcal{R}}
\newcommand{\kK}{\mathcal{K}}

\newcommand{\kN}{\mathcal{N}}

\newcommand{\kT}{\mathcal{T}}
\newcommand{\kS}{\mathcal{S}}
\newcommand{\kX}{\mathcal{X}}

\newcommand{\kZ}{\mathcal{Z}}

\newcommand{\lar}{\longrightarrow}
\newcommand{\llangle}{\langle\!\langle}
\newcommand{\rrangle}{\rangle\!\rangle}

\newcommand{\llbrace}{(\!(}
\newcommand{\rrbrace}{)\!)}

\def\sA{\mathsf A}

\def\sD{\mathsf D} 
\def\sE{\mathsf E} 
\def\sF{\mathsf F} \def\sS{\mathsf S}
\def\sG{\mathsf G} \def\sT{\mathsf T}

\def\kron#1#2{\xymatrix@C=2em{{#1}\ar@/^3pt/[r]\ar@/_3pt/[r]&{#2}}}
\def\bu{{\scriptscriptstyle\bullet}}

\usepackage{tikz}
\usetikzlibrary{calc, arrows, positioning, shapes, fit, matrix, decorations}
\usetikzlibrary{decorations.shapes, decorations.pathreplacing}

\tikzset{
  decorate with/.style={decorate,decoration={shape backgrounds,shape=#1,shape size=1.5mm}},
   deco/.style={decorate with=dart},
   ordi/.style={draw,-stealth,  thick},
   conj/.style={dashed, draw, thick},
   ve/.style={circle, draw, thick, fill=blue!20, inner sep=1pt, outer sep=2pt, minimum size=7pt},
    dot/.style={fill=blue!10,circle,draw, inner sep=1pt, minimum size=5pt},
  dv/.style={star,star points=5,
star point ratio=2, draw, thick, fill=green!20, inner sep=1pt,outer sep=2pt,minimum size=7pt}
}

\tikzset{
    tbl5 nodes/.style={
        rectangle,
        execute at begin node=$,
       execute at end node=$,
       fill=blue!5,
        align=center,
        text depth=0.5ex,
        text height=2ex,
        inner xsep=0pt,
        outer sep=0pt,
           },
    tbl5/.style={
        matrix of nodes,
        row sep=-\pgflinewidth,
        column sep=-\pgflinewidth,
        nodes={
            tbl5 nodes
        },
        execute at empty cell={\node[draw=none]{};}
    }
  }

\input xy
\xyoption{all}

\title[Non-crossing partitions for exceptional curves]{Non-crossing partitions for exceptional hereditary curves}

\date{December 1, 2025}

\author[Baumeister]{Barbara Baumeister}
\email{b.baumeister@math.uni-bielefeld.de}
\address{Fakultät für Mathematik, Universität Bielefeld, Universitätsstraße 25, 33501 Bielefeld}

\author[Burban]{Igor Burban}
\email{burban@math.uni-paderborn.de}
\address{Institut für Mathematik, Universität Paderborn, Warburger Str. 100, 33098 Paderborn}

\author[Neaime]{Georges Neaime}
\email{gneaime@math.uni-bielefeld.de}
\address{Fakultät für Mathematik, Universität Bielefeld, Universitätsstraße 25, 33501 Bielefeld}

\author[Schwabe]{Charly Schwabe}
\email{cschwabe@math.uni-paderborn.de}
\address{Institut für Mathematik, Universität Paderborn, Warburger Str. 100, 33098 Paderborn}

\begin{document}

\begin{abstract}
We introduce a new class of reflection groups associated with the canonical bilinear lattices of Lenzing, which we call reflection groups of canonical type. The main result of this work is a categorification of the corresponding poset of non-crossing partitions for any such group, realized via the poset of thick subcategories of the category of coherent sheaves on an exceptional hereditary curve generated by an exceptional sequence. A second principal result, essential for the categorification, is a proof of the transitivity of the Hurwitz action in these reflection groups. 
\end{abstract}

\maketitle

\setcounter{tocdepth}{1}  
\tableofcontents

\section{Introduction}

Let $\kk$ be a field and let $\cD$ be a $\Hom$-finite $\kk$-linear triangulated category admitting a full exceptional sequence $\bigl(E_1, \dots, E_n\bigr)$; see \cite{Bondal, Helices}.  Then the Grothendieck group
$\Gamma = K_0(\cD)$ is free of rank $n$. Moreover, we have a non-degenerate (in general, non-symmetric) biadditive form $K: \Gamma \times \Gamma \lar \ZZ$ (the Euler form of $\cD$); hence $(\Gamma, K)$ is a so-called \emph{bilinear lattice} in the sense of \cite{LenzingKTheory, HuberyKrause}. The classes $[E_1], \dots, [E_n]$ are so-called \emph{pseudo-roots} of $(\Gamma, K)$, and they form a basis of $\Gamma$.

Let $B: \Gamma \times \Gamma \lar \ZZ$ be the symmetrization of the form $K$, and let $\mathsf{O}(\Gamma, B)$ denote the corresponding group of isometries of $(\Gamma, B)$. Then we obtain the following group-theoretic objects:
\begin{itemize}
\item[(a)] A reflection group $W \subseteq \mathsf{O}(\Gamma, B)$.
\item[(b)] The set of reflections $T \subset W$.
\item[(c)] A distinguished element $c \in W$, called the Coxeter element (defined by the action of the Auslander--Reiten functor of $\cD$).
\item[(d)] The set of real roots $\Phi \subset \Gamma$.
\end{itemize}

Next, we have a reflection length function $\ell_T: W \lar \NN$ as well as the corresponding absolute order $\le_T$ on $W$. The main object of study in this work is the associated poset of \emph{non-crossing partitions}
\begin{equation*}
\mathsf{NC}_T(W, c) := \bigl\{w \in W  \big|  \mathbbm{1} \le_T w \le_T c \bigr\}.
\end{equation*}
The name ``non-crossing partitions’’ arises from the special case where $W = S_n$ is the symmetric group on $n$ elements. In this case, $c$ is an $n$-cycle (e.g.~$c = (1 2 \dots n)$), $T$ is the set of simple transpositions, and the elements of $\mathsf{NC}_T(W, c)$ can be identified with “non-crossing partitions’’ of the set $\bigl\{1, \dots, n\bigr\}$; see, e.g., \cite[Section 4]{RingelCatalanCombinatorics} for a detailed exposition as well as \cite{NonCrossing} for an overview of applications of non-crossing partitions in various fields of mathematics.

There is a natural braid group action on the set of complete exceptional sequences in $\cD$; see  \cite{Bondal, Helices}.  In the case this action is \emph{transitive}, 
the datum $(W, T, c)$ and $\Phi \subset \Gamma$ (and therefore the associated poset $\mathsf{NC}_T(W, c)$) is completely determined by a triangulated category $\cD$ as above and is independent of the choice of a complete exceptional sequence $\bigl(E_1, \dots, E_n\bigr)$.

Of particular interest is the case when $\cD = D^b(\cA)$ is the bounded derived category of an $\Ext$-finite $\kk$-linear hereditary abelian category $\cA$, which is noetherian and admits a tilting object. By a result of Happel and Reiten \cite[Theorem 2.8]{HappelReiten}, any such (connected) category $\cA$ is equivalent either to the module category $A$--$\mathsf{mod}$, where $A$ is a finite-dimensional hereditary $\kk$-algebra, or to the category $\Coh(\XX)$ of coherent sheaves on an \emph{exceptional non-commutative hereditary curve} $\XX$; see also \cite{HappelTilting} for the special case when $\kk$ is algebraically closed.
\begin{itemize}
\item[(a)] In the case $\cA = A$--$\mathsf{mod}$, the corresponding group $W$ is a \emph{crystallographic Coxeter group}. More precisely, $W$ is the \emph{Weyl group} of the symmetrizable Kac--Moody Lie algebra associated with the Cartan matrix of $A$; see \cite[Appendix B]{HuberyKrause} for a detailed discussion. Moreover, all such Weyl groups arise in this way.
\item[(b)] For $\cA = \Coh(\XX)$, we obtain a very interesting new class of discrete groups, which we call \emph{reflection groups of canonical type}. Depending on the representation type of $\XX$, the associated group $W_{\XX}$ is either an \emph{affine Weyl group}, an \emph{elliptic Weyl group} \cite{SaitoI}, or a \emph{cuspidal canonical reflection group}. All affine Weyl groups, as well as all elliptic Weyl groups of codimension one, arise in this way.
\end{itemize}

In both cases, the structure of the set of isomorphism classes of indecomposable objects of $\cA$ (and hence of $\cD$) is controlled by the bilinear lattice $(\Gamma, K)$. For example, two exceptional objects $E, F \in \mathsf{Ob}(\cA)$ are isomorphic if and only if $[E] = [F] \in \Gamma$; see \cite[Section 5]{LenzingSurvey} and references therein. Moreover, it turns out that $[E] \in \Phi$ for any exceptional object $E$ in $\cA$.

It turns out that the poset $\mathsf{NC}_{T}(W, c)$ admits a categorical description. For a hereditary category $\cA$ as above, one can consider the set $\mathsf{Ex}(\cA)$ of its thick exact subcategories generated by an exceptional sequence. This set becomes a partially ordered set with respect to inclusion of subcategories.
It turns out that there is a well-defined map
\begin{equation}\label{E:CoxMap}
\mathsf{cox}: \mathsf{Ex}(\cA) \lar \mathsf{NC}_{T}(W, c), \quad \llangle F_1, \dots, F_r\rrangle \mapsto s_{[F_1]} \dots s_{[F_r]}.
\end{equation}
Here, $\llangle F_1, \dots, F_r\rrangle$ denotes the thick exact (in fact, abelian) subcategory of $\cA$ generated by an exceptional sequence $\bigl(F_1, \dots, F_r\bigr)$, and $s_{[F_i]} \in W$ is the reflection corresponding to the class $[F_i] \in \Phi$ for each $1 \le i \le r$. If $\cA = A$--$\mathsf{mod}$ for a finite-dimensional hereditary algebra $A$, then the map \eqref{E:CoxMap} is an isomorphism of posets. This was proven by Ingalls and Thomas \cite{IngallsThomas} for path algebras of representation-finite and tame quivers, later extended by Igusa and Schiffler \cite{IgusaSchiffler} to arbitrary path algebras, and finally established in full generality by Hubery and Krause \cite{HuberyKrause}.

We now wish to emphasize the role of the base field $\kk$ in this context. If $\kk$ is algebraically closed, then any finite-dimensional hereditary algebra $A$ is Morita equivalent to the path algebra of a finite quiver without loops or oriented cycles. However, this is no longer true over a non-algebraically closed field. In fact, this is “not a bug but a feature,” since the existence of finite (skew-)field extensions of $\kk$ makes it possible to categorify crystallographic root systems with simple roots of different lengths.

Generalizations of the bijection \eqref{E:CoxMap} to the case where $\cA = \Coh(\XX)$ is the category of coherent sheaves on a weighted projective line $\XX$ of Geigle and Lenzing \cite{GeigleLenzingWeightedCurves} were studied by Baumeister, Wegener and Yahiatene in \cite{BaumeisterWegenerYahiateneI, BaumeisterWegenerYahiateneII}. At this point, however, a surprise occurs: while $\mathsf{cox}$ is a bijection when $\XX$ is domestic or wild, this no longer holds when $\XX$ is tubular. To remedy this “defect,” one must replace $W$ by its \emph{hyperbolic extension} $\widetilde{W}$; see \cite{BaumeisterWegener}. 
At this point we emphasize that, from the perspective of studying elliptic Artin groups, passing from an elliptic Weyl group to its hyperbolic extension is entirely natural; see, for example, \cite{SaitoII}, \cite{VanderLekThesis} and \cite{TakahashiEtAl}.

In this paper, we deal with the case of an arbitrary exceptional hereditary curve $\XX$ over an arbitrary base field $\kk$. If $\kk$ is algebraically closed, then such $\XX$ is a weighted projective line (this follows from the vanishing of the Brauer group $\mathsf{Br}\bigl(\kk(t)\bigr)$ of the field of rational functions $\kk(t)$ due to Tsen's Theorem). Over arbitrary fields, however, the class of exceptional hereditary curves is considerably broader.

A key feature of such a curve $\XX$ is the existence of an exact equivalence of derived categories
$$D^b\bigl(\Coh(\XX)\bigr) \simeq D^b\bigl(\Sigma\text{\,--\,}\mathsf{mod}),$$
where $\Sigma$ is an appropriate canonical algebra of Ringel \cite{RingelCrawleyBoevey}. Starting with a canonical algebra $\Sigma$ as a ``primary object'', one can construct a distinguished $t$-structure on $D^b\bigl(\Sigma$--$\mathsf{mod})$, whose heart is equivalent to $\Coh(\XX)$. This allows one to describe $\Coh(\XX)$ in an axiomatic way. This approach was initiated by Lenzing in \cite{LenzingCurveSingularities}, used by Happel and Reiten in \cite{HappelReiten} and further developed by Kussin \cite{KussinMemoirs, KussinWeightedCurve}.

A direct description of $\XX$ in terms of non-commutative algebraic geometry, based on \cite{ArtindeJong, BurbanDrozd}, was initiated in \cite{Burban}. Such an $\XX$ is a ringed space $(X, \kH)$, where $X$ is a proper regular curve over $\kk$ of genus zero and $\kH$ is a sheaf of hereditary orders. The curve $\XX$ is exceptional if and only if the corresponding Brauer class $\eta_\XX \in \mathsf{Br}(\KK)$ is exceptional, where $\KK$ is the field of rational functions of $X$.  Moreover, up to Morita equivalence, such an $\XX$ is determined by the datum $(X, \eta, \rho)$, where $X$ is a proper regular genus-zero curve, $\eta \in \mathsf{Br}(\KK)$ is an exceptional class, and $\rho: X_\circ \lar \NN$ is a weight function on the set of closed points of $X$.

The bilinear lattice $(\Gamma, K)$ of $D^b\bigl(\Sigma$--$\mathsf{mod})$ (and, consequently, of $D^b\bigl(\Coh(\XX)\bigr)$) for a canonical algebra $\Sigma$ was studied by Lenzing in \cite{LenzingKTheory}, where it is named \emph{canonical}. Its invariants are captured by the corresponding \emph{symbol} $\sigma_{\XX}$, which is a table
\begin{equation}\label{E:SymbolIntro}
\sigma_{\Sigma} = \sigma_{\XX} = \left( \begin{array}{ccc|c} 
            p_1 & \dots & p_t &\\
            d_1 & \dots & d_t & \varepsilon\\
            f_1 & \dots & f_t & \end{array} \right)
\end{equation}
where $t \in \NN$, $\varepsilon \in \bigl\{1, 2\bigr\}$, $p_1, \dots, p_t \in \NN_{\ge 2}$, and $d_1, \dots, d_t$, $f_1, \dots, f_t \in \NN$ satisfy $f_i \mid d_i$ for all $1 \le i \le t$.
Key properties of $\sigma_{\XX}$ (and hence of $\XX$) are captured by the parameter
\begin{equation*}
\delta_{\XX}:=
\left(\sum\limits_{i = 1}^t \varepsilon d_i \Bigl(1- \frac{1}{p_i}\Bigr)\right) -2.
\end{equation*}
Regarding representation type, an exceptional hereditary curve $\XX$ is domestic if $\delta_{\XX} < 0$, tubular if $\delta_{\XX} = 0$, and wild if $\delta_{\XX} > 0$. Viewing $\XX$ as a ringed space allows one to give natural interpretations to all parameters arising in the symbol $\sigma_{\XX}$.

We call reflection groups arising from exceptional hereditary curves \emph{reflection groups of canonical type}. It turns out that such groups of domestic type are precisely the affine Weyl groups, whereas the tubular ones are precisely elliptic Weyl groups of codimension one. In the wild case, we obtain an interesting new class of discrete groups, called cuspidal reflection groups of canonical type. The main result of this work is the following:

\smallskip
\noindent
\textbf{Theorem A} (see Theorems \ref{T:MainCategorification} and \ref{T:MainCategorificationII}). Let $\XX$ be an exceptional hereditary curve and $(W, T, c)$ be the corresponding generalized dual Coxeter datum.
\begin{itemize}
\item[(a)] If $\XX$ is domestic or wild, then the map $\mathsf{cox}: \mathsf{Ex}\bigl(\Coh(\XX)\bigr) \lar \mathsf{NC}_{T}(W, c)$ is a bijection of posets.
\item[(b)] If $\XX$ is tubular, then the map $\widetilde{\mathsf{cox}}: \mathsf{Ex}\bigl(\Coh(\XX)\bigr) \lar \mathsf{NC}_{\widetilde{T}}(\widetilde{W}, \tilde{c})$ is a bijection of posets, where $\widetilde{W}$ is the hyperbolic extension of $W$, $\tilde{c}$ is the corresponding Coxeter element, and $\widetilde{T}$ is the associated set of reflections.
\end{itemize}

Another principal result of independent interest, which also plays a crucial role in the proof of the previous theorem, is the transitivity of the Hurwitz action on the set $\mathrm{Red}_{T}(c)$ (respectively,
$\mathrm{Red}_{\widetilde{T}}(\tilde{c})$) of reduced reflection factorizations of the Coxeter element $c$ (respectively, $\tilde{c}$) in the domestic/wild (respectively, tubular) cases. We formulate this result in the next theorem. Note that it was proven in \cite{BaumeisterWegenerYahiateneI, BaumeisterWegenerYahiateneII} for the special case when $\XX$ is a weighted projective line, using a different approach.

\smallskip
\noindent
\textbf{Theorem B} (see Theorem \ref{newMainTransitivity} and Corollaries \ref{C:HurwitzTransitivityNonTub} and \ref{C:HurwitzTransitivityTub}). Let $\XX$ be an exceptional hereditary curve and $(W, T, c)$ be the corresponding generalized dual Coxeter datum.
\begin{itemize}
\item[(a)] If $\XX$ is domestic or wild, then the Hurwitz action on $\mathrm{Red}_{T}(c)$ is transitive.
\item[(b)] If $\XX$ is tubular, then the Hurwitz action on $\mathrm{Red}_{\widetilde{T}}(\tilde{c})$ is transitive.
\end{itemize}

The structure of this paper is as follows. In Section \ref{S:BilinearLattices}, we recall, following \cite{HuberyKrause}, the theory of bilinear lattices $(\Gamma, K)$ and exceptional sequences therein. We also review the definition of the associated datum $(W, T, c)$ together with some basic properties of the objects involved, and we establish several technical results concerning the reflection length of the Coxeter element $c$.

In Section \ref{S:ExceptCollections}, we recall the definition and main properties of (full) exceptional sequences in a $\Hom$-finite $\kk$-linear triangulated category $\cD$ over an arbitrary field $\kk$, with particular emphasis on the case $\cD = D^b(\cA)$, where $\cA$ is an $\Ext$-finite $\kk$-linear hereditary abelian category. This and the previous sections are primarily expositional.

In Section \ref{S:ExceptionalHereditary}, we lay the foundations for the theory of exceptional hereditary curves. The first key result is Theorem \ref{T:Lenzing}, which establishes a derived equivalence $$D^b\bigl(\Coh(\EE)\bigr) \simeq  D^b\bigl(\Lambda\text{\,--\,}\mathsf{mod}),$$
where $\EE$ is an exceptional homogeneous curve and $\Lambda$ is a tame hereditary algebra with two non-isomorphic simple modules. Then we proceed to the case of arbitrary exceptional hereditary curves $\XX$ and give a description of their invariants arising in the corresponding symbol $\sigma_{\XX}$; see (\ref{E:SymbolIntro}). The next key result is given by Theorem \ref{T:ExceptionalMain}, giving a construction of a distinguished full exceptional sequence in the derived category $D^b\bigl(\Coh(\XX)\bigr)$ and describing the corresponding Gram matrix in terms of the symbol $\sigma_{\XX}$. All together, it provides a full description of the bilinear lattice $(\Gamma, K)$ corresponding to the triangulated category  $D^b\bigl(\Coh(\XX)\bigr)$.

In Section \ref{S:ReflectionGroups}, we introduce reflection groups of canonical type associated with canonical bilinear lattices, as well as the corresponding hyperbolic extensions. We highlight here Corollary \ref{C:reflectionLengthHyp} and Proposition \ref{P:reflectionlengthCox}, which describe the reflection length of the Coxeter element in the tubular case. 

The most technically demanding part of our work is carried out in Section \ref{S:HurwitzTransitivity}, which forms the ``heart’’ of the paper. In Theorem \ref{newMainTransitivity}, we prove the transitivity of the Hurwitz  action on the sets $\mathrm{Red}_{T}(c)$ (respectively, $\mathrm{Red}_{\widetilde{T}}(\tilde{c})$) of reduced reflection factorizations of the Coxeter element $c$ (respectively, $\tilde{c}$) in the non-tubular and tubular cases, respectively. The main feature of our uniform proof of transitivity is the existence of an epimorphism into a Coxeter group that satisfies the necessary properties of Theorem \ref{newMainTransitivity}.

After these preparations, we establish in Section \ref{S:OrderPreservingBijections} our main results: Theorem \ref{T:MainCategorification} and Theorem \ref{T:MainCategorificationII}, which assert that the maps
$$
\mathsf{cox}: \mathsf{Ex}\bigl(\Coh(\XX)\bigr) \lar \mathsf{NC}_{T}(W, c)
\;\;  \text{and} \; \; 
\widetilde{\mathsf{cox}}: \mathsf{Ex}\bigl(\Coh(\XX)\bigr) \lar \mathsf{NC}_{\widetilde{T}}(\widetilde{W}, \tilde{c})
$$
are poset isomorphisms in the non-tubular and tubular cases, respectively.

We conclude the paper with two appendices. In Appendix \ref{S:Comparison}, we recall and elaborate a dictionary relating symbols of domestic (respectively, tubular) types to affine (respectively, elliptic) root systems. In Appendix \ref{S:Appendix}, we further expand  the theory of hyperbolic extensions developed by Saito in \cite{SaitoI}; see also \cite{BaumeisterWegener}.

\smallskip
\noindent
\emph{Acknowledgements}. This work was partially supported by the German Research Foundation
SFB-TRR 358/1 2023 -- 491392403. We are very grateful to Daniel Perniok for his explanations of the works \cite{RingelCrawleyBoevey} and \cite{LenzingKTheory} and  many fruitful discussions.

\section{Bilinear lattices}\label{S:BilinearLattices}
In this section, we recall basic notions such as bilinear lattices and reflection groups. We also introduce non-crossing partitions as well as exceptional sequences and the braid group action on them. These are some of the central notions in our work.

\subsection{Generalities on bilinear lattices and associated reflection groups}
\noindent
Following \cite{LenzingKTheory, HuberyKrause}, we recall the following basic definitions. 

\begin{definition}\label{D:BilLattice}
A \emph{bilinear lattice} is a pair $(\Gamma, K)$, where $\Gamma$ is a free abelian group of finite rank and $K = \langle -,\,-\rangle: \; \Gamma \times \Gamma \lar \ZZ$ is a non-degenerate (possibly non-symmetric) bilinear form. Here, non-degenerate means that $\langle \alpha, \,-\rangle = 0$ implies $\alpha = 0$ and 
$\langle -, \,\beta\rangle = 0$  implies $\beta = 0$. 
\end{definition}

\smallskip
\noindent
From now on, let $(\Gamma, K)$ be a bilinear lattice. 

\begin{definition}\label{D:Coxeter element}  A \emph{Coxeter element}  of $(\Gamma, K)$ is a group isomorphism $c:\Gamma\lar \Gamma$ such that
\begin{equation}\label{E:CoxeterDefinition}
\langle \alpha, \beta\rangle + \langle \beta, c(\alpha) \rangle = 0 \; \mbox{\rm for all}\; 
\alpha, \beta\in \Gamma. 
\end{equation}
\end{definition}

\noindent
Note that such $c \in \Aut(\Gamma)$ is unique provided it exists; see \cite[Proposition 2.1]{LenzingKTheory}.

\begin{definition} An element $\alpha \in \Gamma$ is called a \emph{pseudo-root} if the following conditions are satisfied: 
\begin{enumerate}
\item[(i)] $\langle \alpha, \alpha\rangle > 0$ and
\item[(ii)] $\dfrac{\langle \alpha, \gamma\rangle}{\langle \alpha, \alpha\rangle}, \dfrac{\langle \gamma, \alpha\rangle}{\langle \alpha, \alpha\rangle} \in \ZZ$ for all $\gamma \in \Gamma$. 
\end{enumerate}
In what follows, $\Pi$ denotes the set of all pseudo-roots of $(\Gamma, K)$. Obviously, if $\alpha \in \Pi$ then $-\alpha \in \Pi$, too.
\end{definition}

\begin{definition}\label{D:CompleteExcSeq}
For $r \in \NN$ a tuple $E = (\gamma_1, \dots, \gamma_r) \in \Pi^r$ is an \emph{exceptional sequence} if $\langle \gamma_i, \gamma_j\rangle = 0$ for all $1 \le j < i \le r$. Such a sequence $E$ is called \emph{complete} if the subgroup generated by $E$ is $\Gamma$, i.e. $\llangle \gamma_1, \dots, \gamma_r\rrangle_{\ZZ} = \Gamma$ (of course, in this case we have: $r = \mathrm{rk}(\Gamma)$). 
\end{definition}

\smallskip
\noindent
We denote by  $B = ( -,\,-): \; \Gamma \times \Gamma \lar \ZZ$ the symmetrization of the form $K$, i.e.
$
(\alpha, \beta) := \langle \alpha, \beta\rangle + \langle \beta, \alpha \rangle
$ for any $\alpha, \beta \in \Gamma$. Next, we denote by
\begin{equation*}
\mathsf{O}(\Gamma, B) := \bigl\{f \in \Aut(\Gamma) \, \big| \, \bigl(f(\alpha), 
f(\beta)\bigr) = (\alpha, \beta)\; \mbox{for all}\, \alpha, \beta \in \Gamma\}
\end{equation*}
the group of isometries of $(\Gamma, B)$. 

\begin{definition} For $\alpha \in \Pi$ consider the following group homomorphism
\begin{equation}\label{E:Reflection}
s_\alpha:\Gamma\lar \Gamma, \; \gamma \mapsto \gamma - 2 \dfrac{(\gamma, \alpha)}{(\alpha, \alpha)} \alpha
\end{equation}
called a \emph{reflection}. Note that the assumption that $\alpha$ is a pseudo-root implies that $2 \frac{(\gamma, \alpha)}{(\alpha, \alpha)} = \frac{\langle \gamma, \alpha\rangle + \langle \alpha, \gamma\rangle}{\langle \alpha, \alpha\rangle} \in \ZZ$ for all $\gamma \in \Gamma$. 
\end{definition}

\smallskip
\noindent
The following results can be verified by a straightforward computation.
\begin{lemma} For any $\alpha \in \Pi$ we have $s_\alpha^2 = \mathbbm{1}$ and $s_\alpha \in \mathsf{O}(\Gamma, B)$. Moreover, for any other $\beta \in \Pi$ we have
\begin{equation}\label{E:Action1}
s_{\beta}(\alpha) \in \Pi \quad{\rm and} \quad s_\beta s_\alpha s_\beta^{-1} = s_{s_{\beta}(\alpha)}.
\end{equation}
\end{lemma}

\smallskip
\noindent
For a proof of the following results, we refer to \cite[Proposition 2.4 and Lemma 2.7]{HuberyKrause}.

\begin{proposition}\label{P:HKLemmas} Assume that a bilinear lattice $(\Gamma, K)$ admits a complete exceptional sequence $E = (\gamma_1, \dots, \gamma_n)$. Then the following statements are true.
\begin{enumerate}
\item[(a)] The composition $c:= s_{\gamma_1} \dots s_{\gamma_n} \in \mathsf{O}(\Gamma, B)$ satisfies (\ref{E:CoxeterDefinition}), i.e.~$c$ is the Coxeter element of $(\Gamma, K)$. 
\item[(b)] The set of pseudo-roots $\Pi$ is \emph{reduced}, i.e.~for any $\alpha, \beta \in \Pi$ such that $s_\alpha = s_\beta$ we have $\alpha = \pm \beta$. 
\end{enumerate}
\end{proposition}

\smallskip
\noindent
It is well-known that we have an action of the braid group on the set of exceptional sequences in $(\Gamma, K)$; see for instance \cite[Proposition 2.6]{HuberyKrause}.

\begin{proposition}\label{P:braidGroupActionRoots}
The braid group $B_r$ with the standard generators $\sigma_1, \dots, \sigma_{r-1}$ and relations $\sigma_i \sigma_j = \sigma_j \sigma_i$ for $1 \le i, j \le r-1$ such that $\big|i-j\big| \ge 2$ and $\sigma_i \sigma_{i+1} \sigma_i = \sigma_{i+1} \sigma_i \sigma_{i+1}$ for $1 \le i \le r-2$ acts on the set of exceptional sequences of length $r$ in $(\Gamma, K)$ by the rules
\begin{equation*}
\left\{
\begin{array}{l}
\sigma_i(\gamma_1, \dots, \gamma_{r}) = (\gamma_1, \dots, \gamma_{i-1}, 
\gamma_{i+1}, s_{\gamma_{i+1}}(\gamma_i), \gamma_{i+2}, \dots, \gamma_r)\\
\sigma_i^{-1}(\gamma_1, \dots, \gamma_{r}) = (\gamma_1, \dots, \gamma_{i-1}, 
s_{\gamma_{i}}(\gamma_{i+1}),
\gamma_{i},  \gamma_{i+2}, \dots, \gamma_r)
\end{array}
\right.
\end{equation*}
\end{proposition}

\smallskip
\noindent
On the group-theoretic level, this braid group action admits the following description.

\begin{proposition} Let $G$ be a group and $T \subseteq G$ a subset closed
under conjugation. Then for any $r \in \NN$ the braid group $B_r$ acts on the set $T^r$ by the so-called \emph{Hurwitz action}: 
\begin{equation*}
\left\{
\begin{array}{l}
\sigma_i(g_1, \dots, g_r) = (g_1, \dots, g_{i-1}, 
g_{i+1}, g_{i+1}^{-1} g_i g_{i+1} , g_{i+2}, \dots, g_r)\\
\sigma_i^{-1}(g_1, \dots, g_{r}) = (g_1, \dots, g_{i-1}, 
g_i g_{i+1} g_i^{-1},
g_{i},  g_{i+2}, \dots, g_r)
\end{array}
\right.
\end{equation*}
for any $1 \le i \le r-1$. 
\end{proposition}

\begin{definition}\label{D:Coxeter datum} Let $(\Gamma, K)$ be a bilinear lattice and $R = (\alpha_1, \dots, \alpha_n)$ be a complete exceptional sequence. Then we get the following notions (relative to the choice of $R$).
\begin{enumerate}
\item[(a)] $W := \llangle s_{\alpha_1}, \dots, s_{\alpha_n}\rrangle \subseteq \mathsf{O}(\Gamma, B)$ is the corresponding  \emph{reflection group}.
\item[(b)] $\Phi: = \bigl\{w(\alpha_i) \, \big| \, 1 \le i \le n, w \in W \bigr\}$ is the set of \emph{real roots} of $(\Gamma, K)$. 
\item[(c)] $T:= \bigl\{s_\gamma \, \big|\, \gamma \in \Phi\bigr\}$ is the \emph{set of reflections} of $W$, whereas $S := \bigl\{s_{\alpha_1}, \dots, s_{\alpha_n}\bigr\}$ is the set of \emph{simple reflections}.
\item[(d)] In what follows, we shall call $(W,S)$ a \emph{generalized Coxeter datum} and $(W,T,c)$ a \emph{generalized dual Coxeter datum}, where $c$ is the Coxeter element defined via (\ref{E:CoxeterDefinition}).
\end{enumerate}
We have the analogous notions if $R=(\alpha_1, \dots, \alpha_n)$ are non-isotropic vectors in a vector space with symmetric bilinear form $(V,B)$. In this case we define the Coxeter element as a product of simple reflections $c:= s_{\alpha_1} \dots s_{\alpha_n}$. The elements of $R$ are sometimes called \emph{simple roots}.
\end{definition}

\begin{remark}\label{R:CoxData1} Note the following facts.
\begin{enumerate}
\item[(a)] Because of (\ref{E:Action1}), the  set of reflections $T$ admits the following description: $T = \bigl\{wsw^{-1} \, \big| \, w \in W, s \in S \bigr\}$. 
\item[(b)] For reflection groups defined via complete exceptional sequences in a bilinear lattice $(\Gamma,K)$, we have $\Phi \subseteq \Pi$ (i.e.~any real root is a pseudo-root). In particular, whenever $\alpha, \beta \in \Phi$ are such that $s_\alpha = s_\beta$ then we have: $\alpha = \pm \beta$; see Proposition \ref{P:HKLemmas}.
\item[(c)] For any $t \in T$ we have $t^2 = \mathbbm{1}$. On the other hand, there might be elements $g\in W$ such that $g^2 = \mathbbm{1}$, which are not contained in $T$.
\item[(d)] If $R = (\alpha_1, \dots, \alpha_n)$ and $R' = (\alpha'_1, \dots, \alpha'_n)$ are two complete exceptional sequences in $(\Gamma, K)$ which belong to the same orbit of the braid group $B_n$, then for the corresponding sets from Definition \ref{D:Coxeter datum} we have $(W, \Phi, T) = (W', \Phi', T')$.
\item[(e)] Let $R = (\alpha_1, \dots, \alpha_n)$ be a complete exceptional sequence in $(\Gamma, K)$, $V$ the real hull of $\Gamma$ and $B$ the symmetrization of $K$, viewed as a bilinear form on $V$. Then $R$ might be viewed as a tuple in $V$. Both interpretations of $R$ in Definition \ref{D:Coxeter datum} lead to the same reflection group, set of real roots, simple reflections and reflections. Importantly, the Coxeter element defined via Definition \ref{D:Coxeter element} and the Coxeter element defined via Definition \ref{D:Coxeter datum} coincide by Proposition \ref{P:HKLemmas}.
\end{enumerate}
\end{remark}

\begin{definition}
Let $(W,T,c)$ be a generalized dual Coxeter datum as in Definition \ref{D:Coxeter datum} and $w \in W$. Then the \emph{reflection length} $\ell_T(w)$ of $w$ is the minimal number $r \in \NN$, for which there exist $t_1, \dots, t_r \in T$ such that $w = t_1 \dots t_r$. The \emph{absolute order} $\le_T$ on $W$ is defined by the rule
\begin{equation*}
u \le_T v \quad \mbox{\rm provided} \quad \ell_T(u) + \ell_T(u^{-1} v) = \ell_T(v).
\end{equation*}
\end{definition}

\begin{definition} Let $(\Gamma, K)$ be a bilinear lattice which admits a complete exceptional sequence $R$ and let $(W,T,c)$ be the associated generalized dual Coxeter datum. Then
\begin{equation*}
\mathsf{NC}_T(W, c) := \bigl\{w \in W \, \, \big| \, \mathbbm{1} \le_T w \le_T c  \bigr\}
\end{equation*}
is the associated poset of \emph{non-crossing partitions}, which is one of the main objects of study of this work. 
\end{definition}

\begin{remark} Extending Remark \ref{R:CoxData1} note that for any two complete exceptional sequences $R$ and $R'$ in $(\Gamma, K)$ from the same braid group orbit, we get the same poset of non-crossing partitions $\mathsf{NC}_T(W, c)$.
\end{remark}

\subsection{Estimates of the reflection length}
Let $V$ be a finite--dimensional real vector space and $B = ( -,\,-): \; V \times V \lar \RR$ a symmetric bilinear form (of any signature). We denote by
$$\mathsf{O}(V, B) := \bigl\{f \in \Aut(V) \, \big| \, \bigl(f(x), f(y)\bigr) = (x, y)\; \mbox{for all}\, x, y \in V\}$$
the corresponding group of isometries. For any $f \in \End(V)$ we define its fixed space by $\mathsf{Fix}(f) = \bigl\{v \in V \, \big| \, f(v) = v \bigr\}$. A vector $v \in V$ is called \emph{non-isotropic} if $(v, v) \ne 0$. Analogous to (\ref{E:Reflection}), for  any such $v$, we have the associated reflection
\begin{equation*}
s_v:V\lar V, \; x \mapsto x - 2 \dfrac{(x, v)}{(v, v)} v.
\end{equation*}
\begin{remark}
For a non-isotropic vector $v \in V$, we shall denote $v^\sharp = \dfrac{2}{(v, v)} v$, which is sometimes called ``dual vector'' (see \cite{SaitoI}) since $\bigl(v^\sharp\bigr)^\sharp = v$. In these terms we have $$s_v(x) = x - (x,v^\sharp) v.$$
\end{remark}

\begin{lemma}\label{L:ResultsRealReflect} For any non-isotropic vector $v \in V$ the following results are true.
\begin{enumerate}
\item[(a)] We have $s_v^2 = \mathbbm{1}$ and $s_v \in \mathsf{O}(V, B)$.
\item[(b)] If $v' \in V$ is another non-isotropic vector then $s_v = s_{v'}$ if and only if $v = \lambda v'$ for some $\lambda \in \RR^\ast$. 
\item[(c)] For any $f \in \mathsf{O}(V, B)$ we have $f s_v f^{-1} = s_{f(v)}$. 
\item[(d)] $\mathsf{Fix}(s_v) = \llangle v\rrangle_{\RR}^\perp$ and 
$\mathsf{det}(s_v) = -1$. 
\end{enumerate}
\end{lemma}

\begin{lemma}\label{L:EstimateCodim}
For any $f_1, f_2 \in \End(V)$ we have
\begin{equation*}
\mathsf{cod}\bigl(\mathsf{Fix}(f_1 f_2)\bigr) \le \mathsf{cod}\bigl(\mathsf{Fix}(f_1)\bigr) + \mathsf{cod}\bigl(\mathsf{Fix}(f_2)\bigr),
\end{equation*}
where $\mathsf{cod}(U)$ denotes the codimension of a vector subspace $U \subseteq V$. 
\end{lemma}

\begin{proof} Let $U_i = \mathsf{Fix}(f_i)$ for $i = 1, 2$. It is clear that $U_1 \cap U_2 \subseteq \mathsf{Fix}(f_1 f_2)$. As a consequence, 
$\mathsf{cod}\bigl(\mathsf{Fix}(f_1 f_2)\bigr)  \le \mathsf{cod}\bigl(U_1 \cap U_2)$. Since we have an  embedding $V/(U_1 \cap U_2) \lar V/U_1 \oplus V/U_2$, comparing the dimensions of both sides, we get the stated inequality. 
\end{proof}

\begin{lemma}
Let $v_1, \dots, v_m \in V$ be a family of linearly independent non-isotropic vectors. Then for any $x \in V$ we have
$$
s_{v_1} \dots s_{v_m}(x) = x \, \Longleftrightarrow s_{v_i}(x) = x \; \mbox{\rm for all} \; 1 \le i \le m.
$$
\end{lemma}
\begin{proof} The converse implication is obvious. We prove the direct implication by induction on $m$. The case $m = 1$ is again obvious. To prove the induction step, note that
$$
s_{v_1} \dots s_{v_m}(x) = s_{v_1} \dots s_{v_{m-1}}\left(x - 2 \dfrac{(x, v_m)}{(v_m, v_m)} v_m \right) = x - 2 \dfrac{(x, v_m)}{(v_m, v_m)} v_m + u
$$
for some $u \in \llangle v_1, \dots, v_{m-1}\rrangle_{\RR}$. If $s_{v_1} \dots s_{v_m}(x) = x$ then $(x, v_m) v_m \in \llangle v_1, \dots, v_{m-1}\rrangle_{\RR}$. Since $v_1, \dots, v_m$ are linearly independent, it follows that $(x, v_m) = 0$, hence $s_{v_m}(x) = x$. Applying the assumption of induction, we conclude that $s_{v_i}(x) = x$ for all $1 \le i \le m$.
\end{proof}

\begin{corollary}\label{C:FixPoints}
Let $(v_1, \dots, v_n)$ be a basis of $V$ consisting of non-isotropic vectors and $c:= s_{v_1} \dots s_{v_n}$. Then we have
$$
\mathsf{Fix}(c) = \mathsf{Rad}(B) \;\;  \mbox{\rm and}\;\;  \mathsf{cod}\bigl(\mathsf{Fix}(c)\bigr) = n - 
\mathsf{dim}_{\RR}\bigl(\mathsf{Rad}(B)\bigr).
$$
\end{corollary}

\smallskip
\noindent
Now, let $(\Gamma, K)$ be a bilinear lattice with a complete exceptional sequence $R = (\alpha_1, \dots, \alpha_n)$, $B$ be the symmetrization of $K$  and  $(W,T,c)$ be the associated generalized dual Coxeter datum. We denote by $V = \RR \otimes_{\ZZ} \Gamma$ the real hull of $\Gamma$. Abusing the notation, we denote the extension of $B$ on $V$ by the same letter. Moreover, we have natural inclusions $W \subseteq \mathsf{O}(\Gamma, B) \subset \mathsf{O}(V, B)$.

\begin{lemma}\label{L:estimatelength}
For any $w \in W$ we have $\ell_T(w) \ge \mathsf{cod}\bigl(\mathsf{Fix}(w)\bigr)$. 
\end{lemma}

\begin{proof}
Let $r = \ell_T(w)$ and $t_1, \dots, t_r \in T$ be such that $w = t_1 \dots t_r$. Then we have
$$\mathsf{cod}\bigl(\mathsf{Fix}(w)\bigr) = \mathsf{cod}\bigl(\mathsf{Fix}(t_1 \dots t_r)\bigr) \le \sum\limits_{i= 1}^r \mathsf{cod}\bigl(\mathsf{Fix}(t_i)\bigr) = r,$$
where we used Lemma \ref{L:EstimateCodim} and the last statement of Lemma \ref{L:ResultsRealReflect}. 
\end{proof}

\begin{proposition}\label{P:EstimatesLengthCoxeter} The following results are true.
\begin{enumerate}
\item[(a)] $\ell_T(c) \ge \mathsf{cod}\bigl(\mathsf{Fix}(c)\bigr) = n -\mathsf{rk}(I)$, where $I = \mathsf{Rad}(B)$. 
\item[(b)] $\ell_T(c) \equiv n \; \mathsf{mod} \; 2$. 
\end{enumerate}
\end{proposition}

\begin{proof} Recall that by Proposition \ref{P:HKLemmas}, the Coxeter element is given by $c = s_{\alpha_1} \dots s_{\alpha_n}$. The first result is a consequence of Lemma \ref{L:estimatelength} combined with Corollary \ref{C:FixPoints}. To prove the second statement, note that by Lemma \ref{L:ResultsRealReflect} (d) we have: $\mathsf{det}(c) = (-1)^n = (-1)^r$, where $r = \ell_T(c)$. 
\end{proof}

\begin{corollary}\label{C:reflectionLength}
Let $(\Gamma, K)$ be a bilinear lattice with a complete exceptional sequence $R$ and $I$ be the radical of $B$. Assume that $\mathsf{rk}(I) = 0$ or $1$. Then we have $\ell_T(c) = \mathsf{rk}(\Gamma)$. 
\end{corollary}

\begin{remark}
The techniques to establish estimates for the length of an element of a reflection group are well-known and date back to a work of Scherk \cite[Theorem 1]{Scherk};  see also \cite[Theorem 260.1]{SnapperTroyer}. In our setting, we do not put any assumptions on the signature of the bilinear form $B$. Other works on reflection length include, for instance, \cite{BradyMcCammond} and \cite{McCammondPaolini}.
\end{remark}

\section{Exceptional sequences in hereditary and derived categories}\label{S:ExceptCollections}
Exceptional sequences were originally introduced by the Moscow school of vector bundles; see \cite{Helices}.  After the axiomatic treatment in the setting of bilinear lattices in the previous section, we now recall the original context of hereditary abelian and derived categories.

Let $\kk$ be any field and $\cD$ be a $\kk$-linear triangulated category such that for any  $E, F \in \mathsf{Ob}(\cD)$ we have $\dim_{\kk}\bigl(\Hom_{\cD}^\ast(E, F)\bigr) < \infty$, where 
$\Hom^\ast_{\cD}(E, F) = \oplus_{p \in \ZZ} \Hom_{\cD}(E, F[p])$.

\begin{definition} An object $E \in \mathsf{Ob}(\cD)$ is called \emph{exceptional} if the following conditions are satisfied:
\begin{enumerate}[(i)]
\item $\End_{\cD}(E)$ is a skew field (hence, a finite--dimensional  division algebra over $\kk$).
\item $\Hom_{\cD}(E, E[p]) = 0$ for any $p \ne 0$. 
\end{enumerate}
\end{definition}

\begin{definition}
A family of objects $(E_1, \dots E_r)$ of $\cD$ is called an \emph{exceptional sequence} if the following conditions are satisfied:
\begin{enumerate}[(i)]
\item For any $1\le i \le r$ the object $E_i$ is exceptional.
\item $\Hom_{\cD}^\ast(E_i, E_j) = 0$ for any $1 \le j < i \le r$. 
\end{enumerate}
If $r = 2$ then $(E_1, E_2)$ is called an \emph{exceptional pair}. 
An exceptional sequence $(E_1, \dots E_r)$ is \emph{full} if $\cD$ is the smallest triangulated subcategory of $\cD$ containing all elements $E_1, \dots E_r$ of the sequence.
\end{definition}

\smallskip
\noindent
If $\cD$ admits a full exceptional sequence $(E_1, \dots, E_n)$ then the Grothendieck group 
$\Gamma := K_0(\cD)$ is free of rank $n$  and the classes $[E_1], \dots, [E_n]$ form a basis of $\Gamma$. Let $K:\Gamma\times\Gamma\lar\ZZ$ be the Euler form, i.e. for $E,F\in\cD$ we have
$$K([E],[F])=\sum_{p\in\ZZ}(-1)^p\dim_{\kk}\Hom_{\cD}(E, F[p]).$$
Then $(\Gamma, K)$ is a bilinear lattice in the sense of Definition \ref{D:BilLattice}.

\begin{lemma}\label{L:ClassIsPseudoRoot}
Let $E$ be an exceptional object in $\cD$. Then its class $[E] \in \Gamma$ is a pseudo-root. 
\end{lemma}

\begin{proof} The endomorphism algebra $\Lambda = \End_{\cD}(E)$ is a finite--dimensional division algebra over $\kk$. As a consequence, for any object $F \in \mathsf{Ob}(\cD)$ the morphism spaces $\Hom_{\cD}(E, F)$ (respectively, 
$\Hom_{\cD}(F, E)$) are free right (respectively, left) $\Lambda$-modules of finite rank. Hence, $\dfrac{\langle [E], [F]\rangle}{\langle [E], [E]\rangle} $ and $ \dfrac{\langle [F], [E]\rangle}{\langle [E], [E]\rangle}$ are integers for any $F \in \mathsf{Ob}(\cD)$.
\end{proof}

\smallskip
\noindent
Next, we want to define mutations of exceptional sequences. Let $E$ be an exceptional object in $\cD$, $\Lambda = \End_{\cD}(E)$ and $\llangle E\rrangle$ be the triangulated subcategory of $\cD$ generated by $E$. Then any object of $\llangle E\rrangle$ has the form $\bigoplus\limits_{i \in \ZZ} E^{\oplus m_i}[-i]$, where all but finitely many multiplicities $m_i$ are zero. Moreover, the category $\llangle E\rrangle$ is equivalent to the category $\mathsf{mod}^{\ZZ}\mbox{--}\Lambda$ of finite--dimensional graded right $\Lambda$-modules. According to \cite[Theorem 3.2]{Bondal}, $\llangle E\rrangle$ is an \emph{admissible subcategory} of $\cD$, which means that the embedding functor $I:\llangle E\rrangle\lar\cD$ has right and left adjoint functors. For the following concrete descriptions of the corresponding adjunction counit and adjunction unit respectively; see \cite{Bondal}.

\smallskip
\noindent
Let $E$ be an exceptional object in $\cD$ and $\Lambda = \End_{\cD}(E)$. Let $\Phi$ and $\Psi$ be the right and left adjoint functors respectively to the embedding functor $I:\llangle E\rrangle\lar\cD$. The corresponding adjunction counit $\xi:I\Phi\lar\mathsf{Id}_{\cD}$ and adjunction unit $\eta:\mathsf{Id}_{\cD}\lar I\Psi$ admit the following concrete descriptions. Let $F$ be any object of $\cD$.

For any $i \in \ZZ$, the morphism space $\Hom_{\cD}\bigl(E[-i], F\bigr)$ has a natural structure of  a right $\Lambda$--module. We put  $m_i = \dim_{\Lambda} 
\Hom_{\cD}\bigl(E[-i], F\bigr)$ and choose a basis 
$\bigl(\phi^{(i)}_1, \dots, \phi^{(i)}_{m_i}\bigr)$ of $\Hom_{\cD}\bigl(E[-i], F\bigr)$ over $\Lambda$. The adjunction counit morphism $I\Phi(F)\stackrel{\xi_F}\lar F$ (with respect to the adjoint pair $(I, \Phi)$) can be identified with the ``evaluation map''
\begin{equation*}
\bigoplus\limits_{i \in \ZZ} E^{\oplus m_i}[-i] \xrightarrow{\oplus_{i \in \ZZ}\bigl(\phi^{(i)}_1, \dots, \phi^{(i)}_{m_i}\bigr)} F.
\end{equation*}

Dually, for any $i \in \ZZ$, the morphism space 
$\Hom_{\cD}\bigl(F, E[i]\bigr) \cong \Hom_{\cD}\bigl(F[-i], E\bigr)$ is a left  $\Lambda$--module. We put  $n_i = \dim_{\Lambda} 
\Hom_{\cD}\bigl(F, E[i]\bigr)$ and choose a basis 
$\bigl(\psi^{(i)}_1, \dots, \psi^{(i)}_{n_i}\bigr)$ of $\Hom_{\cD}\bigl(F, E[i]\bigr)$ over $\Lambda$. The adjunction unit morphism $F \stackrel{\eta_F}\lar ~I\Psi(F)$ (with respect to the adjoint pair $(\Psi,I)$) can be identified with  the  ``coevaluation  map''
\begin{equation*}
F 
\xrightarrow{\oplus_{i \in \ZZ}\left(\begin{array}{c} \psi^{(i)}_1 \\ \vdots \\ \psi^{(i)}_{n_i}\end{array} \right)}
\bigoplus\limits_{i \in \ZZ} E^{\oplus n_i}[i].
\end{equation*}

\begin{definition}
For any exceptional object $E$ and any $F\in\mathsf{Ob}(\cD)$ we define the \emph{left mutation} $L_E(F)$ and the \emph{right mutation} $R_E(F)$ via the following distinguished triangles.
\begin{equation*}
I\Phi(F)\stackrel{\xi_F}\lar F\lar L_E(F)\lar I\Phi(F)[1]
\quad\mbox{\rm and}\quad
F\stackrel{\eta_F}\lar I\Psi(F)\lar R_E(F)[1]\lar F[1],
\end{equation*}
where $I:\llangle E\rrangle\lar\cD$ is the embedding functor and $\xi$ and $\eta$ are the counit and unit with respect to the adjoint pairs $(I,\Phi)$ and $(\Psi,I)$ respectively.
\end{definition}

\smallskip
\noindent
It is clear that we have the following equalities in the Grothendieck group $K_0(\cD)$:
\begin{equation*}
\bigl[L_E(F)\bigr] = 
[F] - \frac{\langle E, F\rangle}{\langle E, E\rangle}[E] \quad \mbox{\rm and} \quad 
 \bigl[R_E(F)\bigr] = 
[F] - \frac{\langle F, E\rangle}{\langle E, E\rangle}[E].
\end{equation*}

\smallskip
\noindent
For the following statements, we refer to \cite[Assertion 2.1]{Bondal} as well as \cite[Lemma 3.2]{KussinMeltzer}.
\begin{lemma}\label{L:Mutate} 
The following results are true.
\begin{enumerate}
\item[(a)] Let $(E, F)$ be an exceptional pair in $\cD$. Then the pair   $\bigl(L_E(F), E\bigr)$ is also exceptional and $\End_{\cD}\bigl(L_E(F)\bigr) \cong \End_{\cD}(F)$. Moreover,
$$\bigl[L_E(F)\bigr] = [F] - 2 \frac{(E, F)}{(E, E)}[E].$$
\item[(b)] Let $(G, E)$ be an exceptional pair in $\cD$. Then  the pair $\bigl(E, R_E(G)\bigr)$ is also exceptional and $\End_{\cD}\bigl(R_E(G)\bigr) \cong \End_{\cD}(G)$. Moreover,
$$\bigl[R_E(G)\bigr] =  [G] - 2 \frac{(E, G)}{(E, E)}[E].$$
\end{enumerate}
\end{lemma}

\smallskip
\noindent
For a proof of the next key result, we refer to \cite[Assertion 2.3]{Bondal}.
\begin{theorem}
The braid group $B_r$ acts on the set of exceptional sequences of length $r$ by the followings rules:
\begin{equation}\label{E:BraidGroup}
\begin{array}{l}
\sigma_i(E_1, \dots, E_r) = \bigl(E_1, \dots, E_{i-1}, E_{i+1}, R_{E_{i+1}}(E_i), E_{i+2}, \dots, E_r\bigr) \\
\sigma_i^{-1}(E_1, \dots, E_r) = \bigl(E_1, \dots, E_{i-1}, L_{E_{i}}(E_{i+1}), E_i,  E_{i+2}, \dots, E_r\bigr).
\end{array}
\end{equation}
\end{theorem}

For this reason exceptional sequences of a fixed length are important. In particular, the exceptional sequences of maximal length are of importance for us.
\begin{definition}
    Assume that $\cD$ admits a full exceptional sequence, so $K_0(\cD) \cong \ZZ^n$ for some $n \in \NN$. An exceptional sequence in $\cD$ of length $n$ is called \emph{complete}.
\end{definition}

\begin{remark} It was recently shown  in \cite{Krah} that in general a complete exceptional sequence need not be full. In particular, the braid group action (\ref{E:BraidGroup}) on the set of complete exceptional sequences is in general not transitive. An example of such a case  was also given in another recent work \cite{ChangHaidenSchroll}. However, in the categories $\Coh(\XX)$ for exceptional hereditary curves $\XX$, i.e.~the categories we are interested in, the notions of fullness and completeness coincide; see \cite{KussinMeltzer}.
\end{remark}

\begin{definition}
A triangulated subcategory $\sE$ of the category $\sD$ is called
\begin{itemize}
\item[(a)] \emph{thick} if it is closed under direct summands,
\item[(b)] \emph{exceptional} if it is generated by  an exceptional sequence $(E_1, \dots, E_r)$ in $\sD$.
\end{itemize} 
\end{definition}

In abelian categories, we have the following analogue.
\begin{definition} Let $\sA$ be an abelian category and $\sS$ be a full additive subcategory. Then $\sS$ is called \emph{thick} if it is closed under direct summands and has \emph{two out of three property}: an exact sequence
$$0 \lar X' \lar X \lar X'' \lar 0$$
lies in $\sS$ if two out of $X, X', X''$ are in $\sS$.
\end{definition}

\smallskip
\noindent
Let us now shift our focus to hereditary abelian categories. Proofs of the following results can, for instance, be found in \cite[Remark 4.4.16 and Proposition 4.4.17]{KrauseBook}.

\begin{proposition}\label{P:BookHenning} Let $\sA$ be a \emph{hereditary} abelian category and $\sD = D^b(\sA)$ be its derived category. 
\begin{enumerate}
\item[(a)] Let $\sS$ be a thick subcategory of $\sA$. Then $\sS$ is also abelian and hereditary.
\item[(b)] The correspondence $\sS \mapsto D^b(\sS)$ establishes a bijection between thick subcategories of $\sA$ (in the exact sense) and $\sD$ (in the triangulated sense).
\end{enumerate}
\end{proposition}

\begin{definition}
Let $\sA$ be an $\Ext$-finite $\kk$-linear hereditary abelian category.
\begin{itemize}
    \item[(a)] Let $E_1, \dots, E_r\in\sA$. We call $(E_1, \dots, E_r)$ an \emph{exceptional sequence} in $\sA$ if $(E_1, \dots, E_r)$ is an exceptional sequence in $D^b(\sA)$.
    \item[(b)] Let $\sS$ be a thick subcategory of $\sA$. Then $\sS$ is called \emph{exceptional} if there exist an exceptional sequence $(E_1, \dots, E_r)$ in $\sA$ such that $\sS$ is the smallest thick subcategory of $\sA$ containing $E_1, \dots, E_r$.
\end{itemize}
\end{definition}

\begin{remark} Note that Proposition \ref{P:BookHenning} also implies that the assignment  $\sS \mapsto D^b(\sS)$ establishes a bijection between exceptional thick subcategories of $\sA$ (in the exact sense) and $\sD$ (in the triangulated sense).
\end{remark}

\smallskip
\noindent 
Proofs of the following results can be found in the survey article 
\cite[Section 5]{LenzingSurvey}; see also  references therein for the original works. 

\begin{theorem}\label{T:HereditaryKey} Let $\sA$ be an $\Ext$-finite $\kk$-linear hereditary abelian category.
\begin{enumerate}
\item[(a)] Let $E$ be an indecomposable object in $\sA$ such that $\Ext^1_{\sA}(E, E) = 0$, i.e.~$E$ is rigid. Then $E$ is exceptional, i.e.~$ \End_{\sA}(E)$ is a skew field.
\item[(b)] Each exceptional object $E$ in $\sA$ is determined by its class $[E]$ in $K_0(\sA)$. 
\item[(c)] Let $E$, $F$ be two exceptional objects in $\sA$ and assume 
$\Ext^1_{\sA}(F, E) = 0$. Then at most one of the vector spaces $\Hom_{\sA}(E, F)$ and $\Ext^1_{\sA}(E, F)$ is non-zero. 
\end{enumerate}
\end{theorem}

\begin{remark} Let $\sA$ be as in Theorem \ref{T:HereditaryKey} above and $\sD = D^b(\sA)$. It is well-known that for any $X \in \mathsf{Ob}(\sD)$ we have a (non-canonical) isomorphism  $X \cong \oplus_{i \in \ZZ} H^i(X)[-i]$; see for instance \cite[Proposition 4.4.15]{KrauseBook}. In particular, any indecomposable object in $\sD$ is of the form $A[i]$, where $A$ is an indecomposable object in $\sA$ and $i \in \ZZ$. This allows us to talk about the braid group action on exceptional sequences in $\sA$.

Let $(E, F)$ be an exceptional pair in $\sA$ and $\Lambda = \End_{\sA}(E)$.  Let  
    $\overline{L}_E(F)$ be the only non-vanishing cohomology of the complex $L_E(F)$ in $\sD$. Then $\overline{L}_E(F)$ is defined by one of the  following short exact sequences:
\[
0 \lar \overline{L}_EF \lar  \Hom_\sA(E,F)\otimes_{\Lambda}E\lar  F
\lar 0
\]
\[
0\lar \Hom_\sA(E,F)\otimes_{\Lambda}E\lar F\lar \overline{L}_EF\lar 0
\]
\[
0\lar F\lar \overline{L}_EF\lar \Ext_{\sA}^1(E,F)\otimes_{\Lambda}E\lar 0
\]
The dual statement holds for $R_F(E)$. In what follows, we shall identify 
$L_E(F)$ with  $\overline{L}_E(F)$ (respectively, $R_E(F)$ with  $\overline{R}_E(F)$)  and speak about the braid group action  on the set of exceptional sequences of a given length $r$ in $\sA$. 
\end{remark}

\section{Exceptional hereditary curves}\label{S:ExceptionalHereditary}
In this section, we investigate the hereditary abelian category that is the main object of interest of this paper: Coherent sheaves on exceptional hereditary curves. We begin by recalling generalities on non-commutative curves. We then focus on homogeneous exceptional curves and their tilting with tame bimodules. Finally, we collect some combinatorial data and establish crucial results on exceptional sequences in $\Coh(\XX)$.
\subsection{Generalities on non-commutative curves}
In what follows we refer to \cite{Reiner} for the notion of an \emph{order} in a central simple algebra and to \cite{ArtindeJong, BurbanDrozdGavran, BurbanDrozd} for basic results on  non-commutative curves and the corresponding categories of (quasi-)coherent sheaves.

Let $\kk$ be any field and let $X$ be a curve over $\kk$, i.e.~a reduced quasi-projective equidimensional scheme of finite type over $\kk$ of Krull dimension one. We denote by $X_\circ$ the set of closed points of $X$ and $\kO$ the structure sheaf of $X$.

\begin{definition} A \emph{non-commutative curve} over $\kk$ is a ringed space $\XX = (X, \kA)$, where $X$ is a curve as above and $\kA$ is a sheaf of $\kO_X$-orders (i.e.~$\kA(U)$ is an $\kO(U)$-order for any open affine subset $U \subseteq X$), which is coherent as a sheaf of $\kO_X$-modules. Such $\XX$ is called
\begin{enumerate}
\item[(a)] \emph{central} if the stalk $\kO_x$ is the center of $\kA_x$,
\item[(b)] \emph{homogeneous} if the order $\kA_x$ is maximal,
\item[(c)] \emph{hereditary} if the order $\kA_x$ is hereditary
\end{enumerate}
for each closed point $x \in X_\circ$. A non-commutative curve $\XX$ is called \emph{complete} if $X$ is integral and proper (hence projective) over $\kk$.
\end{definition}
From now on, completeness will always be assumed. For such $\XX$, the category $\Coh(\XX)$ is abelian, noetherian, and $\Ext$--finite. Moreover, $\Coh(\XX)$ is hereditary if $\XX$ is hereditary.

\begin{remark}
Without loss of generality, one may assume $\XX=(X,\kA)$ to be central; see \cite[Remark 2.14]{BurbanDrozd}. Moreover, if a central curve $\XX=(X,\kA)$ is hereditary then $X$ is automatically regular; see \cite[Theorem 2.6]{Harada}. 
\end{remark}

Let $\kK$ be the sheaf of rational functions on $X$ and $\KK := \kK(X)$ be its  field   of rational functions. Then $\FF_{\XX} := \Gamma(X, \kK \otimes_\kO \kA)$ is a central simple algebra over $\KK$, called (non-commutative) function field of $\XX$. We denote by $\eta = \eta_{\XX} := \bigl[\FF_{\XX}\bigr]$ the  corresponding class in the Brauer group $\mathsf{Br}(\KK)$ of the field $\KK$. 

For any $\kF \in \Coh(\XX)$ we have a left  $\FF_{\XX}$-module $\Gamma(X, \kK \otimes_\kO \kF)$.  Hence, we define the  \emph{rank} of $\kF$ by the formula
\begin{equation}\label{E:RankFunction}
\mathsf{rk}(\kF) := \mathsf{length}_{\FF_{\XX}}\bigl(\Gamma(X, \kK \otimes_\kO \kF)\bigr).
\end{equation}
Note that we get a group homomorphism $\mathsf{rk}:K_0(\XX)\lar\ZZ$.

Next, we denote by  $\Tor(\XX)$ the category of torsion coherent sheaves on $\XX$, which is the full subcategory of $\Coh(\XX)$ consisting of finite length objects. Alternatively,  $\Tor(\XX)$ can be defined as the category of coherent sheaves on $\XX$ of rank zero.  It  splits into a union of blocks.
\begin{equation*}
\Tor(\XX) = \bigvee\limits_{x \in X_\circ} \Tor_x(\XX).
\end{equation*}
For any $x \in X_\circ$ the category  $\Tor_x(\XX)$ is equivalent to the category of finite--length modules over  the order $\widehat{\kA}_x$ (which is the completion of $\kA_x$). 

\begin{remark} The natural inclusion functor $D^b\bigl(\Tor(\XX)\bigr) \lar D^b\bigl(\Coh(\XX)\bigr)$ is fully faithful. Next, consider the Serre quotient category $\Coh(\XX)/\Tor(\XX)$. Then the functor
$$\Gamma(X, \kK \otimes_\kR \,-\,): \Coh(\XX)/\Tor(\XX) \lar \FF_{\XX}\mathrm{-}\mathsf{mod}$$
is an equivalence of categories.
\end{remark}

Next, we  denote by $\VB(\XX)$ the full subcategory of the category $\Coh(\XX)$ consisting of locally projective objects, i.e.~those $\kB \in \Coh(\XX)$ for which each stalk $\kB_x$ is a projective module over $\kR_x$ for any $x \in X_\circ$. Objects of $\VB(\XX)$ are called \emph{vector bundles} and vector bundles of rank one are called \emph{line bundles}. For any $\kB \in \VB(\XX)$, $\kZ \in \Tor(\XX)$ and $i \ge 1$, we automatically have the following vanishing
\begin{equation*}
\Hom_{\XX}(\kZ, \kB) = 0 = \Ext^i_{\XX}(\kB, \kZ).
\end{equation*}

\begin{lemma}\label{L:SectionStructureSheaf}
Let $\XX = (X, \kA)$ be a complete non-commutative curve over $\kk$ such that $\FF = \FF_{\XX}$ is a skew field. Let $\kL := \kA$ viewed as an object of $\VB(\XX)$. Then
\begin{equation}\label{E:SectionsStructureSheaf}
\ff := \bigl(\End_{\XX}(\kL)\bigr)^\circ \cong \Gamma(X, \kA)
\end{equation}
is a finite-dimensional division algebra over $\kk$, which is a subalgebra of $\FF$. 
\end{lemma}

\begin{proof}
First note that we have an isomorphism of $\kO$--algebras $\kA^\circ \cong \mathit{End}_{\kA}(\kL)$. Hence
$$ \ff = 
\bigl(\End_{\kA}(\kL)\bigr)^\circ = 
\bigl(\End_{\XX}(\kL)\bigr)^\circ \cong \bigl(\Gamma(X, \kA^\circ)\bigr)^\circ \cong \Gamma(X, \kA).
$$
Next, the canonical morphism of $\kk$-algebras
$$
\ff^\circ = 
\End_{\kA}(\kL) \lar \End_{\kK \otimes_\kO \kA}\bigl(\kK \otimes_\kO \kL\bigr) = \FF^\circ
$$
is injective. It follows that $\ff$ has no non-trivial nilpotent elements, hence it is a finite-dimensional semi-simple $\kk$-algebra. Moreover, $\ff$ has no non-trivial idempotents. Hence, by the Wedderburn--Artin theorem, it is a division algebra, as claimed. 
\end{proof}

From now on, we assume $\XX = (X, \kA)$ to be hereditary. In this case, the following results are true; see for example \cite[Theorem A.4]{NaeghvdBergh} or \cite[Proposition 6.14]{YekutieliZhang}.
\begin{itemize}
\item[(a)] The curve $X$ is regular. As in the case of regular commutative curves, for any $\kF\in\Coh(\XX)$ there exist unique $\kB\in\VB(\XX)$ and $\kZ \in \Tor(\XX)$ such that $\kF\cong\kB\oplus\kZ$.
\item[(b)] Let $\mathit{\Omega} = \mathit{\Omega}_{\XX} := \mathit{Hom}_{X}\bigl(\kA, \mathit{\Omega}_{X}\bigr)$, where $\mathit{\Omega}_X$ is the dualizing sheaf of $X$. Then the \emph{Auslander--Reiten translate}
\begin{equation}\label{E:ARTranslate}
\tau := \mathit{\Omega} \otimes_{\kA} \, -\, : \; \Coh(\XX) \lar \Coh(\XX)
\end{equation}
is an auto-equivalence of $\Coh(\XX)$. It restricts to auto-equivalences of its full subcategories $\VB(\XX)$, $\Tor(\XX)$ as well as $\Tor_x(\XX)$ for any $x \in X_\circ$. Moreover, for any $\kF, \kG \in \Coh(\XX)$ there are bifunctorial isomorphisms
\begin{equation}\label{E:ARDuality}
\Hom_{\XX}(\kF, \kG) \cong \Ext^1_{\XX}\bigl(\kG, \tau(\kF)\bigr)^\ast.
\end{equation}
In other words, $\tau[1]$ is a Serre functor of the derived category $D^b\bigl(\Coh(\XX)\bigr)$.
\end{itemize}

\subsection{Homogeneous Curves}
Homogeneous curves provide an important class of hereditary curves. Actually, one has the following key result. Let $\EE = (X, \kR)$ and $\EE' = (X', \kR')$ be two homogeneous curves. Then the corresponding categories of coherent sheaves $\Coh(\EE)$ and  $\Coh(\EE')$ are equivalent if and only if there exists an isomorphism $f:X\lar X'$ such that $f^\ast(\eta') = \eta \in \mathsf{Br}(\KK)$, where $f^\ast : \kK(X') \longrightarrow \kK(X)$ is the field homomorphism induced by $f$, and $\eta$ and $\eta'$ are the Brauer classes of $\EE$ and $\EE'$, respectively; see \cite[Proposition 1.9.1]{ArtindeJong} or \cite[Corollary 7.9]{BurbanDrozd}. Therefore, a homogeneous curve $\EE$ can without loss of generality be assumed to be \emph{minimal}, meaning that $\FF_{\EE}$ is a skew field. 

\begin{lemma}\label{L:GrothGroup}
Let $\EE = (X, \kR)$ be a minimal homogeneous curve, $\kL := \kR$ (viewed as an object of $\Coh(\EE)$), $\Gamma := K_0(\EE)$ be the Grothendieck group of $\Coh(\EE)$, $\alpha := [\kL]$ be the class of $\kL$  and $\Gamma'$ be the subgroup of $\Gamma$ generated by the classes of torsion coherent sheaves. Then $\Gamma = \llangle \alpha, \Gamma'\rrangle$ and $\Gamma' = \mathsf{Ker}(\mathsf{rk}:\Gamma\lar\ZZ)$.
\end{lemma}

\begin{proof}
Let $\kX \in \Coh(\EE)$. We show by induction on $\mathsf{rk}(\kX)$ that $[\kX] \in \llangle \alpha, \Gamma'\rrangle \subseteq \Gamma$. Without loss of generality, one may assume $\kX$ to be locally free. Let $x \in X_\circ$ be any point and $\kS_{[m]}$ be an indecomposable torsion sheaf of length $m$ supported at $x$ (which is unique up to isomorphism). We have an epimorphism $\kL \twoheadarrow \kS_{[m]}$ which defines a short exact sequence
\begin{equation}\label{E:AmpleSequence}
0 \lar \kL_{[m]} \lar \kL \lar \kS_{[m]} \lar 0.
\end{equation}
Since $\kL_{[m]}$ is a line bundle on $\EE$, there exists a sheaf of ideals $\kI_{[m]} \subset \kO$ such that $\kL_{[m]} \cong \kL \otimes_\kO \kI_{[m]}$. Moreover, we have an isomorphism of vector spaces over $\kk$:
\begin{equation}\label{E:SomeIsomorphism}
\Hom_{\EE}(\kL_{[m]}, \kX) \cong \Gamma\bigl(X, \kX \otimes_\kO \kI_{[m]}^\vee\bigr).
\end{equation}
It follows that $\Hom_{\EE}(\kL_{[m]},\kX)\ne0$ for sufficiently large $m\in\NN$. Since any non-zero morphism $f:\kL_{[m]}\lar\kX$ is automatically a monomorphism, we get a short exact sequence
$$ 0 \lar \kL_{[m]} \stackrel{f}\lar \kX \lar \overline\kX \lar 0.$$
It is clear that $\mathsf{rk}(\overline\kX) = \mathsf{rk}(\kX)-1$, hence  $[\overline\kX] \in \llangle \alpha, \Gamma'\rrangle$ by the induction hypothesis.  Moreover, $[\kL_{[m]}] \in \llangle \alpha, \Gamma'\rrangle$ by the construction of $\kL_{[m]}$. Hence, $[\kX] \in \llangle \alpha, \Gamma'\rrangle$, as asserted. It follows that $\gamma \in \Gamma$ belongs to $\Gamma'$ if and only if $\mathsf{rk}(\gamma) = 0$.
\end{proof}

\begin{definition} A minimal complete homogeneous curve $\EE = (X, \kR)$ is called \emph{exceptional} if  $H^1(X, \kR) = 0$.
\end{definition}

\begin{lemma}\label{L:ExcCurvesBasics} Let $\EE = (X, \kR)$ be a minimal exceptional homogeneous curve and $g$ the genus of $X$. Then the following statements hold.
\begin{itemize}
\item[(a)] The line bundle $\kL$ is rigid, i.e.~$\Ext^1_{\EE}(\kL, \kL) = 0$. 
\item[(b)] We have $g = 0$.
\end{itemize}
\end{lemma}

\begin{proof} (a) The rigidity of $\kL$ follows from the isomorphisms
\begin{equation*}
\Ext^1_{\EE}(\kL, \kL) \cong H^1\bigl(X, \mathit{End}{\kR}(\kL)\bigr) = H^1(X, \kR^\circ) = H^1(X, \kR) = 0.
\end{equation*}

\smallskip
\noindent
b) The corresponding statement is due to Artin and de Jong \cite[Proposition 4.2.4]{ArtindeJong}, and we include a proof for the reader's convenience. Since $\EE$ is central and homogeneous, the commutative curve $X$ is regular. Hence, $\kR$ is locally free, viewed 
as an $\kO_X$--module. Let $\mathsf{tr}: \FF \times \FF \lar \KK$ be the reduced trace map; see \cite[Section 9a]{Reiner}.
 There is the corresponding  $\kO_X$-bilinear trace pairing $\mathsf{tr}: \kR \times \kR \lar \kO_X$ inducing a short exact sequence
$$
0 \lar \kR \lar \kR^\vee \lar \kT \lar 0
$$
in the category $\Coh(X)$. 
Note that $\mathsf{Supp}(\kT) = \bigl\{x \in X_\circ \, \big| \,  \kR_x \text{ is not Azumaya}\bigr\}$ is the discriminant locus of $\kR$ (see \cite[Section 10]{Reiner}), so $\mathsf{Supp}(\kT)$ is a finite set. Since $H^1(X, \kR) = 0 = H^1(X, \kT)$, we also obtain $H^1(X, \kR^\vee) = 0$. By the Riemann--Roch theorem we have
$$
\left\{
\begin{array}{l}
0 < \chi(\kR) = \mathsf{deg}(\kR) + (1-g) \mathsf{rk}(\kR) \\
0 < \chi(\kR) = \mathsf{deg}(\kR^\vee) + (1-g) \mathsf{rk}(\kR^\vee).
\end{array}
\right.
$$
Since $\mathsf{rk}(\kR^\vee) = \mathsf{rk}(\kR) > 0$ and $\mathsf{deg}(\kR^\vee) = -\mathsf{deg}(\kR)$, it follows that $(1-g)\mathsf{rk}(\kR) > 0$, hence $g = 0$, as claimed.
\end{proof}

For the rest of this subsection, let $\EE=(X,\kR)$ be a minimal exceptional homogeneous curve and $\ff =  \Gamma(X, \kR)$ be the corresponding $\kk$-algebra of global sections of $\kR$ (recall that $\ff \cong \bigl(\End_{\XX}(\kL)\bigr)^\circ$ is a division algebra; see Lemma \ref{L:SectionStructureSheaf}).

\begin{definition}
Let $\kS$ be any simple object in $\Tor(\EE)$. Then $\Ext^1_{\EE}(\kS, \kL)$ has a natural structure of a right $\ff$--module  and  we put $m := \mathsf{dim}_{\ff} \bigl(\Ext^1_{\EE}(\kS, \kL)\bigr)$. Since $\EE$ is homogeneous, we have $\tau(\kS) \cong \kS$. It follows from (\ref{E:ARDuality}) that
$$\Ext^1_{\EE}(\kS, \kL)^\ast \cong \Hom_{\EE}(\kL, \kS) \cong \Gamma(X, \kS) \ne 0,$$
implying that  $m \ge 1$. Let $(\omega_1, \dots, \omega_m)$ be a basis of $\Ext^1_{\EE}(\kS, \kL)$ over $\ff$. The \emph{companion bundle} $\overline{\kL}$ of $\kL$ (corresponding to  $\kS$) is defined by an exact triangle
\begin{equation*}
\kL^{\oplus m} \lar \overline{\kL} \lar \kS \xrightarrow{\left(\begin{smallmatrix} \omega_1 \\ \vdots \\ \omega_m\end {smallmatrix} \right)} \kL^{\oplus m}[1]
\end{equation*}
in $D^b\bigl(\Coh(\EE)\bigr)$ or, equivalently, by the corresponding  couniversal extension
\begin{equation}\label{E:companion}
0 \lar \kL^{\oplus m} \lar \overline{\kL} \stackrel{\pi}\lar \kS \lar 0
\end{equation}
in $\Coh(\EE)$.
\end{definition}

\begin{lemma}\label{L:rigidVectorBundle}
Let $\kS \in \Tor(\EE)$ be any simple object and $\overline{\kL}$ the companion bundle of $\kL$ corresponding to $\kS$. Then $\kL \oplus \overline\kL$ is a rigid vector bundle on $\EE$.
\end{lemma}
\begin{proof}
It is not hard to see that $\overline{\kL} \in \VB(\EE)$. Indeed, if this is not the case then $\overline{\kL} \cong \kB \oplus \kZ$, where $\kB \in \VB(\XX)$ and $0 \not\cong \kZ \in \Tor(\EE)$. We have an epimorphism $\pi = (\pi' \, \pi''): \kB \oplus \kZ \lar  \kS$. If $\pi''$ is an isomorphism then the short exact sequence (\ref{E:companion}) splits, which contradicts the construction of this sequence.  Otherwise, $\mathsf{Ker}(\pi) \cong \kL^{\oplus m}$ contains a subobject from $\Tor(\EE)$, which again yields a contradiction. Thus it remains to show that $\kL \oplus \overline\kL$ is rigid.

Applying $\Hom_{\EE}(\kL, \,-\,)$ to (\ref{E:companion}), we get an exact sequence
$$\Ext^1_{\EE}\bigl(\kL, \kL^{\oplus m}\bigr) \lar \Ext^1_{\EE}(\kL, \overline{\kL}) \lar \Ext^1_{\EE}(\kL, \kS).$$
Since $\Ext^1_{\EE}(\kL, \kL) = 0 = \Ext^1_{\EE}(\kL, \kS)$, we conclude that  $\Ext^1_{\EE}(\kL, \overline{\kL}) = 0$, too. Now we apply to (\ref{E:companion}) the functor $\Hom_{\EE}(\,-\,,\kL)$.  Since $\Hom_{\EE}(\kS, \kL) = 0 = \Ext^1_{\EE}(\kL^{\oplus m}, \kL)$, we get an exact sequence
$$0  \lar \Hom_{\EE}(\overline\kL, \kL) \lar \Hom_{\EE}(\kL^{\oplus m}, \kL) \stackrel{\delta}\lar \Ext^1_{\EE}(\kS, \kL) \lar \Ext^1_{\EE}(\overline\kL, \kL) \lar 0.$$
By construction, $\delta$ is an isomorphism of $\ff$--modules, hence $\Hom_{\EE}(\overline\kL, \kL) = 0 = \Ext^1_{\EE}(\overline\kL, \kL)$. Finally, applying to (\ref{E:companion}) the functor $\Hom_{\EE}(\overline\kL, \,-\,)$ and using the vanishing $\Ext^1_{\EE}(\overline\kL, \kL) = 0 = \Ext^1_{\EE}(\overline\kL, \kS)$, we conclude that $\Ext^1_{\EE}(\overline\kL, \overline\kL) = 0$, too.
\end{proof}

Again, let $\kS \in \Tor(\EE)$ be any simple object and $\overline{\kL} = \kG_1^{\oplus p_1} \oplus \dots \oplus \kG_t^{\oplus p_t}$ the decomposition of the corresponding companion bundle into pairwise non-isomorphic indecomposable vector bundles. Now, put $\kF = \kL \oplus \kG_1 \oplus \dots \oplus \kG_t$ and $\Lambda := \bigl(\End_{\EE}(\kF)\bigr)^\circ$.

\begin{lemma}\label{L:derivedEquivalence}
We have an exact equivalence of triangulated categories
\begin{equation*}
\sT:D^b\bigl(\mathsf{Coh}(\EE)\bigr)\lar D^b\bigl(\Lambda\mbox{--}\mathsf{mod}\bigr).
\end{equation*}
\end{lemma}
\begin{proof}
By Lemma \ref{L:rigidVectorBundle}, we know that all $\kG_i$ are rigid and moreover $\Ext_{\EE}^1(\kG_i, \kG_j) = 0$ for all $1 \le i, j \le t$. According to \cite[Proposition 5.1]{LenzingSurvey}, any non-zero morphism $\kG_i \lar \kG_j$ is either an epimorphism or a monomorphism. Using this fact, one can easily show  by induction that for any $r \ge 2$ there are no ``cycles'' of morphisms $\kG_{i_1} \stackrel{\phi_1}\lar \kG_{i_2} \stackrel{\phi_2}\lar \dots \lar \kG_{i_r} \stackrel{\phi_r}\lar \kG_{i_{r+1}} = \kG_{i_{1}}$ with $\kG_{i_k} \not\cong \kG_{i_{k+1}}$ and $\phi_k \ne 0$ for all $1\le k \le r$. It follows that $\Lambda$ is directed, hence $\mathsf{gl.dim}(\Lambda) < \infty$.

Let $\llangle \kF\rrangle$ be the thick subcategory of $D^b\bigl(\Coh(\EE)\bigr)$ generated by $\kF$ and $x = \mathsf{supp}(\kS)$ be the support of $\kS$. For any $k \in \NN$, let $\kS_{[k]} \in \Tor_x(\EE)$  be an indecomposable object of length $k$ (which is unique up to  isomorphism). It is clear that $\kS_{[k]} \in \llangle \kF\rrangle$. Hence, $\kL_{[k]}$ (defined by (\ref{E:AmpleSequence})) belongs to $\llangle \kF\rrangle$ as well. For any coherent (and, as a consequence, for any quasi-coherent) sheaf $\kX$ on $\EE$ there exists $k \in \NN$ such that $\Hom_{\EE}(\kL_{[k]}, \kX)  \ne 0$; see (\ref{E:SomeIsomorphism}). Hence, $\kF$ is a \emph{compact object} in the unbounded derived category $\widetilde\cD = D\bigl(\mathsf{QCoh}(\EE)\bigr)$ of quasi-coherent sheaves on $\EE$ which \emph{compactly generates} $\widetilde\cD$, i.e.
$$\kF^\perp := \left\{\kX^\bu \in \Ob(\widetilde{\cD}) \, \Big| \, \Hom_{\widetilde{\cD}}(\kF, \kX^\bu[n]) = 0 \;  \mbox{for all} \;  n \in \ZZ \right\} = 0.$$
Since $\Ext^1_{\EE}(\kF, \kF) = 0$, $\kF$ is a tilting object in $\widetilde\cD$ in the sense of \cite{Keller} and we get an exact equivalence of triangulated categories 
\begin{equation}\label{E:EQUnbounded}
\sT= \mathsf{RHom}_{\EE}(\kF, \,-\,): D\bigl(\mathsf{QCoh}(\EE)\bigr) \lar D\bigl(\Lambda\mbox{--}\mathsf{Mod}\bigr).
\end{equation}
Moreover, since  $\mathsf{gl.dim}(\Lambda) < \infty$, we get a restricted exact equivalence 
\begin{equation*}
\sT:D^b\bigl(\mathsf{Coh}(\EE)\bigr)\lar D^b\bigl(\Lambda\mbox{--}\mathsf{mod}\bigr)
\end{equation*}
of the triangulated subcategories of compact objects of both sides of (\ref{E:EQUnbounded}); see for instance \cite[Proposition 2.3]{KrauseStableDerived}.
\end{proof}

\begin{lemma}\label{L:decompositionCompanionBundle}
Let $\kS \in \Tor(\EE)$ be any simple object. The corresponding companion bundle has the form $ \overline\kL\cong\kG^{\oplus p}$ for some indecomposable $\kG \in \VB(\EE)$ and some multiplicity $p\in\NN$.
\end{lemma}
\begin{proof}
As a consequence of Lemma \ref{L:derivedEquivalence}, $\Gamma := K_0(\EE) \cong K_0(\Lambda)$ is a free abelian group of finite rank. Recall that the homomorphism of abelian groups 
$$\Gamma \lar \Hom_{\ZZ}(\Gamma, \ZZ), \; \gamma \mapsto \langle\gamma, \,-\,\rangle$$
is injective, where $\langle -\,,\,-\rangle$ is the Euler form on $\Gamma$; see for instance \cite[Section III.1.3]{HappelTriangulated}. By Lemma \ref{L:GrothGroup} we have $\Gamma = \llangle \alpha, \Gamma'\rrangle$, where $\alpha = [\kL]$ and $\Gamma'$ is the subgroup of $\Gamma$ generated by the classes of torsion sheaves. Since $\EE$ is homogeneous, we have $\langle\gamma_1, \gamma_2\rangle = 0$ for any $\gamma_1, \gamma_2 \in \Gamma'$. As a consequence,  for any $\gamma_1, \gamma_2 \in \Gamma'$ we have $\gamma_1 = \gamma_2$  if and only if $\langle\alpha, \gamma_1\rangle = \langle\alpha, \gamma_2\rangle$. Hence, there exist $k_1, k_2 \in \ZZ$ such that $k_1 \gamma_1 = k_2 \gamma_2$. Since $\Gamma' = \mathsf{Ker}({\mathsf{rk}}:\Gamma\lar \ZZ)$ and $\Gamma$ is a free abelian group of finite rank, it follows that $\Gamma' \cong \ZZ$ and  $\Gamma  \cong \ZZ^2$. As a consequence, $\kF=\kL\oplus\kG$ and thus $\overline{\kL}$ is isomorphic to $\kG^{\oplus p}$ for some indecomposable $\kG \in \VB(\EE)$ and some multiplicity $p\in\NN$.
\end{proof}

We want to specify our choice of $\kS$ in (\ref{E:companion}). To this end, we first need the following construction.

\begin{definition}
For any $x \in X_\circ$ we shall denote by $\kS_x$ the unique (up to isomorphisms) simple object of $\Tor_x(\EE)$  and put $w_x := \bigl[\kS_x\bigr] \in \Gamma'$, where $\Gamma'$ is as in Lemma \ref{L:decompositionCompanionBundle}. It is clear that $\Hom_{\EE}(\kL, \kS_x) \ne 0$. As a consequence, $\Ext^1_{\EE}(\kS_x, \kL) \cong \bigl(\Hom_{\EE}(\kL, \kS_x)\bigr)^\ast \ne 0$. We denote by $\kL\bigl([x]\bigr)$ the middle term of any non-split extension
\begin{equation}\label{E:ExtensionByDivisor}
0 \lar \kL \lar \kL\bigl([x]\bigr) \lar \kS_x \lar 0. 
\end{equation}
It is not difficult to see that $\kL\bigl([x]\bigr)$ is torsion free, hence a line bundle; see the proof of Lemma \ref{L:rigidVectorBundle} for a detailed treatment of a similar statement. Since $\bigl[\kL\bigl([x]\bigr)\bigr] = \alpha + w_x$, it follows from Theorem \ref{T:HereditaryKey}(b) that the isomorphism class of $\kL\bigl([x]\bigr)$ does not depend on the choice of a non-split extension (\ref{E:ExtensionByDivisor}).

\noindent
Moreover, we denote by $\kL\bigl(q[x]\bigr)$ the line bundle obtained by iterating the above construction $q$ times (replacing $\kL$ and $\kL\bigl([x]\bigr)$ by $\kL\bigl((q-1)[x]\bigr)$ and $\kL\bigl(q[x]\bigr)$ respectively).
\end{definition}

\begin{lemma}\label{L:torsionGeneratingSublattice}
There exists a simple torsion sheaf $\kS$ such that $[\kS]$ generates $\Gamma'$.
\end{lemma}
\begin{proof}
Let $w \in \Gamma'$ be such that $\Gamma' = \llangle w\rrangle$. To show that there is an $x\in X_\circ$ such that $w=w_x$, consider the group homomorphism $\overline{\mathsf{deg}}:\Gamma \lar \ZZ$ defined by $\gamma \mapsto \langle \alpha, \gamma\rangle$.

For any $y \in X_\circ$ we have $\overline{\mathsf{deg}}(w_y) > 0$. Let $x \in X_\circ$ be such that $\overline{\mathsf{deg}}(w_x)$ is minimal. We claim that $\Gamma' = \llangle w_x\rrangle$. Indeed, for any $y \in X_\circ$ there exist unique $q \in \NN$ and $0 \le r < \overline{\mathsf{deg}}(w_x)$ such that $\overline{\mathsf{deg}}(w_y) = q \,\overline{\mathsf{deg}}(w_x) + r$. It suffices to show that $r = 0$. 

We put $\kL':= \kL\bigl([y]\bigr)$ and $\kL'':= \kL\bigl(q[x]\bigr)$. Then $[\kL'] = \alpha + w_y$, $[\kL''] = \alpha + q w_x$. As $\EE$ is homogeneous, we have $\langle w_x,  w_y\rangle=0$ and $\langle w_x, \alpha \rangle=-\langle \alpha, w_x\rangle$. Thus
$$\bigl\langle \kL'', \kL'\bigr\rangle = \langle \alpha + q w_x, \alpha + w_y\rangle = \langle \alpha, \alpha\rangle + \langle \alpha, w_y \rangle - q \langle \alpha, w_x\rangle = 
\langle \alpha, \alpha\rangle + r > 0.$$
Assume that $r \ne 0$. Then $\kL' \not \cong \kL''$. On the other hand, we have a non-zero morphism $g:\kL''\lar \kL'$, which is automatically a monomorphism since $\kL'$ and $\kL''$ are line bundles. Consider the corresponding short exact sequence
$$0 \lar \kL'' \stackrel{g}\lar \kL' \lar \kZ \lar 0.$$
Then $[\kZ] = w_y - q w_x$ and $\kZ \in \Tor(\EE)$. It follows that $\overline{\mathsf{deg}}([\kZ]) = r < \overline{\mathsf{deg}}(w_x)$, contradicting the minimality of $\overline{\mathsf{deg}}(w_x)$.
\end{proof}

By Lemma \ref{L:torsionGeneratingSublattice}, we know that a simple torsion sheaf $\kS$ with $[\kS] = w$ exists. Let us specify this choice in (\ref{E:companion}). We get a stronger version of Lemma \ref{L:decompositionCompanionBundle}.
\begin{lemma}\label{L:companionIndecomposable}
Let $\kS$ be a simple torsion sheaf such that $w=[\kS]$ generates $\Gamma'$ and let $\overline\kL$ be the corresponding companion bundle. Then $\overline{\kL}$ is indecomposable.
\end{lemma}
\begin{proof}
We already know by Lemma \ref{L:decompositionCompanionBundle} that $\overline\kL \cong \kG^{\oplus p}$, where $\kG$ is an indecomposable rigid vector bundle on $\EE$. We only need to prove that $p = 1$.

Since $(\alpha, w)$  is a basis of $\Gamma$,  we can find $k, l \in \ZZ$ such that $[\kG] = k \alpha + l w\in \Gamma$. It follows that $p[\kG] = k p \alpha   + l p w$.  On the other hand, the  short exact sequence (\ref{E:companion}) implies the equality $p[\kG] = m \alpha  + w$. Since $\alpha$ and $w$ are linearly independent, it follows that $p = 1$.
\end{proof}

We can summarize the previous discussion on $\kL$ and its companion bundle $\overline\kL$ with the following result.
\begin{theorem}\label{T:Lenzing}
Let $\EE = (X, \kR)$ be a minimal  exceptional homogeneous curve. Then there exists a tilting object $\kF \in \VB(\EE)$ such that 
\begin{equation}\label{E:TameBimodule}
\Lambda := \bigl(\End_{\EE}(\kF)\bigr)^\circ \cong
\left(
\begin{array}{cc}
\ff & \mathbbm{w} \\
0   & \gg \\
\end{array}
\right),
\end{equation}
where $\mathbbm{f}$ and $\mathbbm{g}$ are finite--dimensional division algebras over $\kk$, and $\mathbbm{w}$ is a \emph{tame} $(\ff$--$\gg)$-bimodule (meaning that $\dim_{\ff}(\mathbbm{w}) \cdot \dim_{\gg}(\mathbbm{w}) = 4$; see  \cite{RingelBimod, DlabRingel}). 
\end{theorem}
\begin{proof}
Let $\kS$ be a simple torsion sheaf such that $w=[\kS]$ generates $\Gamma'$, which exists by Lemma \ref{L:torsionGeneratingSublattice}. Let $\overline\kL$ be the companion bundle of $\kL$ corresponding to $\kS$. Then $\overline\kL$ is indecomposable by Lemma \ref{L:companionIndecomposable}. Moreover, $\kF=\kL\oplus\overline\kL$ is a tilting bundle by the construction in Lemma \ref{L:derivedEquivalence}. It is clear that its endomorphism algebra $\Lambda = \bigl(\End_{\EE}(\kF)\bigr)^\circ$ has the form (\ref{E:TameBimodule}), where $\ff=\bigl(\End_{\EE}(\kL)\bigr)^\circ$, $\gg=\bigl(\End_{\EE}(\overline{\kL})\bigr)^\circ$ and $\mathbbm{w}=\Hom_{\EE}(\kL,\overline\kL)$. Note that by Theorem \ref{T:HereditaryKey}, $\ff$ and $\gg$ are finite--dimensional division algebras over $\kk$.

\smallskip
\noindent
It remains to show that $\mathbbm{w}$ is a tame bimodule. We put
\begin{equation}\label{E:KappaEpsilon}
\varepsilon:= \mathsf{dim}_{\ff}\bigl(\Ext^1_{\EE}(\kS, \kL)\bigr) = 
\frac{\mathsf{dim}_{\kk}\bigl(\Ext^1_{\EE}(\kS, \kL)\bigr)}{\mathsf{dim}_{\kk}\bigl(\End_{\EE}(\kL)\bigr)} \quad \mbox{\rm and} \quad \kappa:= \mathsf{dim}_{\kk}\bigl(\End_{\EE}(\kL)\bigr) = \mathsf{dim}_{\kk}(\ff).
\end{equation}
Now use the fact that all morphisms from torsion sheaves to vector bundles vanish to obtain $\langle w, \alpha\rangle = \bigl\langle  [\kS], [\kL]\bigr\rangle = - \kappa \varepsilon$. Further, $w\in\Gamma'$ for a homogeneous curve $\EE$, so $\langle w, w\rangle = \bigl\langle  [\kS], [\kS]\bigr\rangle = 0$ and $\langle  \alpha, w \bigr\rangle=-\langle w,\alpha \bigr\rangle = \kappa \varepsilon$. Finally, $\langle\alpha, \alpha \rangle = \bigl\langle  [\kL], [\kL]\bigr\rangle = \kappa$ by definition. Note that $\varepsilon$ is the value of $m$ in (\ref{E:companion}) if we specify $\kS$ to be a simple torsion sheaf such that $[\kS]$ generates $\Gamma'$. It follows that $[\overline\kL]=\varepsilon\alpha+w$. Thus we get
\begin{equation*}
\mathsf{dim}_{\kk}(\gg) = \bigl\langle  [\overline\kL], [\overline\kL]\bigr\rangle = \bigl\langle  \varepsilon \alpha  + w, \varepsilon \alpha  + w \bigr\rangle = \kappa\varepsilon^2
\end{equation*}
and
\begin{equation*}
\mathsf{dim}_{\kk}(\mathbbm{w}) = \bigl\langle  [\kL], 
[\overline\kL]\bigr\rangle = \bigl\langle  \alpha, 
\varepsilon \alpha  + w\bigr\rangle = 2 \kappa \varepsilon. 
\end{equation*}
Note that $\mathsf{dim}_{\ff}(\mathbbm{w}) = \dfrac{2 \kappa \varepsilon}{\kappa} = 2 \varepsilon$ and $\mathsf{dim}_{\gg}(\mathbbm{w}) = \dfrac{2 \kappa \varepsilon}{\kappa \varepsilon^2} = \dfrac{2}{\varepsilon}$. Hence, $\varepsilon \in \{1, 2\}$ and 
$\mathsf{dim}_{\ff}(\mathbbm{w}) \cdot \mathsf{dim}_{\gg}(\mathbbm{w}) = 4$, as asserted.
\end{proof}

\begin{corollary} Let $\EE = (X, \kR)$ be a minimal exceptional homogeneous curve over $\kk$, $\eta \in \mathsf{Br}(\KK)$ be the corresponding Brauer class and 
$\EE' = (X, \kR')$ be another minimal homogeneous curve  such that $\eta_{\EE'} = \eta$. Then we have: $H^1(X, \kR') = 0$. In other words, the condition for $\eta \in \mathsf{Br}(\KK)$ to be exceptional is well-defined. Accordingly, any homogeneous (not necessarily minimal) curve corresponding to $\eta$ will be called \emph{exceptional}. 
\end{corollary}

\begin{proof} We have an equivalence of categories $\Coh(\EE') \simeq \Coh(\EE)$, which restricts to an equivalence  $\VB(\EE') \simeq \VB(\EE)$ sending  $\kL'$ to some line bundle $\kN \in \VB(\EE)$. By \cite{DlabRingel}, all indecomposable preprojective and preinjective $\Lambda$-modules are rigid, where $\Lambda$ 
is the $\kk$-algebra defined by (\ref{E:TameBimodule}). As a consequence, any indecomposable object of $\VB(\EE)$ and $\VB(\EE')$ is rigid, too. Hence, we have:
$$
H^1(X, \kR') \cong \Ext^1_{\EE'}(\kL', \kL') \cong  \Ext^1_{\EE}(\kN, \kN) = 0,
$$
as asserted.
\end{proof}

\begin{remark} Let $X$ be a regular proper curve over $\kk$ of genus zero. Then the zero class $0 \in \mathsf{Br}(\KK)$ is exceptional. Moreover,  Theorem \ref{T:Lenzing} is true for all exceptional hereditary curves and not only for the minimal ones.
\end{remark}

\begin{remark}  The statement of Theorem \ref{T:Lenzing} goes back to a work of Lenzing \cite[Theorem 4.5]{LenzingExceptionalCurve}, where the corresponding proof was briefly sketched. Some of our arguments are inspired by the proof of \cite[Proposition 4.1]{LenzingdelaPena}.
\end{remark}

\subsection{Combinatorial Parameters}
In this section, we review some combinatorial parameters that will be important later. More precisely, we give a sheaf theoretic interpretation of the parameters appearing in the symbol of Lenzing \cite{LenzingKTheory}. This symbol is used in Section \ref{S:ReflectionGroups} to define the reflection groups associated to the categories $\Coh(\XX)$.

\smallskip\noindent
First, consider the homogeneous case $\EE=(X,\kR)$. Recall that by Lemma \ref{L:derivedEquivalence} we have a derived equivalence $\sT:D^b\bigl(\mathsf{Coh}(\EE)\bigr)\lar D^b\bigl(\Lambda\mbox{--}\mathsf{mod}\bigr)$.
\begin{definition}
For any $x\in X_\circ$, let $\kS_x\in\Tor_x(\EE)$ be the corresponding simple object and $S_x:=\Hom_{\EE}(\kF,\kS_x)\in \Lambda\mbox{--}\mathsf{mod}$ be its image under ~$\sT$. Then $S_x=\begin{bmatrix}\mathbbm{u}_x \\ \mathbbm{v}_x\end{bmatrix}$ is a simple regular $\Lambda$--module (in the sense of \cite{DlabRingel}), where
\begin{equation}\label{E:UandV}
\mathbbm{u}_x = \Hom_{\EE}(\kL, \kS_x) \quad \mbox{\rm and} \quad \mathbbm{v}_x = \Hom_{\EE}(\overline{\kL}, \kS_x).
\end{equation}
It is clear that $\dd_x:= \bigl(\End_{\EE}(\kS_x)\bigr)^\circ \cong \bigl(\End_{\Lambda}(S_x)\bigr)^\circ$ is a finite--dimensional division algebra over $\kk$. Moreover, $\dd_x \cong \widehat{\kR}_x/\kJ_x$, where $\kJ_x = \mathsf{rad}(\widehat{\kR}_x)$. Let $\varepsilon$ be as in (\ref{E:KappaEpsilon}). We define the \emph{parameters associated to $x\in X_\circ$} by
\begin{equation}\label{E:LenzingsParameters}
e_x := \dim_{\dd_x}\bigl(\Hom_{\EE}(\kL, \kS_x)\bigr), \;
f_x :=  \frac{1}{\varepsilon}\dim_{\ff}\bigl(\Hom_{\EE}(\kL, \kS_x)\bigr)  \;  \mbox{\rm and} \;
d_x := e_x f_x.
\end{equation}
\end{definition}

Let us give an interpretation of $f_x$ as well as dimension formulas using the combinatorial data.
\begin{lemma}
We have an isomorphism of groups
$$\mathsf{deg}:\Gamma'\lar \ZZ, \; \gamma \mapsto \frac{1}{\kappa \varepsilon}\langle \alpha, \gamma\rangle,$$
where $\alpha=[\kL]$. For any $x\in X_\circ$, we have $f_x=\mathsf{deg}([S_x])=\mathsf{dim}_{\gg}\bigl(\Hom_{\EE}(\overline\kL, \kS_x)\bigr)$. Moreover, we have the dimension formulas
\begin{equation}\label{E:DimensionD}
\dim_{\kk}(\mathbbm{u}_x)=\kappa\varepsilon f_x, \;\; \dim_{\kk}(\mathbbm{v}_x)=\kappa\varepsilon^2 f_x\;\;  \mbox{\rm and} \; \;\dim_{\kk}(\dd_x) = \frac{\kappa\varepsilon f_x}{e_x}.
\end{equation}
\end{lemma}
\begin{proof}
It is clear that the map $\Gamma'\lar \ZZ$, $ \gamma \mapsto \frac{1}{\kappa \varepsilon}\langle \alpha, \gamma\rangle$ is a morphism of groups. Now let $x\in X_\circ$ and note that $\bigl\langle [\kL], [\kS_x]\bigr\rangle = \kappa\varepsilon f_x$, where $\kappa$ is as defined in (\ref{E:KappaEpsilon}). Therefore 
$$\mathsf{dim}_{\gg}\bigl(\Hom_{\EE}(\overline\kL, \kS_x)\bigr) = 
\dfrac{\bigl\langle [\overline\kL], [\kS_x]\bigr\rangle}{\kappa\varepsilon^2} = 
\dfrac{\bigl\langle \varepsilon[\kL] + [\kS], [\kS_x]\bigr\rangle}{\kappa\varepsilon^2} = \dfrac{\bigl\langle [\kL], [\kS_x]\bigr\rangle}{\kappa\varepsilon} = f_x,$$
which, in particular, implies that $f_x=\mathsf{deg}([S_x]) \in \NN$. In fact, it was shown in Lemma \ref{L:torsionGeneratingSublattice} that there exists $x\in X_\circ$ such that $f_x=1$, so $\mathsf{deg}$ is an isomorphism.

From the above discussion, we immediately deduce the following dimension formulas for $\dim_{\kk}(\mathbbm{u}_x)$ and $\dim_{\kk}(\mathbbm{v}_x)$. Finally, $\kappa\varepsilon f_x=\dim_{\kk}\bigl(\Hom_{\EE}(\kL, \kS_x)\bigr)=e_x \, \dim_{\kk}(\dd_x)$, from which we conclude the formula for $\dim_{\kk}(\dd_x)$.
\end{proof}

We now turn from homogeneous curves to the general case.
\begin{definition}
Let $\XX = (X, \kH)$ be a hereditary curve, $\KK = \kk(X)$ and $\eta = \bigl[\FF_{\XX}\bigr] \in \mathsf{Br}(\KK)$.  We have the following notions describing local properties of  $\XX$.
\begin{enumerate}
\item[(a)] For any $x \in X_\circ$, let ${O}_x$ be the \emph{completion} of the stalk of $\kO$ at $x$ and ${\KK}_x$ be its quotient field. Then we have: ${O}_x \cong \kk_x\llbracket w\rrbracket$ and ${\KK}_x \cong \kk_x\llbrace w\rrbrace$, where $\kk_x$ is the residue field of $O_x$.
\item[(b)] We have a natural field extension $\KK \lar {\KK}_x$, which induces a group homomorphism $\mathsf{Br}(\KK) \lar \mathsf{Br}({\KK}_x)$. Let $\eta_x$ be the image of the class $\eta$ under this map. Then $\eta_x$ defines a uniquely determined  skew field $\FF_x$, whose center is ${\KK}_x$.
\end{enumerate}
\end{definition}

\begin{remark}\label{R:HereditaryOrders}
Let $\XX = (X, \kH)$ be a hereditary curve and $x\in X_\circ$. Let $R_x$ be the maximal order in $\FF_x$. In fact, $R_x$  is the integral closure of ${O}_x$  in $\FF_x$ (see \cite[Theorem 12.8]{Reiner}), so $R_x$ is actually unique. Let $\widehat{\kH}_x$ be the completion of the stalk of $\kH$ at the point $x$. Then $\widehat{\kH}_x$ is Morita equivalent to 
\begin{equation}\label{E:Hereditary}
H_x:=  \left[
\begin{array}{cccc}
R_x & R_x & \dots & R_x \\
J_x & R_x & \dots & R_x\\
\vdots & \vdots & \ddots & \vdots \\J_x  & J_x & \dots & R_x\\
\end{array}
\right] \subseteq \mathsf{Mat}_{\rho(x)}(R_x)
\end{equation}
for some $\rho(x) \in \NN$, where $J_x$ is the Jacobson radical of $R_x$; see \cite[Theorem 39.14]{Reiner}. In fact, $\rho(x)$ is the number of pairwise non-isomorphic simple objects of $\Tor_x(\XX)$.
\end{remark}

This $\rho$ provides the final important combinatorial parameter, which we record in the following definition.
\begin{definition}\label{D:weightFunction}
The map $\rho:X_{\circ}\lar \NN$ given by the number of pairwise non-isomorphic simple objects of $\Tor_x(\XX)$ is called the \emph{weight function} of $\XX$. It defines a finite set
$$\mathfrak{E}_{\rho} := \bigl\{x \in X_\circ \, \big| \, \rho(x) \ge 2 \bigr\}.$$
\end{definition}

\begin{remark}
Let $\XX = (X, \kH)$ be a hereditary curve and $x\in X_\circ$. If $\eta_x = 0$ then $R_x = O_x$ and
$$
H_x  \cong  \left[
\begin{array}{cccc}
\kk_x\llbracket w\rrbracket & \kk_x\llbracket w\rrbracket  & \dots & \kk_x\llbracket w\rrbracket \\
(w) & \kk_x\llbracket w\rrbracket & \dots & \kk_x\llbracket w\rrbracket \\
\vdots & \vdots & \ddots & \vdots \\ (w)  & (w)  & \dots & \kk_x\llbracket w\rrbracket\\
\end{array}
\right],
$$
which is  isomorphic to the arrow completion  of the path algebra of the cyclic quiver 
\begin{equation*}
\begin{array}{c}
\xymatrix{ & \stackrel{\rho(x)}{\circ}  \ar[rr] && \circ  \ar[rd]\\
\stackrel{1}\circ  \ar[ru] &         &   & \vdots  & \ar[ld] \circ\\
& \stackrel{2}\circ  \ar[lu] & & \circ \ar[ll] & \\
}
\end{array}
\end{equation*}
over the field $\kk_x$. 
\end{remark}

\begin{lemma}\cite[Corollary 7.9]{BurbanDrozd}\label{L:Weighting} Given a datum $(X, \eta, \rho)$ as in Definition \ref{D:weightFunction} there exists an associated  hereditary curve $\XX$. Moreover, such $\XX$ is unique up to \emph{Morita equivalence}. Namely, let $(X', \eta', \rho')$ be another  datum as above and $\XX'$ be a  hereditary curve attached to it. Then the categories $\Coh(\XX)$ and $\Coh(\XX')$ are equivalent if and only if there exists an isomorphism $f:X \lar X'$ such that $f^\ast(\eta') = \eta \in \mathsf{Br}(\KK)$ and $\rho' f = \rho$.
\end{lemma}

\begin{remark} Let $\EE = (X, \kR)$ be a minimal homogeneous curve. Then for any 
$x \in X_\circ$, we have the following isomorphisms of vector spaces over $\kk$:
$$
\Hom_{\EE}(\kL, \kS_x) \cong 
\Gamma\bigl(X, \mathit{\Hom}_{\kR}(\kL, \kS_x)\bigr) \cong 
\Hom_{\widehat{\kR}_x}(\widehat{\kR}_x, U_x),
$$
where $U_x$ is the simple $\widehat{\kR}_x$--module. It is clear that $\dd_x \cong \bigl(\End_{\widehat{\kR}_x}(U_x)\bigr)^\circ$. In the notation of Remark \ref{R:HereditaryOrders}, we get $\widehat{\kR}_x \cong \Mat_{e_x}(R_x)$, providing another interpretation of the parameter $e_x$ defined by (\ref{E:LenzingsParameters}); see also \cite[Section 3]{KussinWeightedCurve}.
\end{remark}

\subsection{Complete Exceptional Sequences}
In this section, we discuss complete exceptional sequences for the category $\Coh(\XX)$, where $\XX$ is an exceptional hereditary curve. We construct the standard exceptional sequence, which will be employed in the subsequent sections. Moreover, we compare  the category $\Coh(\XX)$ to derived equivalent categories known in the literature. Finally, we recall some key properties of complete exceptional sequences in $\Coh(\XX)$ from \cite{KussinMeltzer}.

Let us recall how to associate to any hereditary curve a minimal homogeneous curve. We refer to \cite[Section 4]{BurbanDrozdGavran} for the proof and a more detailed treatment.
\begin{lemma}\label{L:MinorsFunctors}
Let $\XX$ be a hereditary curve, $\kP$ be any line bundle on $\XX$ and $\kR := 
\bigl(\mathit{End}_{\kH}(\kP)\bigr)^\circ$. Then $\EE := (X, \kR)$ is a minimal homogeneous curve and $\eta_{\EE} = \eta_{\XX}$, so the Morita type of $\EE$ is determined by $(X, \eta)$. We have the following functors:
\begin{itemize}
\item $\sG := \sHom_\kH(\kP, \,-\,)$ from  $\Coh(\XX)$ to $\Coh(\EE)$.
\item $\sF := \kP \otimes_\kR \,-\,$  from $\Coh(\EE)$ to $\Coh(\XX)$.
\end{itemize}
Note that $(\sF, \sG)$ is an adjoint pair of exact functors. Moreover, both $\sF$ and the induced derived functor $\sD\sF: D^b\bigl(\Coh(\EE)\bigr) \lar D^b\bigl(\Coh(\XX)\bigr)$ are fully faithful.
\end{lemma}

\begin{definition} A hereditary curve $\XX = (X, \kH)$ is called \emph{exceptional} if $X$ has genus zero and $\eta_\XX$ is an exceptional Brauer class. 
\end{definition}

The following standard exceptional sequence is used in the construction of the reflection group associated to $\Coh(\XX)$ for its nice computational properties. In fact, the construction does not depend on the choice of a complete exceptional sequence, as we will see in Section ~\ref{S:OrderPreservingBijections}.
\begin{definition}
Let $\XX$ be an exceptional hereditary curve with the associated  datum $(X,\eta,\rho)$ and $\mathfrak{E}_{\rho}=:\bigl\{x_1, \dots, x_t\bigr\}$. Let $\kP$ be a line bundle on $\XX$ and $\EE$ be the corresponding exceptional minimal homogeneous curve as in Lemma \ref{L:MinorsFunctors}. For any $1\le i\le t$, we put $p_i:=\rho(x_i)$ and denote by $\kS_i^{(0)},\dots,\kS_i^{(p_i-1)}$ the simple objects of $\Tor_{x_i}(\XX)$ such that $\Hom_{\XX}\bigl(\kP,\kS_i^{(0)}\bigr)\ne0$ and $\tau\bigl(\kS_{i}^{(j)}\bigr)\cong \kS_{i}^{(j-1)}$ for all $j\in\ZZ_{p_i}$. Let $\kS\in\Coh(\EE)$ be a torsion sheaf such that its class $[\kS]$ in $K_0(\EE)$ generates the subgroup $\Gamma'$ of $K_0(\EE)$ generated by the classes of all torsion sheaves in $\Coh(\EE)$. Let $\overline{\kL}\in\VB(\EE)$ be the companion bundle of $\kL$ corresponding to $\kS$ and let $\overline{\kP}=\sF(\overline{\kL})\in\VB(\XX)$. Then we call
\begin{equation}\label{E:StandardCollection}
\bigl(\kS_t^{(p_t-1)}, \dots,\kS_t^{(1)},  \dots, \kS_1^{(p_1-1)}, \dots,\kS_1^{(1)}, \kP, \overline\kP\bigr)
\end{equation}
the \emph{standard exceptional sequence} in $D^b\bigl(\Coh(\XX)\bigr)$.
\end{definition}

The combinatorial parameters introduced in the previous subsection will appear in the following result. Recall that $\kappa$ and $\varepsilon$ are the parameters of $\EE$ from Theorem \ref{T:Lenzing}, whereas $e_i=e_{x_i}$ and $f_i=f_{x_i}$ are given by the formula (\ref{E:LenzingsParameters}) for all $1\le i\le t$.

\begin{theorem}\label{T:ExceptionalMain} Let $\XX$ be an exceptional hereditary curve. Then the standard exceptional sequence (\ref{E:StandardCollection}) is indeed a complete exceptional sequence in $D^b\bigl(\Coh(\XX)\bigr)$. Its Gram matrix with respect to the Euler form is given by
\begin{equation}\label{E:GramMatrix}
\left[ \begin{array}{@{}c@{}|@{}c@{}|@{}c@{}|@{}c@{}|@{}c@{}|@{}c@{}}
\begin{array}{@{}c@{}@{}c@{}@{}c@{}@{}c@{}}
\frac{\kappa\varepsilon f_t}{e_t} & \frac{-\kappa\varepsilon f_t}{e_t} &  &  0 \\ 
 & \frac{\kappa\varepsilon f_t}{e_t} & \ddots  &   \\ 
 & & \ddots & \frac{-\kappa\varepsilon f_t}{e_t}  \\ 
0 &     & & \frac{\kappa\varepsilon f_t}{e_t}  
\end{array} & \cdots & 0 & 0 & \begin{array}{@{}c@{}} 0 \vspace{1mm} \\ \vdots \vspace{2mm} \\ 0 \vspace{1mm} \\ -\kappa\varepsilon f_t
\end{array} & \begin{array}{@{}c@{}} 0 \vspace{1mm} \\ \vdots \vspace{2mm} \\ 0 \vspace{1mm} \\ -\kappa\varepsilon^2 f_t
\end{array} \\ \hline
\vdots & \ddots & \vdots & \vdots & \vdots & \vdots \\ \hline
0 & \cdots & \begin{array}{@{}c@{}@{}c@{}@{}c@{}@{}c@{}}
\frac{\kappa\varepsilon f_2}{e_2} & \frac{-\kappa\varepsilon f_2}{e_2} &  &  0 \\ 
 & \frac{\kappa\varepsilon f_2}{e_2} & \ddots  &   \\ 
 & & \ddots & \frac{-\kappa\varepsilon f_2}{e_2}  \\ 
0 &     & & \frac{\kappa\varepsilon f_2}{e_2}  
\end{array} & 0 & \begin{array}{@{}c@{}} 0 \vspace{1mm} \\ \vdots \vspace{2mm} \\ 0 \vspace{1mm} \\ -\kappa\varepsilon f_2
\end{array} & \begin{array}{@{}c@{}} 0 \vspace{1mm} \\ \vdots \vspace{2mm} \\ 0 \vspace{1mm} \\ -\kappa\varepsilon^2 f_2
\end{array} \\ \hline
0 & 0 & \cdots & \begin{array}{@{}c@{}@{}c@{}@{}c@{}@{}c@{}}
\frac{\kappa\varepsilon f_1}{e_1} & \frac{-\kappa\varepsilon f_1}{e_1} &  &  0 \\ 
 & \frac{\kappa\varepsilon f_1}{e_1} & \ddots  &   \\ 
 & & \ddots & \frac{-\kappa\varepsilon f_1}{e_1}  \\ 
0 &     & & \frac{\kappa\varepsilon f_1}{e_1}  
\end{array} & \begin{array}{@{}c@{}} 0 \vspace{1mm} \\ \vdots \vspace{2mm} \\ 0 \vspace{1mm} \\ -\kappa\varepsilon f_1
\end{array} & \begin{array}{@{}c@{}} 0 \vspace{1mm} \\ \vdots \vspace{2mm} \\ 0 \vspace{1mm} \\ -\kappa\varepsilon^2 f_1
\end{array} \\ \hline
\begin{array}{@{}c@{}@{}c@{}@{}c@{}@{}c@{}} 0 \hspace{4mm} & \cdots \hspace{4mm} & 0 \hspace{5mm} & 0
\end{array} & \cdots & \begin{array}{@{}c@{}@{}c@{}@{}c@{}@{}c@{}} 0 \hspace{4mm} & \cdots \hspace{4mm} & 0 \hspace{5mm} & 0
\end{array} & \begin{array}{@{}c@{}@{}c@{}@{}c@{}@{}c@{}} 0 \hspace{4mm} & \cdots \hspace{4mm} & 0 \hspace{5mm} & 0
\end{array} & \kappa & 2\kappa\varepsilon\\ \hline
\begin{array}{@{}c@{}@{}c@{}@{}c@{}@{}c@{}} 0 \hspace{4mm} & \cdots \hspace{4mm} & 0 \hspace{5mm} & 0
\end{array} & \cdots & \begin{array}{@{}c@{}@{}c@{}@{}c@{}@{}c@{}} 0 \hspace{4mm} & \cdots \hspace{4mm} & 0 \hspace{5mm} & 0
\end{array} & \begin{array}{@{}c@{}@{}c@{}@{}c@{}@{}c@{}} 0 \hspace{4mm} & \cdots \hspace{4mm} & 0 \hspace{5mm} & 0
\end{array} & 0 & \kappa\varepsilon^2
	\end{array}\right]
\end{equation}
\end{theorem}

\begin{proof} By \cite[Theorem 4.5]{BurbanDrozdGavran}, we have a semi-orthogonal decomposition
\begin{equation}\label{E:KeySemiOrth}
D^b\bigl(\Coh(\XX)\bigr) = \bigl\langle \mathsf{Ker}(\sD\sG), \mathsf{Im}(\sD \sF)\bigr\rangle.
\end{equation}
It follows from Theorem \ref{T:Lenzing} that $\mathsf{Im}(\sD \sF) = \llangle \kP, \overline\kP\rrangle$. Moreover, since $\sD\sF$ is fully faithful, the pair $(\kP, \overline\kP)$ is  exceptional. Next, for any $1 \le i \le t$, let $H_i = H_{x_i}$  be given by (\ref{E:Hereditary}) and
\begin{equation*}
I_i:=  \left[
\begin{array}{cccc}
R_i & R_i & \dots & R_i \\
J_i & J_i & \dots & J_i\\
\vdots & \vdots & \ddots & \vdots \\J_i  & J_i & \dots & J_i\\
\end{array}
\right] \subseteq \mathsf{Mat}_{p_i}(R_i),
\end{equation*}
where $R_i = R_{x_i}$ and  $J_i = J_{x_i}$. Then we have 
$$
L_i:= H_i/I_i \cong \left[
\begin{array}{cccc}
\dd_i & \dd_i & \dots & \dd_i \\
0 & \dd_i & \dots & \dd_i\\
\vdots & \vdots & \ddots & \vdots \\0  & 0 & \dots & \dd_i\\
\end{array}
\right] \subseteq \mathsf{Mat}_{p_i-1}(\dd_i),
$$
where $\dd_i = \dd_{x_i}$. We put $L := L_t \times \dots \times L_1$. It follows from \cite[Theorem 4.6]{BurbanDrozdGavran} that
$$D^b(L\mbox{--}\mathsf{mod}) \simeq \mathsf{Ker}(\sD\sG) = \llangle \kS_t^{(1)}, \dots, \kS_t^{(p_t-1)}, \dots, \kS_1^{(1)}, \dots, \kS_1^{(p_1-1)} \rrangle.$$
It is clear that $\bigl(\kS_i^{(p_i-1)}, \dots, \kS_i^{(1)}\bigr)$ is a full exceptional sequence in each block $D^b(L_i\mbox{\textsf{--mod}})$ of the category $D^b(L\mbox{\textsf{--mod}})$. This implies that (\ref{E:StandardCollection}) is indeed a full and thus complete exceptional sequence, as asserted.

Let us now compute the Gram matrix of the standard exceptional sequence. First, note that $\bigl\langle \bigl[\kS_i^{(j)}\bigr], \, \bigl[\kS_{i'}^{(j')}\bigr]\bigr\rangle = 0$ for all $1\le i\ne i'\le t$ and $1\le j\le p_i-1$, $1 \le j'\le p_{i'}-1$. Let $1\le i\le t$ and $1\le j,k\le p_i-1$. It is clear that
\begin{equation*}
\Hom_{\XX}(\kS_i^{(j)}, \kS_i^{(k)}) \cong \left\{
\begin{array}{cc}
\dd_i^\circ & \mbox{\rm if} \; j=k \\
0 & \mbox{\rm otherwise}.
\end{array}
\right.
\end{equation*}
Using (\ref{E:ARDuality}), we conclude that 
\begin{equation*}
\Ext^1_{\XX}(\kS_i^{(j)}, \kS_i^{(k)}) \cong \left\{
\begin{array}{cc}
\dd_i & \mbox{\rm if} \; k=j-1\\
0 & \mbox{\rm otherwise}.
\end{array}
\right.
\end{equation*}
Expressing the dimensions with the combinatorial data as in (\ref{E:DimensionD}) gives 
$$\bigl\langle \bigl[\kS_i^{(j)}\bigr], \bigl[\kS_i^{(k)}\bigr]\bigr\rangle=\left\{
\begin{array}{cl}
\frac{\kappa\varepsilon f_i}{e_i} & \mbox{if} \; j=k\\
- \frac{\kappa\varepsilon f_i}{e_i} & \mbox{if} \; k=j-1\\
0 & \mbox{otherwise}.
\end{array}
\right.$$
Since the functor  $\sD\sF:D^b\bigl(\Coh(\EE)\bigr)\lar D^b\bigl(\Coh(\XX)\bigr)$ is fully faithful, it follows from Theorem ~\ref{T:Lenzing} that the Gram matrix of the exceptional pair $(\kP, \overline\kP)$ is given by
$$\left(
\begin{array}{cc}
\bigl\langle [\kP], [\kP]\bigr\rangle & \bigl\langle [\kP], [\overline\kP]\bigr\rangle \\
\bigl\langle [\overline\kP], [\kP]\bigr\rangle & \bigl\langle [\overline\kP], [\overline\kP]\bigr\rangle
\end{array}
\right)=\left(
\begin{array}{cc}
\kappa & 2\kappa\varepsilon \\
0 & \kappa\varepsilon^2
\end{array}
\right).$$
For any $\kZ\in\Tor(\XX)$ and $\kB\in\VB(\XX)$, we have the vanishing $\Hom_{\XX}(\kZ,\kB)=0$. Using (\ref{E:ARDuality}), we get further vanishings
$$\Ext^1_{\XX}\bigl(\kS_i^{(j)},\kP\bigr)^\ast\cong\Hom_{\XX}\bigl(\kP, \tau(\kS_i^{(j)})\bigr)=\Hom_{\XX}\bigl(\kP,\kS_i^{(j-1)}\bigr)=0$$
for any $2\le j\le p_i-1$. As a consequence,
$$\bigl\langle[\kS_i^{(j)}],[\kP]\bigr\rangle=0\quad\mbox{\rm for all}\;1\le i\le t\;\mbox{\rm and}\;2\le j\le p_i-1.$$
Analogously, we get $\bigl\langle[\kS_i^{(j)}],[\overline\kP]\bigr\rangle=0$ for all $1\le i\le t$ and $2\le j\le p_i-1$. Since $\kP=\sF(\kL)$, $\sG(\kS_i^{(0)})=\kS_i$ and $(\sF, \sG)$ is an adjoint pair, we have the following isomorphisms of vector spaces over $\kk$:
$$\Ext^1_{\XX}\bigl(\kS_i^{(1)}, \kP\bigr)^\ast \cong  \Hom_{\XX}(\kP, \kS_i^{(0)})\cong\Hom_{\EE}(\kL, \kS_i).$$
In a similar vein, we have $\Ext^1_{\XX}\bigl(\kS_i^{(1)}, \overline{\kP}\bigr)^\ast \cong  
\Hom_{\EE}(\overline{\kL},\kS_i)$. From (\ref{E:DimensionD}) we conclude that 
$$\bigl\langle[\kS_i^{(1)}],[\kP]\bigr\rangle=-\kappa\varepsilon f_i\;\;\mbox{\rm and}\;\;\bigl\langle[\kS_i^{(1)}],[\overline\kP]\bigr\rangle=-\kappa\varepsilon^2 f_i\quad\mbox{\rm for all}\;1\le i\le t.$$
This concludes the proof.
\end{proof}

\begin{remark} In \cite[Theorem 3.12]{Burban}, it was shown that $D^b\bigl(\Coh(\XX)\bigr)$ admits a natural tilting complex $\kH^\bu$ such that $\widetilde\Sigma := \bigl(\End_{D^b(\XX)}(\kH^\bu)\bigr)^\circ$ is the \emph{squid algebra} from \cite{RingelCrawleyBoevey}.
There are further important finite-dimensional algebras, namely the \emph{canonical algebra} of Ringel $\Sigma$ and the \emph{Coxeter--Dynkin algebra} $\widehat{\Sigma}$, for which we have exact equivalences
\begin{equation}\label{E:CanonicalSquid}
 D^b\bigl(\widetilde\Sigma\mbox{--}\mathsf{mod}\bigr)\simeq D^b\bigl(\Sigma\mbox{--}\mathsf{mod}\bigr)\simeq D^b\bigl(\widehat\Sigma\mbox{--}\mathsf{mod}\bigr);
\end{equation}
see \cite{RingelCrawleyBoevey} and \cite{Perniok} for a detailed treatment.
\end{remark}

\begin{remark} In the case $\kk = \bar{\kk}$,  the theory of exceptional hereditary curves admits a significant simplification. First, we automatically have $X = \PP^1_\kk$. By a Theorem of Tsen,  $\mathsf{Br}\bigl(\kk(X)\bigr) = 0$; see \cite[Proposition 6.2.3 and Theorem 6.2.8]{GilleSzamuely}.  The tilted algebra $\Lambda$ given by (\ref{E:TameBimodule}) is the path algebra of the Kronecker quiver: $
\Lambda = \kk\bigl[\xymatrix{
\bu  \ar@/^/[r] \ar@/_/[r]  & \bu
}\bigr]
$. An exceptional curve $\EE$ in this case is a weighted projective line of Geigle and Lenzing \cite{GeigleLenzingWeightedCurves} (a connection between the original formalism of \cite{GeigleLenzingWeightedCurves} with the setting of non-commutative curves was elaborated by Chan and Ingalls in \cite{ChanIngalls}).
An exceptional hereditary curve $\XX = (\PP^1_\kk, 0, \rho)$ is determined (up to Morita equivalence) by its weight function $\rho:\PP^1_\kk\lar\NN$. Let $\mathfrak{E}_\rho = \left\{x_1, \dots, x_t\right\}$ be its special locus with $x_i = (\alpha_i : \beta_i)$ for $1 \le i \le t$. For the exact equivalence $\sT:D^b\bigl(\mathsf{Coh}(\EE)\bigr)\lar D^b\bigl(\Lambda\mbox{--}\mathsf{mod}\bigr)$ from Lemma \ref{L:derivedEquivalence}, we have
$\sT\bigl(\kS_{x_i}\bigr) \cong 
\xymatrix{
\kk  \ar@/^/[r]^{\alpha_i} \ar@/_/[r]_{\beta_i}  & \kk
}
$ and 
$\dd_i = \bigl(\mathsf{End}_{X}(\kS_{x_i})\bigr)^\circ \cong \kk$. The squid algebra $\Sigma$ is isomorphic to the path algebra of the following quiver 
$$
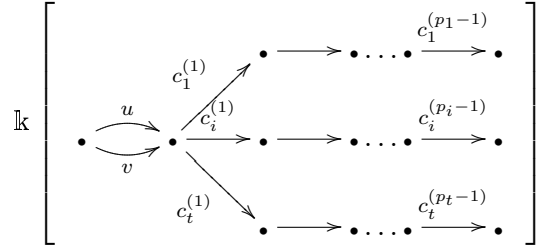
\kk\left[
\begin{array}{c}
\xymatrix{
                                    &     & \bu  \ar[r] & \bu \dots \bu \ar[r]^-{c_1^{(p_1-1)}} & \bu  \\
\bu  \ar@/^/[r]^{u} \ar@/_/[r]_{v}  & \bu \ar[r]^-{c_{i}^{(1)}} 
\ar[ru]^-{c_{1}^{(1)}} \ar[rd]_-{c_{t}^{(1)}} & \bu  \ar[r] & \bu \dots   \bu \ar[r]^-{c_i^{(p_i-1)}} & \bu \\
                                    &     & \bu \ar[r] & \bu  \dots  \bu \ar[r]^-{c_t^{(p_t-1)}} & \bu }
                                    \end{array}\right]
$$
subject to the relations $c_i^{(1)}(\beta_i u - \alpha_i v) = 0$ for all $1 \le i \le t$. 
\end{remark}

\smallskip
\noindent
We have constructed the standard exceptional sequence, which will be used in the upcoming definition of the reflection groups in Section \ref{S:ReflectionGroups}. We will now recall some key properties of exceptional sequences in $\Coh(\XX)$ from Kussin and Meltzer \cite{KussinMeltzer}. In particular, these results imply that any choice of a complete exceptional sequence defines the same associated reflection group.
\begin{lemma}\cite[Lemma 3.5]{KussinMeltzer}\label{L:ExtendabilityCohX}
    Any exceptional sequence $(E_1,\dots,E_p,F_r,\dots,F_r)$ in $\Coh(\XX)$ can be enlarged to a complete exceptional sequence $(E_1,\dots,E_p,H_1,\dots,H_q,F_r,\linebreak \dots,F_r)$.
\end{lemma}

\begin{lemma}\label{L:PerpendicularCalculus}
    For any complete exceptional sequence $(E_1,\dots,E_n)$ in $\Coh(\XX)$ and any $1\leq r\leq n$, we have
    $$\llangle E_1,\dots,E_r\rrangle=\llangle E_{r+1},\dots,E_n\rrangle^\perp.$$
\end{lemma}
\begin{proof}
    Let $E$ be an exceptional object in $\Coh(\XX)$. A key fact in \cite{KussinMeltzer} is that the (right) perpendicular subcategory $E^\perp$ is a category with Grothendieck group of rank $n-1$ in which every complete sequence is full, i.e. $E^\perp$ is isomorphic either to $H\mbox{--}\mathsf{mod}$ for some hereditary algebra $H$ or to $\Coh(\XX')$ for some exceptional hereditary curve $\XX'$ (or a product of them). In particular, the perpendicular calculus follows by induction.
\end{proof}

\begin{theorem}\cite[Theorem 1.1]{KussinMeltzer}\label{T:KussinMeltzer}
The braid group $B_n$ acts transitively  on the set of complete exceptional sequences in $\Coh(\XX)$, where $\XX$ is an exceptional hereditary curve. In particular, any complete exceptional sequence can be mutated to the standard exceptional sequence (\ref{E:StandardCollection}).
\end{theorem}

\noindent
Lemma \ref{L:ExtendabilityCohX} and Theorem \ref{T:KussinMeltzer} imply the following result. 

\begin{corollary}
Let $\XX$ be an exceptional hereditary curve. The standard exceptional sequence (\ref{E:StandardCollection}) defines a complete exceptional sequence of pseudo-roots in the Grothendieck group $\Gamma$ of $\Coh(\XX)$. Thus (\ref{E:StandardCollection}) gives rise to a set of real roots $\Phi \subset \Gamma$. Let $E \in \Coh(\XX)$ be an exceptional object. Then we have $[E] \in \Phi$, i.e.~the class of $E$ in $\Gamma$  is a real root. 
\end{corollary}

\section{Reflection groups of canonical type}\label{S:ReflectionGroups}
In this section we give a definition and discuss first properties of an interesting class of discrete groups which we call \emph{reflection groups of canonical type}. Moreover, we introduce two related groups: the quotient Coxeter group, which as the name suggests is a Coxeter group, and the hyperbolic extension, which is a central extension of the reflection group of canonical type.
\subsection{Introduction to reflection groups of canonical type}
We define \emph{reflection groups of canonical type} via some combinatorial data. These turn out to be precisely the reflection groups arising from categories of coherent sheaves $\Coh(\XX)$ on an exceptional hereditary curve $\XX$; see Proposition \ref{P:ReflectionGroupForCategory}. Moreover, we shall divide reflection groups of canonical type into three cases according to some underlying geometrical property.

\begin{definition}\label{D:Symbol}
Let $t \in \NN$, $\varepsilon \in \bigl\{1, 2\bigr\}$ and  $p_1, \dots, p_t \in \NN_{\ge 2}$. Next, let $d_1, \dots, d_t; f_1, \dots, f_t \in \NN$ be such that $f_i \mid d_i$ for all $1\le i \le t$. 
Following Lenzing \cite{LenzingKTheory}, we call the following table 
\begin{equation}\label{E:Symbol}
\sigma = \left( \begin{array}{ccc|c} 
            p_1 & \dots & p_t &\\
            d_1 & \dots & d_t & \varepsilon\\
            f_1 & \dots & f_t & \end{array} \right)
\end{equation}
a \emph{symbol}. For any $1 \le i \le t$, we put $e_i := \dfrac{d_i}{f_i}$ and set 
$n := \left(\sum\limits_{i=1}^t (p_i-1)\right) +2$. Moreover,
\begin{equation*}
\Omega := \bigl\{(i, j)\, \big| \, 1 \le i \le t, 1 \le j \le p_i-1 \bigr\} \quad \mbox{and} \quad
\overline\Omega = \Omega \sqcup \bigl\{0, 0^\ast\bigr\}.
\end{equation*}
\end{definition}

\begin{definition}\label{D:CanonicalLattice}
The symbol $\sigma$  determines a \emph{canonical bilinear lattice} $(\Gamma, K)$ defined as follows. Let $\Gamma$ be the free abelian group of rank $n$ generated by the tuple
\begin{equation}\label{E:SimpleRoots}
R:= \bigl(\alpha_{(t, p_t-1)}, \dots, \alpha_{(t, 1)},\dots,\alpha_{(1, p_1-1)}, \dots, \alpha_{(1,1)}, \alpha_0, \alpha_{0^\ast}\bigr)
\end{equation} of elements $\alpha_\omega\in\Gamma$ for $\omega \in \overline\Omega$.
Let $K$ be the bilinear form on $\Gamma$ given by the Gram matrix (\ref{E:GramMatrix}) with respect to $R$ for some $\kappa \in \NN$ such that $\frac{\kappa \varepsilon f_i}{e_i} \in \NN$ for all $1 \le i \le t$. It is easy to see that elements of $R$ are pseudo-roots forming  a complete exceptional sequence in $\Gamma$.  Let $B = K + K^t \in \mathsf{Mat}_{n}(\ZZ)$ be the symmetrization of $K$.

The complete exceptional sequence $R$ in $(\Gamma,K)$ defines a reflection group $W\subset\mathsf{O}(\Gamma,B)$ as well as the associated notions $c$, $S$, $T$ and $\Phi$ via Definitions \ref{D:Coxeter element} and \ref{D:Coxeter datum}. We call $W$ a \emph{reflection group of canonical type}.
\end{definition}

\begin{definition}\label{D:rankFunction}
Let $(\Gamma,K)$ be a canonical bilinear lattice with basis (\ref{E:SimpleRoots}). We define the \emph{rank function} on $\Gamma$ as the group homomorphism $\rk:\Gamma\to\ZZ$ with
$$\rk\left(\lambda_0\alpha_0+\lambda_{0^*}\alpha_{0^*}+\sum_{(i,j)\in\Omega}\lambda_{(i,j)}\alpha_{(i,j)}\right)=\lambda_0+\varepsilon\lambda_{0^*}.$$
\end{definition}

We will show that some information in the symbol $\sigma$ is superfluous if we are only interested in $W$, $c$, $S$ and $T$. To this end, let us introduce a reduced version of $\sigma$.
\begin{definition}
Let $\sigma$ be a symbol as in (\ref{E:Symbol}). We call the table
\begin{equation*}
\bar{\sigma} = \left( \begin{array}{ccc} 
            p_1 & \dots & p_t \\
            \varepsilon d_1 & \dots & \varepsilon d_t\\
             \end{array} \right)
\end{equation*}
the \emph{reduced symbol} of $\sigma$ and set
\begin{equation*}
\delta = \delta_\sigma:= 
\left(\sum\limits_{i = 1}^t \varepsilon d_i\Bigl(1- \frac{1}{p_i}\Bigr)\right) -2.
\end{equation*}
Note that $\delta$ actually depends solely on $\bar\sigma$. 
\end{definition}

We now show that the reduced symbol determines all relevant data.

\begin{proposition}\label{P:ReductionCaseEpsilon1} Let $\sigma$ and $\sigma'$ be two symbols such that $\bar\sigma = \bar\sigma'$. Then the corresponding data $(W, c,  S, T)$ and $(W', c',  S', T')$ can be naturally identified.
\end{proposition}
\begin{proof} Let $(\Gamma, K)$ be the bilinear lattice defined by $\sigma$. We normalize all elements of $R$ to be of length one with respect to the symmetrization $B$ of $K$. Then the corresponding Gram matrix is of the following shape:
\begin{equation}\label{E:MatrixRescaled}
 [B]=
\left[ \begin{array}{@{}c@{}|@{}c@{}|@{}c@{}|@{}c@{}|@{}c@{}|@{}c@{}}
\begin{array}{@{}c@{}@{}c@{}@{}c@{}@{}c@{}}
1 \hspace{2mm} & -\frac{1}{2} &  & \hspace{2mm} 0 \\
-\frac{1}{2} \hspace{2mm} & 1 & \ddots  &   \\
 & \ddots & \ddots & \hspace{2mm} -\frac{1}{2}  \\
0 \hspace{2mm} &  & -\frac{1}{2} & \hspace{2mm} 1
\end{array} & \cdots & 0 & 0 & \begin{array}{@{}c@{}} 0 \vspace{1mm} \\ \vdots \vspace{2mm} \\ 0 \vspace{1mm} \\ -\frac{\sqrt{\varepsilon d_t}}{2}
\end{array} & \begin{array}{@{}c@{}} 0 \vspace{1mm} \\ \vdots \vspace{2mm} \\ 0 \vspace{1mm} \\ -\frac{\sqrt{\varepsilon d_t}}{2}
\end{array} \\ \hline
\vdots & \ddots & \vdots & \vdots & \vdots & \vdots \\ \hline
0 & \cdots & \begin{array}{@{}c@{}@{}c@{}@{}c@{}@{}c@{}}
1 \hspace{2mm} & -\frac{1}{2} &  & \hspace{2mm} 0 \\
-\frac{1}{2} \hspace{2mm} & 1 & \ddots  &   \\
 & \ddots & \ddots & \hspace{2mm} -\frac{1}{2}  \\
0 \hspace{2mm} &  & -\frac{1}{2} & \hspace{2mm} 1
\end{array} & 0 & \begin{array}{@{}c@{}} 0 \vspace{1mm} \\ \vdots \vspace{2mm} \\ 0 \vspace{1mm} \\ -\frac{\sqrt{\varepsilon d_2}}{2}
\end{array} & \begin{array}{@{}c@{}} 0 \vspace{1mm} \\ \vdots \vspace{2mm} \\ 0 \vspace{1mm} \\ -\frac{\sqrt{\varepsilon d_2}}{2}
\end{array} \\ \hline
0 & 0 & \cdots & \begin{array}{@{}c@{}@{}c@{}@{}c@{}@{}c@{}}
1 \hspace{2mm} & -\frac{1}{2} &  & \hspace{2mm} 0 \\
-\frac{1}{2} \hspace{2mm} & 1 & \ddots  &   \\
 & \ddots & \ddots & \hspace{2mm} -\frac{1}{2}  \\
0 \hspace{2mm} &  & -\frac{1}{2} & \hspace{2mm} 1
\end{array} & \begin{array}{@{}c@{}} 0 \vspace{1mm} \\ \vdots \vspace{2mm} \\ 0 \vspace{1mm} \\ -\frac{\sqrt{\varepsilon d_1}}{2}
\end{array} & \begin{array}{@{}c@{}} 0 \vspace{1mm} \\ \vdots \vspace{2mm} \\ 0 \vspace{1mm} \\ -\frac{\sqrt{\varepsilon d_1}}{2}
\end{array} \\ \hline
\begin{array}{@{}c@{}@{}c@{}@{}c@{}@{}c@{}} \hspace{4mm} 0 \hspace{2mm} & \cdots \hspace{4mm} & 0 & \hspace{2mm} -\frac{\sqrt{\varepsilon d_t}}{2}
\end{array} & \cdots & \begin{array}{@{}c@{}@{}c@{}@{}c@{}@{}c@{}} \hspace{4mm} 0 \hspace{2mm} & \cdots \hspace{4mm} & 0 & \hspace{2mm} -\frac{\sqrt{\varepsilon d_2}}{2}
\end{array} & \begin{array}{@{}c@{}@{}c@{}@{}c@{}@{}c@{}} \hspace{4mm} 0 \hspace{2mm} & \cdots \hspace{4mm} & 0 & \hspace{2mm} -\frac{\sqrt{\varepsilon d_1}}{2}
\end{array} & 1 & 1\\ \hline
\begin{array}{@{}c@{}@{}c@{}@{}c@{}@{}c@{}} \hspace{4mm} 0 \hspace{2mm} & \cdots \hspace{4mm} & 0 & \hspace{2mm} -\frac{\sqrt{\varepsilon d_t}}{2}
\end{array} & \cdots & \begin{array}{@{}c@{}@{}c@{}@{}c@{}@{}c@{}} \hspace{4mm} 0 \hspace{2mm} & \cdots \hspace{4mm} & 0 & \hspace{2mm} -\frac{\sqrt{\varepsilon d_2}}{2}
\end{array} & \begin{array}{@{}c@{}@{}c@{}@{}c@{}@{}c@{}} \hspace{4mm} 0 \hspace{2mm} & \cdots \hspace{4mm} & 0 & \hspace{2mm} -\frac{\sqrt{\varepsilon d_1}}{2}
\end{array} & 1 & 1
	\end{array}\right]
\end{equation}
We see that $B$ only depends on the reduced symbol $\bar\sigma$. Let $(\Gamma', K')$ be the bilinear lattice defined by $\sigma'$. It is clear that we can naturally identify the real spans of $\Gamma$ and $\Gamma'$ with a common real vector space $V = \RR^n$  so that the rescaled bases $R$ and $R'$ get identified with the standard basis of $V$ and the induced pairing $V \times V \lar \RR$ is given by the matrix (\ref{E:MatrixRescaled}). Since for any non-isotropic vector $v \in V$ and $\lambda \in \RR^\ast$, we have $s_v = s_{\lambda v} \in \mathsf{O}(V, B)$, it follows that $(W, c,  S, T)$ and $(W', c',  S', T')$ coincide.
\end{proof}

\begin{proposition}\label{P:TypeDistinction} Let $\sigma$ be a symbol and $\Gamma$ the corresponding canonical bilinear lattice with symmetrized bilinear form $B$. The signature of $B$ is given by the expression
\[
		\left\lbrace{\begin{array}{@{}lr}
			(n-1,1,0)&\text{if }\delta<0,\\
			(n-2, 2,0)&\text{if }\delta=0,\\
			(n-2, 1,1)&\text{if }\delta>0.
		\end{array}}\right.
	\]
These cases are called domestic, tubular and wild, respectively.
\end{proposition}
\begin{proof}
    Let us consider $B:V\times V\lar\RR$, where $V$ is the vector space with basis $R$. Clearly, this does not change the signature of $B$. First, note that $\{\alpha_\omega|\omega\in\Omega\}$ forms a basis for $\overline V:=\llangle\alpha_\omega\,\,|\omega\in\Omega\rrangle\subset V$ and is the simple system for a root system of type $A_{p_t-1}\times\dots\times A_{p_1-1}$. Thus the bilinear form $B$ restricted to the $(n-2)$--dimensional subspace $\overline V\subset V$ is positive definite.
    
    Since $\llangle\alpha_0^*-\varepsilon\alpha_0\rrangle\subseteq\Rad_{B}$, we know that $\Rad_B$ is at least one--dimensional. Moreover, set $V_\circ:=\llangle\alpha_\omega|\omega\in\Omega\cup\{0\}\rrangle$. Then $\llangle\alpha_0^*-\varepsilon\alpha_0\rrangle\not\subseteq V_\circ$ implies that the signature of $B$ on $V$ is determined by the signature of $B\big|_{V_\circ\times V_\circ}$. But $B\big|_{\overline V\times\overline V}$ is positive definite for the one codimensional subspace $\overline V$ of $V_\circ$. Thus, the signature of $B|_{V_\circ\times V_\circ}$ is controlled by the determinant of $B\big|_{V_\circ\times V_\circ}$, i.e. the determinant of the matrix (\ref{E:MatrixRescaled}) without the last row and column, which we call $M$.
    
    Let us use our knowledge on the $(n-2)\times(n-2)$ minor of type $A_{p_t-1}\times\dots\times A_{p_1-1}$ to define a block diagonal matrix $N$. First, define the blocks $N^{(i)}$ to be upper triangular matrices of size $p_i-1$ with $N^{(i)}_{k,l}=\frac{k}{l}$ for $k\leq l$. Now $N$ consists of the blocks $N^{(t)},\dots,N^{(1)}$ and a 1 in the bottom right corner.
    
    The product $MN$ has the same determinant as $M$ and is almost lower triangular. We only need to eliminate the entries $-\frac{\sqrt{\varepsilon d_i}}{2}$. To this end, note that the diagonal entries on the row with $-\frac{\sqrt{\varepsilon d_i}}{2}$ are given by $1-\frac{1}{2}\frac{p_i-2}{p_i-1}=\frac{p_i}{2(p_i-1)}$. Finally, multiply $MN$ by the matrix that has ones on the diagonal, $\frac{\sqrt{\varepsilon d_i}}{2}\frac{2(p_i-1)}{p_i}$ on the appropriate entries and zeros elsewhere. We obtain a lower triangular $(n-1)\times(n-1)$ matrix with positive values on the first $n-2$ diagonal entries and
    \[1-\sum_{i=1}^t\frac{\varepsilon d_i}{2}\frac{p_i-1}{p_i}=-\frac{1}{2}\left(-2+\sum_{i=1}^t\varepsilon d_i\left(1-\frac{1}{p_i}\right)\right)=-\frac{1}{2}\delta\]
    on the last diagonal entry. Therefore, the determinant is positive if $\delta<0$, zero if $\delta=0$, and negative if $\delta>0$. This concludes the proof.
\end{proof}

\begin{proposition}\label{P:categoryInterpretation}
Let $\XX$ be an exceptional hereditary curve. Let $(E_1,\dots,E_n)$ be a complete exceptional sequence in $\Coh(\XX)$ and let $\tau$ be the Auslander--Reiten translate given by (\ref{E:ARTranslate}). Then the Grothendieck group $\Gamma = K_0(\XX)$ equipped with the Euler form $K:\Gamma \times \Gamma \lar \ZZ$ is a canonical bilinear lattice in the sense of Definition ~\ref{D:CanonicalLattice}. Moreover, $([E_1],\dots,[E_n])$ is a complete exceptional sequence in  $(\Gamma, K)$ and the automorphism induced by the Auslander--Reiten translate $\tau$ is the Coxeter element $c$ defined via ~(\ref{E:CoxeterDefinition}).
\end{proposition}
\begin{proof}
By Theorem \ref{T:ExceptionalMain}, we know that $\Gamma = K_0(\XX)$ is a free abelian group of rank $n = \left(\sum\limits_{i=1}^t (p_i-1)\right) +2$. The Euler form $K:\Gamma \times \Gamma \lar \ZZ$ is non-degenerate, hence $(\Gamma, K)$ is a bilinear lattice. It is clear from the definition of complete exceptional sequences $(E_1,\dots,E_n)$ in $\Coh(\XX)$ that $([E_1],\dots,[E_n])$ is a complete exceptional sequence in $(\Gamma,K)$. Next, we use the standard exceptional sequence (\ref{E:StandardCollection}). Put
$$
\alpha_{0} = [\kP], \alpha_{0^\ast} = [\overline\kP] \; \mbox{\rm and} \;
\alpha_{(i, j)} = \bigl[\kS_i^{(j)}\bigr] \; \mbox{\rm for} \; 1 \le i \le t 
\; \mbox{\rm and} \; 1 \le j \le p_i -1.
$$
By Theorem \ref{T:ExceptionalMain}, we know that
\begin{equation*}
R= \bigl(\alpha_{(t, p_t-1)}, \dots, \alpha_{(t, 1)}, \dots, \alpha_{(1, p_1-1)}, \dots, \alpha_{(1, 1)}, \alpha_0, \alpha_{0^\ast}\bigr)
\end{equation*}
is a complete exceptional sequence in $(\Gamma, K)$. Thus, $(\Gamma, K)$ is a canonical bilinear lattice in the sense of Definition \ref{D:CanonicalLattice}.

Let $c:\Gamma \lar \Gamma$ be the automorphism induced by the Auslander--Reiten translate $\tau$ given by (\ref{E:ARTranslate}). It follows from (\ref{E:ARDuality}) that $c$ is a Coxeter element of $(\Gamma, K)$, i.e.~it satisfies the condition (\ref{E:CoxeterDefinition}).
\end{proof}

\begin{remark}
    Let $\XX$ be an exceptional hereditary curve, and let $(\Gamma,K)$ be its Grothendieck group equipped with the Euler form. Because of the equivalence (\ref{E:CanonicalSquid}), $(\Gamma, K)$ is a canonical bilinear lattice in the sense of \cite{LenzingKTheory}.

    Moreover, Proposition \ref{P:categoryInterpretation} gives a categorical interpretation of the group homomorphism $\rk:\Gamma\to\ZZ$ introduced in Definition \ref{D:rankFunction}. It is induced by the rank function defined in (\ref{E:RankFunction}). The kernel
    $$\mathsf{Ker}(\mathsf{rk}:\Gamma\lar\ZZ)=\llangle \alpha_{0^*}-\varepsilon\alpha_0,\alpha_{(i,j)}\,|\,(i,j)\in\Omega\rrangle$$
    can be interpreted as the subgroup of $\Gamma$ generated by the classes of torsion coherent sheaves.
\end{remark}

The reflection groups of canonical type do not only arise from a representation theoretical point of view. The domestic and tubular types are already well studied, as they are affine Coxeter groups and elliptic Weyl groups, respectively. A more detailed discussion of these groups is provided in Appendix \ref{S:Comparison}.

\subsection{Quotient Coxeter group}\label{SubsectionQuotientCoxGrp}
Coxeter groups have been extensively studied in the literature. In particular, the Hurwitz transitivity of reduced reflection factorizations of Coxeter elements in Coxeter groups as well as some generalizations are very well understood. The fact that $W$ admits a quotient group, which is a Coxeter group, is particularly advantageous. It enables the application of established results from Coxeter group theory. In this section we investigate the aforementioned quotient group.

Let $\sigma$ be a symbol, $(\Gamma,K)$ the canonical bilinear lattice and $(W,S)$ the associated generalized Coxeter datum. Recall that for any $1 \le i \le t$, we have
\begin{equation*}
(\alpha_{(i, 1)}^\sharp,\alpha_0)=-e_i\quad\mbox{\rm and}\quad(\alpha_{(i, 1)},\alpha_0^\sharp)=-\varepsilon f_i,
\end{equation*}
whereas for $\alpha_{0^\ast}$ we have the formulas
\begin{equation*}
(\alpha_{(i, 1)}^\sharp,\alpha_{0^*})=-\varepsilon e_i\quad\mbox{\rm and}\quad(\alpha_{(i,1)},\alpha_{0^*}^\sharp)=-f_i.
\end{equation*}

Note that the element $a:=\alpha_{0^\ast}-\varepsilon\alpha_{0}$ belongs to the radical of the form $B$. Moreover, define
\begin{equation*}
\Gamma_\circ := \llangle \alpha_{(t, p_t-1)}, \dots, \alpha_{(t, 1)}, \dots, \alpha_{(1, p_1-1)}, \dots, \alpha_{(1, 1)}, \alpha_0\rrangle \subset \Gamma.
\end{equation*}
It is clear that $\Gamma = \Gamma_\circ \oplus \llangle a\rrangle$. Note that $\bigl(\Gamma_\circ, K\big|_{\Gamma_\circ \times \Gamma_\circ}\bigr)$ is again a bilinear lattice equipped with  a complete exceptional sequence 
\begin{equation*}
R_\circ = \bigl(\alpha_{(t, p_t-1)}, \dots, \alpha_{(t, 1)}, \dots, \alpha_{(1, p_1-1)}, \dots, \alpha_{(1, 1)}, \alpha_0 \bigr).
\end{equation*}
Thus we obtain the associated generalized Coxeter datum $\bigl(W_\circ, S_\circ\bigr)$ with its set of real roots $\Phi_\circ$.

\begin{lemma}\label{epi}
Through the decomposition of $\Gamma=\Gamma_\circ\oplus\llangle a\rrangle$, we have the following results.
\begin{enumerate}
\item[(a)] The projection $\Gamma_\circ\oplus\llangle a\rrangle\twoheadarrow \Gamma_\circ$ induces a split group epimorphism $p:W\twoheadarrow W_\circ$.
\item[(b)] For any $\beta\in\Phi$ there exist unique $d \in\{1,\varepsilon\}$, $\beta_\circ\in\Phi_\circ$ and $k\in\ZZ$ such that $\beta=d \beta_\circ+ka$.
\end{enumerate}
\end{lemma}
\begin{proof}
Let $\beta=\beta_\circ+ka,\gamma=\gamma_\circ+la\in\Gamma$ with $\beta_\circ,\gamma_\circ\in\Gamma_\circ$ and $k,l\in\ZZ$. Recall that $a\in\Rad(B)$.  We have
\begin{equation}\label{E:decompositionReflection}
    s_\beta(\gamma)=\gamma-(\gamma,\beta^\sharp)\beta=\gamma-(\gamma_\circ,\beta_\circ^\sharp)\beta=s_{\beta_\circ}(\gamma_\circ)+(l-(\gamma_\circ,\beta_\circ^\sharp)k)a
\end{equation}
This implies that the map $p:S\to S_\circ$ sending $s_\beta:\Gamma\to\Gamma$ to $s_{\beta_\circ}:\Gamma_\circ\to\Gamma_\circ$ extends to a morphism of groups. As $R$ maps surjectively onto $R_\circ$ (up to scalars), we get that $S$ maps surjectively onto the generating set $S_\circ$ of $W_\circ$ and hence that $p$ is an epimorphism. The inclusion $\Gamma_\circ\hookrightarrow\Gamma$ implies that $p$ splits.

Any $\beta\in\Phi$ decomposes uniquely into $\beta=\overline{\beta_\circ}+ka$ with $\overline{\beta_\circ}\in\Gamma_\circ$ and $k\in\ZZ$. We need to show that $\overline{\beta_\circ}= d \beta_\circ$ with $d \in\{1,\varepsilon\}$ and $\beta_\circ\in\Phi_\circ$. The fact that $\Phi_\circ$ is reduced (see Proposition \ref{P:HKLemmas}) gives the uniqueness. For any $\beta\in\Phi$, there exists $w\in W$, $\alpha\in R$ such that $\beta=w(\alpha)$. We prove the claim by induction on $l_S(w)$. For $l_S(w)=0$, note that we have
$$R=R_\circ\cup\{\varepsilon\alpha_0+a\}\subset(R_\circ+\ZZ a)\cup(\varepsilon R_\circ+\ZZ a)$$
where we consider $R$ and $R_\circ$ as sets. For the induction step apply (\ref{E:decompositionReflection}) using $R\subset(R_\circ+\ZZ a)\cup(\varepsilon R_\circ+\ZZ a)$.
\end{proof}

\begin{proposition}\label{P:CoxeterPositiveReduced}
    The datum $\bigl(W_\circ, S_\circ\bigr)$ is a Coxeter system. Moreover, the root system $\Phi_\circ$ is reduced and every root in $\Phi_\circ$ is either a non-negative or a non-positive linear combination with respect to the basis $R_\circ$.
\end{proposition}
\begin{proof}
    This follows from the fact that the bilinear lattice $\bigl(\Gamma_\circ, K\big|_{\Gamma_\circ \times \Gamma_\circ}\bigr)$ together with the complete exceptional sequence $R_\circ$ is a generalized Cartan lattice in the sense of \cite{HuberyKrause}. See Lemma 2.7 and the discussion after Lemma 3.1 therein.
\end{proof}

In Appendix \ref{S:Comparison}, we illustrate the \emph{Dynkin diagram} associated with the datum $(W_\circ,S_\circ)$. It turns out that the diagram is ``star-shaped'' (see Definition \ref{D:DynkinDiagramQuotient} and Figure \ref{rootDiagramQuotient}).

\subsection{Hyperbolic Extension}
In this section, we recall the notion of hyperbolic extensions following \cite{SaitoI}. We provide a detailed treatment on the linear algebra behind them in Appendix \ref{S:Appendix}.

Note that we can restrict our discussion on hyperbolic extensions to the tubular case as any hyperbolic extension $\widetilde{W}$ of a non-tubular reflection group of canonical type $W$ is just isomorphic to the group itself $W\cong\widetilde{W}$ by Lemma \ref{L:lowDimExt1}.

A careful definition of hyperbolic extensions following \cite{SaitoI} gives a choice between many hyperbolic extensions of a reflection group. However, for a tubular reflection group of canonical type $W$, any two hyperbolic extensions $\widetilde{W}, \widetilde{W}'$ of $W$ with respect to proper subspaces of $\Rad(B)$ will be isomorphic $\widetilde{W}\cong\widetilde{W}'$ by Corollary \ref{C:lowDimExt2}. Thus we shall simply speak of \emph{the} hyperbolic extension of $W$ and work with the following simple construction.

Let $(\Gamma, K)$ be a canonical bilinear lattice associated with a symbol 
\begin{equation*}
\sigma = \left( \begin{array}{ccc} 
            p_1 & \dots & p_t \\
            d_1 & \dots & d_t \\
            f_1 & \dots & f_t 
             \end{array} \right).
\end{equation*}
of tubular type such that $\varepsilon = 1$. We know that $\Gamma$ has a basis
\begin{equation}\label{E:distinguishedBasis2}
    \bigl(\alpha_{(t, p_t-1)}, \dots, \alpha_{(t, 1)}, \dots, \alpha_{(1, p_1-1)}, \dots, \alpha_{(1, 1)}, \alpha_0, a \bigr),
\end{equation}
i.e. $\Gamma$ decomposes as $\Gamma_\circ\oplus\llangle a\rrangle$. Let $V = \RR\otimes_{\ZZ} \Gamma$ be the real hull of $\Gamma$. Abusing the notation, we denote by $B:V \times V\lar\RR$ the extension of $B:\Gamma \times\Gamma\lar\ZZ$ to $V$.

\begin{definition}\label{hypConstruction}
    The \emph{hyperbolic extension} $(\widetilde{V},\widetilde{B},\iota)$ of $(V,B)$ is defined by
    \begin{itemize}
        \item $\widetilde{V}$ is the vector space spanned by
        $$
        \bigl(\alpha_{(t, p_t-1)}, \dots, \alpha_{(t, 1)}, \dots, \alpha_{(1, p_1-1)}, \dots, \alpha_{(1, 1)}, \alpha_0, a, a' \bigr).
        $$
        \item $\widetilde{B}$ is the symmetric bilinear form on $\widetilde{V}$ such that $\widetilde{B}\big|_{V \times V} = B$ as well as $\widetilde{B}(x,a')=0=\widetilde{B}(a',a')$ for any $x\in\Gamma_\circ$ and $\widetilde{B}(a,a')=1$.
        \item $\iota:V\to\widetilde{V}$ is the natural inclusion.
    \end{itemize}
\end{definition}

For any non-isotropic vector $v\in\widetilde{V}$, we denote by $\tilde{s}_v:\widetilde{V}\to\widetilde{V}$ its reflection in $\widetilde{V}$. Note that non-isotropic vectors $v\in{V}$ are non-isotropic in $\widetilde{V}$ and thus define a reflection $\tilde{s}_v:=\tilde{s}_{\iota(v)}$ in $\widetilde{V}$. In particular, we will again abbreviate $\tilde{s}_{\alpha_\omega}$ by $\tilde{s}_\omega$ for any $\omega\in\overline{\Omega}$.
\begin{definition}\label{D:hyperbolicExtensionReflGrpCanType}
    Let $(\Gamma, K)$ be a canonical bilinear lattice and $W$ the associated reflection group of canonical type. The \emph{hyperbolic extension} $\widetilde{W}$ is given by Definition \ref{D:Coxeter datum} for the simple roots
    \[\widetilde{R}:= \bigl(\alpha_{(t, p_t-1)}, \dots, \alpha_{(t, 1)},\dots,\alpha_{(1, p_1-1)}, \dots, \alpha_{(1, 1)}, \alpha_0, \alpha_{0^\ast}\bigr)\in\widetilde{V}^n.\]
    We denote by $\widetilde{S}$ and $\widetilde{T}$ the corresponding sets of (simple) reflections. In particular, the \emph{hyperbolic Coxeter element} is given by
        \begin{equation}\label{E:DefinitionHyperbolicCoxeter}
    \tilde{c}: = \tilde{s}_{(t, p_t-1)} \dots \tilde{s}_{(t, 1)} \dots \tilde{s}_{(1, p_1-1)} \dots \tilde{s}_{(1, 1)} \tilde{s}_{0} \tilde{s_{0^\ast}} \in \widetilde{W}.
    \end{equation}
\end{definition}

\begin{remark}
    Note that by Proposition \ref{P:HKLemmas} the Coxeter element $c$ in $W$ can also be written as a product of simple reflections.
    \begin{equation*}
    {c}: = {s}_{(t, p_t-1)} \dots {s}_{(t, 1)} \dots {s}_{(1, p_1-1)} \dots {s}_{(1, 1)} {s}_{0} {s_{0^\ast}} \in {W}.
    \end{equation*}
\end{remark}

Take any $\omega\in\overline{\Omega}$, $v\in V$. We have
$$\tilde{s}_\omega(v)=v-\frac{\widetilde{B}(v,\alpha_\omega)}{\widetilde{B}(\alpha_\omega,\alpha_\omega)}\alpha_\omega=v-\frac{{B}(v,\alpha_\omega)}{{B}(\alpha_\omega,\alpha_\omega)}\alpha_\omega=s_\omega(v).$$
Thus, the map sending $\tilde{s}_\omega$ to $s_\omega$ extends to an epimorphism $\pi:\widetilde{W}\to W$. Saito showed that this epimorphism as well as the hyperbolic Coxeter element are instrumental in the investigation of the structure of $\widetilde{W}$.
\begin{lemma}[\cite{SaitoI}, Lemma C]\label{hypSES}
    Let $W$ be a tubular reflection group of canonical type and $\widetilde{W}$ its hyperbolic extension. Then $\widetilde{W}$ is a central extension of $W$. More precisely, we have a short exact sequence
    \[1\to\langle\tilde{c}^p\rangle\xrightarrow{\xi}\widetilde{W}\xrightarrow{\pi}W\to1,\]
    where $\xi$ is the inclusion of $\langle\tilde{c}^p\rangle\subset\widetilde{W}$ into $\widetilde{W}$ and $p=\lcm(p_1,\dots,p_t)=\max(p_1,\dots,p_t)$ is the order of $c\in W$.
\end{lemma}
Note that Saito uses a different realization of the hyperbolic extension compared to Definition \ref{hypConstruction}. The result still holds true by Corollary \ref{C:lowDimExt2}. However, we cannot use his explicit formula for $\tilde{c}$ or $\tilde{c}^p$ and have to compute ourselves. A straightforward computation gives the following result.
\begin{lemma}\label{L:UsefulFormulaReflection}
For any $\beta \in \Phi_\circ$ and $k,k' \in \ZZ$ we have:
\begin{equation}
\tilde{s}_{\beta + ka} \tilde{s}_{\beta +k'a}(a') = 
a' + \dfrac{2(k'-k)}{(\beta, \beta)} \beta -
\dfrac{2(k'-k)^2}{(\beta, \beta)}a. 
\end{equation}
In particular, for $\beta = \alpha_0$, $k = 0$ and $k' = 1$ we obtain
\begin{equation}\label{E:Useful}
\tilde{s}_{0} \tilde{s}_{0^\ast}(a') = a' + \dfrac{2}{(\alpha_0, \alpha_0)}(\alpha_0-a). 
\end{equation}
\end{lemma}
\begin{proposition}\label{P:ActionCoxeter}  The following formula is true:
\begin{equation}\label{E:ActionCoxeter}
\tilde{c}(a') = a' + \frac{2}{(\alpha_0, \alpha_0)}\left(\alpha_0-a + \sum\limits_{(i,j)\in\Omega} e_i\alpha_{(i,j)}\right).
\end{equation}
\end{proposition}
\begin{proof} Using (\ref{E:Useful}) we obtain: 
\begin{align*}
\tilde{c}(a') &= \tilde{s}_{(t, p_t-1)} \dots \tilde{s}_{(t, 1)} \dots \tilde{s}_{(1, p_1-1)} \dots \tilde{s}_{(1, 1)} \bigl(\tilde{s}_{0} \tilde{s}_{0^\ast}(a')\bigr) \\
&= a' + \frac{2}{(\alpha_0, \alpha_0)}
\bigl(s_{(t, p_t-1)} \dots s_{(t,1)} \dots s_{(1, p_1-1)} \dots s_{(1,1)}(\alpha_0)-a\bigr)
\end{align*}
since $\tilde{s}_{(i,j)}\big|_{V} = s_{(i,j)}$, $\tilde{s}_{(i,j)}(a') = a'$  and ${s}_{(i,j)}(a) = a$ for all $(i, j) \in \Omega$. A  straightforward computation shows that 
$$
s_{(t, p_t-1)} \dots s_{(t,1)} \dots s_{(1, p_1-1)} \dots s_{(1,1)}(\alpha_0) = \alpha_0 + \sum\limits_{(i,j)\in\Omega} e_i\alpha_{(i,j)}
$$
which implies the claim.
\end{proof}

\begin{proposition}\label{P:JNFCoxeter} The Jordan normal form of $\tilde{c}$ is
\begin{equation}\label{E:JNFCoxeter}
\diag(\xi_t,\dots,\xi_t^{p_t-1},\dots,\xi_1,\dots,\xi_1^{p_1-1},1)\oplus\begin{pmatrix}1&1\\ 0 &1\end{pmatrix}
\end{equation}
where $\xi_l = \exp{\dfrac{2\pi i}{p_l}}$ for all $1\le l \le t$. In particular, $\mathsf{dim}_{\RR}\bigl(\mathsf{Fix}(\tilde{c})\bigr) = 2$.
\end{proposition}
\begin{proof} It is clear that $V$ is a $\tilde{c}$--invariant subspace of $\widetilde{V}$ and $\tilde{c}\big|_{V} = c$. Recall the rank function defined in Definition \ref{D:rankFunction}. A direct computation shows that
$$\left\{v\in V \,\big|\, \mathsf{rk}(v) = 0 \right\}=\llangle a, \alpha_{(i,j)}|(i,j)\in\Omega\rrangle$$
is a $c$--invariant subspace of $V$ and thus a $\tilde{c}$--invariant subspace of $V\subset\widetilde{V}$. By \cite[Proposition 7.8]{LenzingKTheory}, the characteristic polynomial of $c$ is given by the formula
$$\chi_c(x) = (x-1)^2 \prod\limits_{l=1}^t \dfrac{x^{p_l}-1}{x-1}.$$
Furthermore, $c^p = \mathbbm{1}$, where $p=\lcm(p_1, \dots, p_t)$ as in Lemma \ref{hypSES}. This implies that $c$ is diagonalizable. Next, $\mathsf{Rad}(B)\subset\Gamma$ is fixed pointwise by $c$. Its real (or complex) hull is generated by the elements $a=\alpha_{0^*}-\alpha_0$ and
\begin{equation}\label{E:radElmtB}
    b=\alpha_0+\sum_{(i,j)\in\Omega}\frac{p_i-j}{p_i}e_i\alpha_{(i,j)}.
\end{equation}
The Jordan normal form of $c$ is given by the matrix
$$\diag(\xi_t,\dots,\xi_t^{p_t-1},\dots,\xi_1,\dots,\xi_1^{p_1-1},1,1)$$
where the last eigenvalue corresponds to the eigenvector $v_n:=b$ and the second-to-last to $v_{n-1}:=a$. Hence, the span of the first $n-1$ eigenvectors $v_1,\dots,v_{n-1}$ of $c$ is the vector space $\left\{v\in V_{\CC} \,\big|\, \mathsf{rk}(v) = 0 \right\}$. The matrix of $\tilde{c}$ in the basis $(v_1, \dots, v_n, a')$ of $\widetilde{V}_\CC$ has the following form:
$$
\left[
\begin{array}{ccccccccccc}
\xi_1 &      &  & & & & &  & & &  \ast\\
     & \ddots &      &  & & & & & & & \vdots\\
     & &  \xi_1^{p_1-1} &  & & & & & & & \ast\\
    &   &  &    \ddots & & & & & & & \vdots\\
     &   &  &    & &  \xi_t & & & & & \ast\\ 
&   &  &    & &   &   \ddots & & & & \vdots\\
&   &  &    & &   &   &    \xi_t^{p_t-1}  & & & \ast\\ 
&   &  &    & &   &   &      &    1 & & \ast\\
&   &  &    & &   &   &      &     & 1 & \lambda\\
&   &  &    & &   &   &      &     &  & 1\\
\end{array}
\right]
$$
for some $\lambda \in \CC$. By (\ref{E:ActionCoxeter}) we have
$$
\tilde{c}(a')-a' = \frac{2}{(\alpha_0, \alpha_0)}\left(\alpha_0-a + \sum\limits_{(i,j)\in\Omega}e_i\alpha_{(i,j)}\right) \in V.
$$
It follows that $\mathsf{rk}(\tilde{c}(a')-a') > 0$ implying that $\lambda \ne 0$. Hence, the Jordan normal form of $\tilde{c}$ is given by (\ref{E:JNFCoxeter}), as asserted. 
\end{proof}

\begin{corollary}\label{C:reflectionLengthHyp}
We have: $\ell_{\widetilde{T}}(\tilde{c}) = n$.
\end{corollary}
\begin{proof} By Proposition \ref{P:JNFCoxeter}, the fixed space of $\tilde{c}$ has dimension $\mathsf{dim}_{\RR}\bigl(\mathsf{Fix}(\tilde{c})\bigr) = 2$. Now Lemma \ref{L:estimatelength} implies $\ell_{\widetilde{T}}(\tilde{c}) \ge (n+1)-2 = n-1$. On the other hand, Proposition \ref{P:EstimatesLengthCoxeter}(b) implies that $\ell_{\widetilde{T}}(\tilde{c}) \equiv n \; \mathsf{mod}\; 2$. Combining these two facts, we get the result.
\end{proof}

\begin{proposition}\label{P:reflectionlengthCox} Let $(\Gamma, K)$ be a canonical bilinear lattice of tubular type and $c \in W$ be the corresponding Coxeter element. Then we have $\ell_T(c) = n$. 
\end{proposition}
\begin{proof} It is clear that $\ell_{T}(c) \le n$. Moreover, Proposition \ref{P:EstimatesLengthCoxeter} implies that $\ell_T(c) \ge n-2$ and $\ell_T({c}) \equiv n \; \mathsf{mod}\; 2$.

Suppose that $\ell_T(c) =  n-2$. Then $c = t_1 \dots t_{n-2}$, where $t_i = s_{\gamma_i}$ with some $\gamma_i \in \Phi$ for all $1 \le i \le n-2$. Consider the element  $\tilde{d} := \tilde{t}_1 \dots \tilde{t}_{n-2} \in \widetilde{W}$, where $\tilde{t}_i = \tilde{s}_{\gamma_i}$, i.e. $\pi(\tilde{t}_i)=t_i$ under the group homomorphism $\pi:\widetilde{W}\twoheadarrow W$ from Lemma \ref{hypSES}. Recall that the kernel of $\pi$ is a free cyclic group generated by the central element $\tilde{c}^p$. Since $\pi(\tilde{d})=c=\pi(\tilde{c})$, we conclude that $\tilde{d} = \tilde{c}^{1+pk}$ for some $k\in \ZZ$. Lemma \ref{L:estimatelength} and Proposition \ref{P:JNFCoxeter} imply that
$$
n-2 \ge \ell_{\widetilde{T}}(\tilde{d}) \ge \mathsf{cod}\bigl(\mathsf{Fix}(\tilde{d})\bigr) = (n+1)-2 = n-1,
$$
giving a contradiction. Hence, $\ell_T(c) = n$, as asserted.
\end{proof}

\section{Hurwitz transitivity}\label{S:HurwitzTransitivity}
Our main goal is to show the transitivity of the Hurwitz action on the set $\Red_{T}(c)$ of reduced reflection factorizations of the Coxeter element corresponding to a symbol $\sigma$. According to Proposition \ref{P:ReductionCaseEpsilon1}, we may, without loss of generality, assume $\varepsilon = 1$, i.e.
\begin{equation*}
\sigma = \left( \begin{array}{ccc} 
            p_1 & \dots & p_t \\
            d_1 & \dots & d_t \\
            f_1 & \dots & f_t 
             \end{array} \right).
\end{equation*}

In what follows, we investigate the associated bilinear lattice $(\Gamma, K)$ and the corresponding set of real  roots $\Phi\cong\widetilde\Phi$.
\subsection{Computations in the root system}
Using the star-like structure associated with $\Phi_\circ$ (see Figure \ref{rootDiagramQuotient} of Appendix \ref{S:Comparison}), we are able to give some nice results on the structure of $\Phi$. These results are then used for explicit computations regarding a particular form of factorizations in $\Red_T(c)$ that will arise in the proof of Theorem \ref{newMainTransitivity}.

\begin{definition} Given $\beta, \gamma \in \Phi$, we say that $\beta \sim \gamma$ if there exists $w \in W$ such that $\gamma = w(\beta)$.
\end{definition}

\begin{remark}
In particular, for any $\beta \in \Phi$, we have $\beta \sim -\beta$ and there exists a (not necessarily unique) simple root $\alpha = \alpha_{\omega}$ (with $\omega \in \overline\Omega$) such that $\beta \sim \alpha$. Moreover, as $\alpha_{(i,j)}\sim\alpha_{(i,1)}$ for any $(i,j)\in\Omega$, we even have: for any $\beta \in \Phi$ there exists a (not necessarily unique) simple root $\alpha = \alpha_{\omega}$ with $\omega \in\{0,0^*,(i,1)|1\leq i\leq t\}$ such that $\beta \sim \alpha$.
\end{remark}

We will now examine $\Phi$ through the relation $\sim$.

\begin{proposition}\label{P:Divisibility1} Let $\beta = \lambda_{0} \alpha_0 + \lambda_{0^\ast} \alpha_{0^\ast} + \sum\limits_{(i, j) \in \Omega} \lambda_{(i, j)} \alpha_{(i, j)} \in \Phi$ be  such that 
$\beta \sim \alpha_{(i,j)}$ for some $(i, j) \in \Omega$. Then the following statements hold.
\begin{enumerate}
\item[(a)] We have $f_i \, \big|\, \lambda_0$ and $f_i \,\big|\, \lambda_{0^\ast}$.
\item[(b)] For any $(k, l) \in \Omega$ with $k \ne i$ we have $f_i e_k \, \big|\,  \lambda_{(k, l)}$.
\end{enumerate}
\end{proposition}
\begin{proof}
By definition, there exists $w \in W$ such that $\beta = w\bigl(\alpha_{(i, j)}\bigr)$. We prove both divisibility results by induction on the length $\ell_S(w)$. The basis of induction for $w \in W$ with $\ell_S(w) = 0$ is trivial.

To prove the induction step, suppose that $\beta = s_\gamma(\beta')=\beta'-(\beta',\gamma^\sharp)\gamma$, where $\gamma$ is a simple root and 
$\beta' = \lambda'_{0} \alpha_0 + \lambda'_{0^\ast} \alpha_{0^\ast} + \sum\limits_{(i, j) \in \Omega} \lambda'_{(i, j)} \alpha_{(i, j)} \in \Phi$ is such that 
\begin{enumerate}
\item[(a)] $f_i \, \big|\, \lambda'_0$, $f_i \,\big|\, \lambda'_{0^\ast}$ and
\item[(b)] for any $(k, l) \in \Omega$ with $k \ne i$ we have $f_i e_k \, \big|\,  \lambda'_{(k, l)}$.
\end{enumerate}
Note that all coefficients $\lambda_\omega$ and $\lambda'_\omega$ of $\beta$ and $\beta'$ are the same except for $\omega\in\overline\Omega$ with $\gamma=\alpha_\omega$. We go through all the possible cases.
\begin{itemize}
\item If $\gamma = \alpha_{(i, p)}$ for some $1\le p \le p_i-1$ then all the ``relevant'' coefficients of 
$\beta$ and $\beta'$ are the same. Hence, the statement of proposition is true. 
\item Suppose that $\gamma = \alpha_0$. Then
\begin{equation*}
\lambda'_0 - \lambda_0 =  \bigl(\beta', \alpha_0^\sharp\bigr) = 
2\lambda'_0 +2\lambda'_{0^\ast}+ \sum\limits_{k = 1}^t \lambda'_{(k, 1)} \bigl(\alpha_{(k, 1)}, \alpha_0^\sharp\bigr).
\end{equation*}
By the induction hypothesis, $\lambda'_0$ and $\lambda'_{0^*}$ are divisible by $f_i$. For any $1 \le k \le t$ with $k\neq i$ the even stronger condition $f_i e_k\, \big| \, \lambda'_{(k,l)}$ is fulfilled. Finally, $\bigl(\alpha_{(i, 1)}, \alpha_0^\sharp\bigr)=-f_i$ is divisible by $f_i$. It follows that $f_i \, \big| \, \lambda_0$, as asserted.
\item In the case $\gamma = \alpha_{0^\ast}$ we proceed as in the previous case. 
\item Finally, suppose that $\gamma = \alpha_{(k, p)}$ with $k \ne i$ and $1 \le p \le p_k-1$. Then we have
\begin{equation*}
\lambda'_{(k,p)} - \lambda_{(k, p)} =  \bigl(\beta', \alpha_{(k, p)}^\sharp\bigr) = 
\lambda'_0 \bigl(\alpha_0, \alpha_{(k, p)}^\sharp\bigr) + \lambda'_{0^\ast} \bigl(\alpha_{0^\ast}, \alpha_{(k, p)}^\sharp\bigr) + \sum\limits_{l = 1}^{p_k-1} \lambda'_{(k, l)} \bigl(\alpha_{(k, l)}, \alpha_{(k, p)}^\sharp\bigr).
\end{equation*}
By the induction hypothesis, $\lambda'_{(k, l)}$ is divisible by $f_i e_k$ for any $1 \le l \le p_k-1$. Moreover, $\lambda'_0$ and $\lambda'_{0^\ast}$ are both divisible by $f_i$, whereas
$$
\bigl(\alpha_0, \alpha_{(k, p)}^\sharp\bigr) =  \bigl(\alpha_{0^\ast}, \alpha_{(k, p)}^\sharp\bigr) = 
\left\{
\begin{array}{ccl}
-e_k & \mbox{if} & p = 1 \\
0 & \mbox{if} & 2 \le p \le p_k-1.
\end{array}
\right.
$$
It follows that $f_i e_k \, \big| \, \lambda_{(k, p)}$, as asserted.
\end{itemize}
This concludes the proof of the induction step and implies the statement.
\end{proof}

\begin{corollary}\label{C:Divisibility2} Let $\beta = \lambda_{0} \alpha_0 + \nu_0 a  + \sum\limits_{(i, j) \in \Omega} \lambda_{(i, j)} \alpha_{(i, j)} \in \Phi$ be  such that 
$\beta \sim \alpha_{(i,j)}$ for some $(i, j) \in \Omega$. Then the following statements hold.
\begin{enumerate}
\item[(a)] We have $f_i \, \big|\, \lambda_0$ and $f_i \,\big|\, \nu_0$.
\item[(b)] For any $(k, l) \in \Omega$ with $k \ne i$ we have $f_i e_k \, \big|\,  \lambda_{(k, l)}$.
\end{enumerate}
\end{corollary}

\begin{proposition}\label{P:Divisibility2} Let $\beta = \lambda_{0} \alpha_0 + \nu_0 a  + \sum\limits_{(i, j) \in \Omega} \lambda_{(i, j)} \alpha_{(i, j)} \in \Phi$ be  such that 
$\beta \sim \alpha_0$  or $\beta \sim \alpha_{0^\ast}$. Then we have $e_i \, \big| \, \lambda_{(i, j)}$ for all $(i, j) \in \Omega$. 
\end{proposition}

\begin{proof} We give a proof in the case $\beta \sim \alpha_0$. There exists $w \in W$ such that $\beta = w(\alpha_0)$. We prove this result by induction on $\ell_S(w)$. The basis of induction for $w \in W$ with $\ell_S(w) = 0$ is trivial.

To prove the induction step, suppose that $\beta = s_\gamma(\beta')$, where $\gamma$ is a simple root and 
$\beta' = \lambda'_{0} \alpha_0 + \nu'_{0} a + \sum\limits_{(i, j) \in \Omega} \lambda'_{(i, j)} \alpha_{(i, j)} \in \Phi$ is such that $e_i \, \big| \, \lambda'_{(i, j)}$ for all $(i, j) \in \Omega$. 
\begin{itemize}
\item If $\gamma = \alpha_0$ or $\gamma = \alpha_{0^\ast}$, then $\beta - \beta' \in \llangle \alpha_0, a\rrangle_{\ZZ}$ and the statement follows from the statement follows from the induction hypothesis.
\item If $\gamma = \alpha_{(i, j)}$ for some $(i, j) \in \Omega$, then $\beta' - \beta = \bigl(\beta',\alpha_{(i, j)}^\sharp\bigr)\alpha_{(i, j)}$. We compute
\[\bigl(\beta',\alpha_{(i, j)}^\sharp\bigr) = \lambda'_0 \bigl(\alpha_0,\alpha_{(i,j)}^\sharp\bigr) + \sum\limits_{l = 1}^{p_i-1} \lambda'_{(i, l)} \bigl(\alpha_{(i,l)},\alpha_{(i, j)}^\sharp\bigr).\]
By the induction hypothesis, we have $e_i \, \big| \, \lambda'_{(i, l)}$ for all $1 \le l \le p_i-1$. Since
$$
\bigl(\alpha_0,\alpha_{(i, j)}^\sharp\bigr) =
\left\{
\begin{array}{ccl}
-e_i & \mbox{if} & j = 1 \\
0 & \mbox{if} & 2 \le j \le p_i-1.
\end{array}
\right.
$$
we conclude that $e_i \, \big| \, \lambda_{(i, j)}$, as asserted.
\end{itemize}
This concludes the proof of the induction step and implies the statement.
\end{proof}

\begin{lemma}\label{L:MissingLemma} Let $(i, j) \in \Omega$ and $m \in \ZZ$ be such that $\alpha_{(i,j)} + ma \in \Phi$. Then we have $\alpha_{(i,j)} + ma  \sim \alpha_{(i,j)}$. 
\end{lemma}
\begin{proof}
By the assumption that $\alpha_{(i,j)}+ma\in\Phi$, we know that there exists $\omega\in\overline\Omega$ such that $\alpha_{(i,j)}+ma\sim\alpha_\omega$. If $\omega=(k,l)$ with $k=i$, then we are done. If $\omega=(k,l)$ with $k\neq i$ or $\omega\in\{0,0^*\}$, we only need to show that $\alpha_{(k,1)}\sim\alpha_{(i,1)}$, $\alpha_0\sim\alpha_{(i,1)}$ or $\alpha_{0^*}\sim\alpha_{(i,1)}$ respectively.

First, let $\omega=0$. By Proposition \ref{P:Divisibility2}, we know that $e_i$ divides $\lambda_{(i,j)} = 1$, hence $e_i=1$. Since the norms $\Vert \alpha_0\Vert= \Vert\alpha_{(i,j)}+ma\Vert = \Vert\alpha_{(i,j)}\Vert = \Vert\alpha_{(i,1)}\Vert $, we get $f_i = -(\alpha_{(i,1)},\alpha_0^\sharp) = -(\alpha_{(i,1)}^\sharp,\alpha_0) = e_i$.
 Therefore, $s_{0}s_{(i,1)}(\alpha_0)=\alpha_{(i,1)}.$
The case $\omega=0^*$ is analogous. 

Now, let $\omega=(k,1)$ with $k\neq i$. According to  Corollary \ref{C:Divisibility2}, $f_k e_i$ divides $\lambda_{(i,j)} = 1$, hence $f_k=e_i=1$. Next,  $\Vert \alpha_{(k,1)}\Vert = \Vert \alpha_{(i,j)}+ma\Vert = \Vert\alpha_{(i,j)}\Vert = \Vert\alpha_{(i,1)}\Vert$. Since $f_k = -(\alpha_{(k,1)}, \alpha_0^\sharp) = 1$ and $e_k = - (\alpha_{(k,1)}^\sharp, \alpha_0) = \dfrac{(\alpha_0,\alpha_0)}{(\alpha_{(k,1)},\alpha_{(k,1)})} f_k = \dfrac{(\alpha_0,\alpha_0)}{(\alpha_{(k,1)},\alpha_{(k,1)})}$ is an integer, we infer that $\Vert \alpha_{(k,1)}\Vert \le \Vert \alpha_{0}\Vert$. 

 Similarly, $e_i=1$ implies that $\Vert \alpha_0\Vert \leq \Vert\alpha_{(i,1)}\Vert$. 
 Since, $\Vert \alpha_{(k,1)}\Vert \le \Vert\alpha_0\Vert \le \Vert \alpha_{(i,1)}\Vert$ and $\Vert \alpha_{(k,1)}\Vert  = \Vert\alpha_{(i,1)}\Vert$, we conclude that 
 $\Vert \alpha_{(k,1)}\Vert = \Vert \alpha_{0}\Vert = \Vert\alpha_{(i,1)}\Vert$, hence 
  $f_i=e_k=1$. Therefore, we have
\[s_{0}s_{(i,1)}s_{(k,1)}s_{0}(\alpha_{(k,1)})=s_{0}s_{(i,1)}(\alpha_0)=\alpha_{(i,1)},\]
implying the statement.
\end{proof}

\begin{lemma}\label{L:ModificationOfEquation}
Let $\beta = \lambda_0 \alpha_0 + \sum\limits_{(i,j) \in \Omega} \lambda_{(i,j)} \alpha_{(i,j)} \in \Phi_\circ$ and $k,k'\in\ZZ$ be such that $\beta+ka,\beta+k'a\in\Phi$ and $\dfrac{(\alpha_0, \alpha_0)}{(\beta, \beta)}\lambda_0(k'-k)=1$. Then we have $\Vert\beta\Vert =\Vert\alpha_0\Vert$ and $\lambda_0=k'-k=\pm1$.
\end{lemma}
\begin{proof}
As $\lambda_0,k,k'\in\ZZ$, we only need to show that $\Vert\beta\Vert=\Vert\alpha_0\Vert$. Thus, if $\beta \sim \alpha_0$, we are done immediately.

Suppose that $\beta \sim \alpha_{(i,j)}$ for some $(i,j)\in\Omega$. Then $||\beta||=||\alpha_{(i,j)}||=||\alpha_{(i,1)}||$ and we have $\beta+ka \sim \alpha_{(i,j)}+ka$ and $\beta+k'a \sim \alpha_{(i,j)}+k'a$. By Lemma \ref{L:MissingLemma}, we have $\beta+ka,\beta+k'a \sim \alpha_{(i,j)}$. Corollary \ref{C:Divisibility2} implies that $f_i \,\big|\, \lambda_0$, $f_i \,\big|\, k$ and $f_i \,\big|\, k'$. As a consequence, $f_i^2 \,\big|\, \lambda_0 (k'-k)$. Next, we have
$$
\dfrac{(\alpha_0, \alpha_0)}{(\beta, \beta)}f_i^2 = 
- \dfrac{(\alpha_0, \alpha_0)}{(\beta, \beta)}\bigl(\alpha_0^\sharp, \alpha_{(i,1)}\bigr) f_i = - \bigl(\alpha_0, \alpha_{(i,1)}^\sharp\bigr) f_i = d_i.
$$
It follows that  $1= \dfrac{(\alpha_0, \alpha_0)}{(\beta, \beta)} \lambda_0(k'-k)$ is divisible by $d_i$. Hence $d_i = 1$ and, as a consequence, $e_i = 1 = f_i$. Finally, $e_i = \dfrac{(\alpha_0,\alpha_o)}{(\alpha_{(i,1)},\alpha_{(i,1)})} f_i$ implies $||\alpha_0||=||\alpha_{(i,1)}||=||\beta||$.
\end{proof}

\begin{proposition}\label{P:KillingTheArms} Let $\beta = \alpha_0 + \sum\limits_{(k, l) \in \Omega} \lambda_{(k, l)} \alpha_{(k, l)}
\in \Phi_\circ$ be such that $\Vert \beta\Vert  = \Vert \alpha_0\Vert$. Then there exists
$w \in \llangle s_{(i, j)} \, \big|\, (i, j) \in \Omega\rrangle$ such that $w(\beta) = \alpha_0$. 
\end{proposition}

\begin{proof} Our arguments are  inspired by the proof of \cite[Lemma 7.2]{BaumeisterWegenerYahiateneII}. We first show the following preparatory statement.

\smallskip
\noindent
\underline{Claim}. Let $\beta = \alpha_0 + \sum\limits_{l= 1}^m \lambda_{(1, l)} \alpha_{(1, l)} + \beta' \in \Phi_\circ$ be such that $\beta' = \sum\limits_{k = 2}^t \sum\limits_{l = 1}^{p_k-1} \lambda_{(k,l)} \alpha_{(k,l)}$, $\Vert\beta\Vert  = \Vert\alpha_0\Vert$ and $\lambda_{(1, m)} \ne 0$ for some $1 \le m \le p_1-1$. Then we have $\lambda_{(1,1)} \ne 0$. 

\smallskip
\noindent
To show this statement, we put $t:= s_{(1,2)} \dots s_{(1, m)}$. Note that $\alpha_0+\beta'$ is perpendicular to $\alpha_{(1,j)}$, $2\leq j\leq p_1-1$ (with respect to $B$). Thus, $t(\alpha_0+\beta')=\alpha_0+\beta'$. Further, a direct computation shows
\[t\left(\sum_{j=1}^m\lambda_{(1,j)}\alpha_{(1,j)}\right) = \lambda_{(1, 1)} \alpha_{(1,1)} + \sum_{j=2}^m(\lambda_{(1,j-1)}-\lambda_{(1,m)})\alpha_{(1,j)}.\]
Therefore, we conclude
\[t(\beta)=\alpha_0+\lambda_{(1, 1)} \alpha_{(1,1)} + \sum_{j=2}^m(\lambda_{(1,j-1)}-\lambda_{(1,m)})\alpha_{(1,j)}+\beta'\in\Phi_\circ.\]
If $\lambda_{(1,1)} = 0$ then $t(\beta)$ is a root in $\Phi_\circ$ which is neither positive nor negative. This contradiction proves the claim.

\smallskip
\noindent
Now assume that $\beta = \alpha_0 + \sum\limits_{l= 1}^m \lambda_{(1, l)} \alpha_{(1, l)} + \beta'$, where 
$\beta' = \sum\limits_{k = 2}^t \sum\limits_{l = 1}^{p_k-1} \lambda_{(k,l)} \alpha_{(k,l)}$ and $\lambda_{(1, m)} \ne 0$. 
To prove the proposition, it is sufficient to show that there exists $w \in \llangle s_{(1, j)} \, \big|\, (1,j)\in\Omega\rrangle$ such that $w(\beta) = \alpha_0 + \beta'$. 

\smallskip
\noindent
We already know that $\lambda_{(1, 1)} \ne 0$. For $s := s_{(1,1)} s_{(1,2)} \dots s_{(1, m)} = s_{(1,1)} t$ we have
\begin{align*}
s(\beta) &= s_{(1,1)}\left(\alpha_0+\lambda_{(1, 1)} \alpha_{(1,1)} + \sum_{j=2}^m(\lambda_{(1,j-1)}-\lambda_{(1,m)})\alpha_{(1,j)}+\beta'\right)\\
&= \alpha_0 + (e_1 - \lambda_{(1, m)})\alpha_{(1,1)} + \sum_{j=2}^m(\lambda_{(1,j-1)}-\lambda_{(1,m)})\alpha_{(1,j)} + \beta'\in\Phi_\circ. 
\end{align*}
Suppose that $\beta \sim \alpha_0$ or $\beta \sim \alpha_{(i,1)}$ for some $2 \le i \le t$. Then Corollary \ref{C:Divisibility2} or Proposition ~\ref{P:Divisibility2} respectively imply that $e_1 \, \big| \, \lambda_{(1, m)}$. As a consequence, $e_1-\lambda_{(1, m)} \le 0$. However, $e_1-\lambda_{(1, m)} < 0$ is impossible since then $s(\beta)$ is a root in $\Phi_\circ$ which is neither positive nor negative, yielding a contradiction. Hence, $e_1-\lambda_{(1, m)} = 0$. But then the claim above implies that $\lambda_{(1,1)} = \dots = \lambda_{(1, m-1)} = \lambda_{(1, m)}$. Hence, 
$s(\beta) = \alpha_0 + \beta'$ and we are done. 

It remains to consider the last possibility $\beta \sim \alpha_{(1,1)}$. By Corollary \ref{C:Divisibility2}, we know that $f_1$ divides $1$, hence $f_1 = 1$. Since $\Vert\alpha_{(1,1)}\Vert = \Vert\beta\Vert = \Vert\alpha_0\Vert$, we conclude that $e_1 = 1$, too. Hence, $e_1 - \lambda_{(1,m)} \le 0$ and we can proceed as above. 
\end{proof}

We have investigated $\Phi$ sufficiently. We will use the following lemma to investigate a certain form of factorization in $\Red_T(c)$.

\begin{lemma}\label{L:ReflectionShiftRadicalElement} Let $V$ be a finite dimensional real vector space, $V \times V \stackrel{B}\lar \RR$ a symmetric bilinear form, $\gamma_1, \dots, \gamma_n \in V$ a collection of non-isotropic vectors, $m_1, \dots, m_n \in \ZZ$ and $a \in \mathsf{Rad}(B)$. Then for any $x \in V$ we have
\begin{equation}\label{E:ReflectionRadicalShift}
s_{\gamma_n + m_na} \dots s_{\gamma_1 +m_1a}(x) = s_{\gamma_n} \dots s_{\gamma_1}(x) -
\left(\sum\limits_{i=1}^n m_i\left(x, s_{\gamma_1} \dots s_{\gamma_{i-1}}(\gamma_i^\sharp)\right)\right)a.
\end{equation}
\end{lemma}
\begin{proof} We prove this formula by induction on $n$. 

\smallskip
\noindent
Let $n = 1$ and $\gamma = \gamma_1$ and $m \in \ZZ$. Since $a \in  \mathsf{Rad}(B)$, it follows that for any $x\in V$ we have
\begin{equation}\label{E:ReflectionRadicalShiftSpecial}
    s_{\gamma +ma}(x) = x -\dfrac{2(\gamma+ma,x)}{(\gamma+ma,\gamma+ma)}(\gamma +ma)= x -\dfrac{2(\gamma,x)}{(\gamma,\gamma)}(\gamma +ma)= s_\gamma(x) - m(\gamma^\sharp, x) a,
\end{equation}
proving the formula (\ref{E:ReflectionRadicalShift}) in the case $n=1$. 

\smallskip
\noindent
Now we proceed with a proof of the induction step. For any non-isotropic $\gamma \in V$ and $m \in \ZZ$, we have
\begin{align*}
s_{\gamma + ma}\left(s_{\gamma_n + m_na} \dots s_{\gamma_1 +m_1a}(x)\right) =
s_{\gamma +ma}\left(s_{\gamma_n} \dots s_{\gamma_1}(x) - \left(\sum\limits_{i=1}^n m_i\left(x, s_{\gamma_1} \dots s_{\gamma_{i-1}}(\gamma_i^\sharp)\right)\right)a\right) \\
= s_{\gamma +ma}\left(s_{\gamma_n} \dots s_{\gamma_1}(x)\right) - \left(\sum\limits_{i=1}^n m_i\left(x, s_{\gamma_1} \dots s_{\gamma_{i-1}}(\gamma_i^\sharp)\right)\right)a,
\end{align*}
where we use the fact that $a$ remains fixed under any reflection. 
Using (\ref{E:ReflectionRadicalShiftSpecial}) we obtain
$$
s_{\gamma +ma}\bigl(s_{\gamma_n} \dots s_{\gamma_1}(x)\bigr) = s_\gamma s_{\gamma_n} \dots s_{\gamma_1}(x) -  m \bigl(\gamma^\sharp, s_{\gamma_n} \dots s_{\gamma_1}(x)\bigr) a. 
$$
However, $\bigl(\gamma^\sharp, s_{\gamma_n} \dots s_{\gamma_1}(x)\bigr) = \bigl(x, s_{\gamma_1} \dots s_{\gamma_n}(\gamma^\sharp)\bigr)$, where we use the facts that $s_{\gamma_i}$ is an isometry and $s_{\gamma_i}^2 = \mathbbm{1}$ for all $1 \le i \le n$. This concludes the proof of the induction step. 
\end{proof} 

We examine the formula (\ref{E:ReflectionRadicalShift}) in the following special case.
\begin{definition}
For $\beta\in\Phi_\circ$, we set
\begin{equation*}
\vec{\gamma}(\beta) = (\gamma_n, \dots, \gamma_1) :=
\bigl(\alpha_{(t, p_t-1)}, \dots,\alpha_{(t, 1)}, \dots, \alpha_{(1, p_1-1)}, \dots,\alpha_{(1, 1)}, \beta, \beta\bigr) \in \Phi_\circ^n.
\end{equation*}
Further, let
\begin{equation}\label{E:AnsatzVectorK}
\vec{k} = (k_n, \dots, k_1) = \bigl(k_{(t, p_t-1)}, \dots, k_{(t,1)}, \dots, k_{(1, p_1-1)}, \dots, k_{(1,1)}, k, k'\bigr) \in \ZZ^n
\end{equation}
be such that $\vec{\gamma}(\beta, \vec{k}) := \vec{\gamma}(\beta)  + \vec{k} a \in \Phi^n$.
Then we put $t(\beta, \vec{k}) := s_{\gamma_n + k_n a} \dots s_{\gamma_1 + k_1 a}$ and analogously $\tilde t(\beta, \vec{k}) :=\tilde s_{\gamma_n + k_n a} \dots\tilde s_{\gamma_1 + k_1 a}$ in the tubular case.
\end{definition}

\begin{lemma}\label{L:FactorizationFormula1}
Let $\beta\in\Phi_\circ$ and $\vec{k}$ as in (\ref{E:AnsatzVectorK}) such that $\vec{\gamma}(\beta, \vec{k}) \in \Phi^n$. Then for any $x\in\Gamma$, we have
\begin{equation*}
t(\beta, k)(x) = \hat{c}(x) -
\left((k'-k)(x, \beta^\sharp) + \sum_{(i,j)\in\Omega}\left(\sum_{q=j}^{p_i-1}k_{(i,q)}\right)(x,\alpha_{(i,j)}^\sharp)\right) a,
\end{equation*}
where $\hat{c}:= s_{(t, p_t-1)} \dots s_{(t,1)} \dots s_{(1, p_1-1)} \dots s_{(1,1)}$. In particular, $c(x)=\hat{c}(x) - (x, \alpha_0^\sharp) a$ for any $x\in\Gamma$.
\end{lemma}
\begin{proof} Since $\gamma_1 = \beta = \gamma_2$, we have $s_{\gamma_2} s_{\gamma_1} = \mathbbm{1}$ and 
$s_{\gamma_n} \dots s_{\gamma_1}(x) = \hat{c}(x)$. Next, we have
$$\sum\limits_{l=1}^2 k_l\bigl(x, s_{\gamma_1} \dots s_{\gamma_{l-1}}(\gamma_l^\sharp)\bigr) = 
(x, \beta^\sharp)(k'-k).
$$
For any $(i, j) \in \Omega$, we have $s_{(i,1)} \dots s_{(i, j-1)}\bigl(\alpha_{(i,j)}^\sharp\bigr) = \sum\limits_{p=1}^j \alpha_{(i,p)}^\sharp$. It follows that
$$
\sum\limits_{l=3}^n k_l\bigl(x, s_{\gamma_1} \dots s_{\gamma_{l-1}}(\gamma_l^\sharp)\bigr) = \sum\limits_{(i, j) \in \Omega} k_{(i,j)} \sum\limits_{p=1}^j \bigl(x, \alpha_{(i,p)}^\sharp\bigr)=\sum_{(i,j)\in\Omega}\left(\sum_{q=j}^{p_i-1}k_{(i,q)}\right)(x,\alpha_{(i,j)}^\sharp)
$$
and the statement follows from Lemma \ref{L:ReflectionShiftRadicalElement}. 
\end{proof}

\begin{corollary}\label{C:FactorizationCorollary} Let $\beta\in\Phi_\circ$ and $\vec{k}\in \ZZ^n$ be such that $\gamma(\beta,\vec{k})\in\Phi^n$. Then $t(\beta,\vec{k})=c$ if and only if 
\begin{equation*}
(x, \alpha_0^\sharp) = (k'-k)(x, \beta^\sharp) + \sum_{(i,j)\in\Omega}\left(\sum_{q=j}^{p_i-1}k_{(i,q)}\right)(x,\alpha_{(i,j)}^\sharp)
\end{equation*}
for all $x \in \Gamma$.  Equivalently,  the following identity is true: 
\begin{equation}\label{E:FactorizationCondition2}
\alpha_0^\sharp =  (k'-k)\beta^\sharp+ \sum_{(i,j)\in\Omega}\left(\sum_{q=j}^{p_i-1}k_{(i,q)}\right)\alpha_{(i,j)}^\sharp \; \mathsf{mod} \;  \mathsf{Rad}(B).
\end{equation}
\end{corollary}

For a canonical bilinear lattice $(\Gamma, K)$, we have $\mathsf{rk}\bigl(\mathsf{Rad}(B)\bigr) = 1$ if $\delta \ne 0$ and $\mathsf{rk}\bigl(\mathsf{Rad}(B)\bigr) = 2$ if $\delta = 0$; see Proposition \ref{P:TypeDistinction}. Equation (\ref{E:FactorizationCondition2}) explains why the tubular case $\delta = 0$ requires special treatment. All the technical computations regarding the roots in $\Phi$ were needed to prove the following essential results. Note that each result is given in two versions - one for non-tubular and one for tubular types.
\begin{lemma}\label{L:condition1}
Let $(\Gamma,K)$ be a canonical bilinear lattice of non-tubular type. Let $\beta\in\Phi_\circ$ such that there exists $\vec{k}\in\ZZ^n$ with $\vec{\gamma}(\beta, \vec{k}) \in \Phi^n$ and $t(\beta, \vec{k}) = c$. Then there exists a $w\in\llangle s_{(i, j)} \, \big|\, (i, j) \in \Omega\rrangle\subset W$ such that $w(\beta)=\pm \alpha_0$ and $k'-k=\pm1$.
\end{lemma}
\begin{proof}
Let $\beta=\lambda_0\alpha_0+\sum_{(i,j)\in\Omega}\lambda_{(i,j)}\alpha_{(i,j)}\in\Phi_\circ$. Then 
$$\beta^\sharp=\frac{(\alpha_0,\alpha_0)}{(\beta,\beta)}\lambda_0\alpha_0^\sharp+\sum_{(i,j)\in\Omega}\frac{(\alpha_{(i,j)},\alpha_{(i,j)})}{(\beta,\beta)}\lambda_{(i,j)}\alpha_{(i,j)}^\sharp.$$
Thus, equation (\ref{E:FactorizationCondition2}) implies that $t(\beta, \vec{k}) = c$ if and only if the following element
$$x=\left(\frac{(\alpha_0,\alpha_0)}{(\beta,\beta)}\lambda_0(k'-k)-1\right)\alpha_0^\sharp+\sum_{(i,j)\in\Omega}\left(\frac{(\alpha_{(i,j)},\alpha_{(i,j)})}{(\beta,\beta)}\lambda_{(i,j)}(k'-k)+\sum_{q=j}^{p_i-1}k_{(i,q)}\right)\alpha_{(i,j)}^\sharp$$
is in $\Rad(B)$. Therefore, $x\in\Rad(B)\cap \llangle\Gamma_\circ\rrangle_{\mathbb{Q}}=\{0\}$. The fact that $\{\alpha_\omega^\sharp|\omega\in\overline\Omega\}$ are linearly independent, implies in particular that $\dfrac{(\alpha_0\,\vert\,\alpha_0)}{(\beta\,\vert\,\beta)}\lambda_0(k'-k)=1$. This in turn implies $\Vert\alpha_0\Vert=\Vert\beta\Vert$ and $\lambda_0=k'-k=\pm1$ by Lemma \ref{L:ModificationOfEquation}. So, up to sign, we can assume that $\lambda_0=k'-k=1$ and Proposition \ref{P:KillingTheArms} concludes the proof.
\end{proof}

\begin{lemma}\label{L:condition1hyp}
Let $(\Gamma,K)$ be a canonical bilinear lattice of tubular type. Let $\beta\in\Phi_\circ$ such that there exists $\vec{k}\in\ZZ^n$ with $\vec{\gamma}(\beta, \vec{k}) \in \Phi^n$ and $\tilde{t}(\beta,\vec{k}) = \tilde{c}$. Then there exists a $w\in\llangle\tilde s_{(i, j)} \, \big|\, (i, j) \in \Omega\rrangle\subset\widetilde W$ such that $w(\beta)=\pm \alpha_0$ and $k'-k=\pm1$.
\end{lemma}
\begin{proof}
Note that $\tilde{t}(\beta, \vec{k})(a')-a' = \tilde{c}(a')-a'\in V$. Thus we can compare $\rk(\tilde{t}(\beta, \vec{k})(a')-a')$ and $\rk(\tilde{c}(a')-a')$. By Proposition \ref{P:ActionCoxeter} we have $\rk(\tilde{c}(a')-a')=\dfrac{2}{(\alpha_0, \alpha_0)}$. The action of the composition of reflections $\tilde{s}_{\alpha_{(i,j)} + k_{(i,j)}a}$ does not affect the rank as $a,\alpha_{(i,j)}\in\mathsf{Ker}(\rk)$ for any $(i,j)\in\Omega$. Therefore, Lemma \ref{L:UsefulFormulaReflection} implies that $\rk(\tilde{t}(\beta, \vec{k})(a')-a')=\dfrac{2(k'-k)}{(\beta, \beta)}\lambda_0$.

We conclude that $\dfrac{(\alpha_0,\alpha_0)}{(\beta,\beta)}\lambda_0(k'-k)=1$ as in the non-tubular case. This in turn implies $\Vert\alpha_0\Vert=\Vert\beta\Vert$ and $\lambda_0=k'-k=\pm1$ by Lemma \ref{L:ModificationOfEquation}. So, up to sign, we can assume that $\lambda_0=k'-k=1$ and Proposition \ref{P:KillingTheArms} concludes the proof.
\end{proof}

\begin{lemma}\label{L:condition2}
Let $(\Gamma,K)$ be a canonical bilinear lattice of non-tubular type. Let $\vec{k} \in \ZZ^n$ be such that $\vec{\gamma}(\alpha_0, \vec{k}) \in \Phi^n$ and $t(\alpha_0, \vec{k}) = c$. Then $k_{(i,j)}=0$ for any $(i, j) \in \Omega$.
\end{lemma}
\begin{proof}
Apply equation (\ref{E:FactorizationCondition2}) for $\beta=\alpha_0$ and $k'-k=1$. It is immediate that
\[x=\sum_{(i,j)\in\Omega}\left(\sum_{q=j}^{p_i-1}k_{(i,q)}\right)\alpha_{(i,j)}^\sharp\in\mathsf{Rad}(B).\]
Thus $x\in\Rad(B)\cap \llangle\Gamma_\circ\rrangle_{\mathbb{Q}}=0$. As $\{\alpha_\omega^\sharp|\omega\in\Omega\}$ are linearly independent, we have
$\sum_{q=j}^{p_i-1}k_{(i,q)}=0$ for any $(i,j)\in\Omega$. This concludes the proof.
\end{proof}

\begin{lemma}\label{L:condition2hyp}
Let $(\Gamma,K)$ be a canonical bilinear lattice of tubular type. Let $\vec{k} \in \ZZ^n$ be such that $\vec{\gamma}(\alpha_0, \vec{k}) \in \Phi^n$ and $\tilde t(\alpha_0, \vec{k}) = \tilde c$. Then $k_{(i,j)}=0$ for any $(i, j) \in \Omega$.
\end{lemma}
\begin{proof}
Note that $\tilde t(\alpha_0, \vec{k}) = \tilde c$ implies that in particular $t(\alpha_0,\vec{k})=c$. Apply equation (\ref{E:FactorizationCondition2}) for $\beta=\alpha_0$ and $k'-k=1$. It is immediate that
\[x=\sum_{(i,j)\in\Omega}\left(\sum_{q=j}^{p_i-1}k_{(i,q)}\right)\alpha_{(i,j)}^\sharp\in\mathsf{Rad}(B).\]
Thus $x\in\Rad(B)\cap \llangle\Gamma_\circ\rrangle_{\mathbb{Q}}$. In the tubular case, we have $\Rad(B)\cap \llangle\Gamma_\circ\rrangle_{\mathbb{Q}}=\llangle b\rrangle$, where $b=\alpha_0+\sum_{(i,j)\in\Omega}\frac{p_i-j}{p_i}e_i\alpha_{(i,j)}$ as in (\ref{E:radElmtB}). Note that $\rk(b)\neq0=\rk(x)$, implying that $x=0$. As $\{\alpha_\omega^\sharp|\omega\in\Omega\}$ are linearly independent, we have
$\sum_{q=j}^{p_i-1}k_{(i,q)}=0$ for any $(i,j)\in\Omega$. This concludes the proof.
\end{proof}

\subsection{Proof of transitivity}
In this section we prove the Hurwitz transitivity in a slightly more general setting. The abstract formulation provides an easy to understand insight into why the Hurwitz transitivity follows from the previously established technical results. Moreover, it shows how to approach the investigation of Hurwitz orbits in a wider generality. Further applications of this strategy will appear in the fourth author's PhD thesis.

Throughout this section, let $(W,S)$ be a generalized Coxeter datum in the sense of Definition \ref{D:Coxeter datum}, where $S=\{s_1,\dots,s_n\}$ and $c=s_1\cdots s_n$ with $\ell_T(c)=n$. Let $(\overline{W},\overline{S})$ be a Coxeter system with $\overline{S}=\{\bar{s}_1,\dots,\bar{s}_{n-1}\}$ and set of reflections $\overline{T}$. Assume that we have a group epimorphism $p:W\to\overline{W}$ such that the following conditions hold.
\begin{enumerate}
    \item[(T1)] We have $p(s_i)=\bar{s}_i$ for any $1\leq i\leq n-1$, and $p(s_n)=\bar{s}_{n-1}$.
    \item[(T2)] All elements in $\{\bar{t}\in\overline{T}\,|\,p^{-1}(\bar{s}_1,\dots,\bar{s}_{n-2},\bar{t},\bar{t})\cap\Red_T(c)\}$ are conjugate to $\bar{s}_{n-1}$ under $\llangle\bar{s}_1,\dots,\bar{s}_{n-2}\rrangle\subset\overline W$.
    \item[(T3)] Let $(t_1,\dots,t_n)\in p^{-1}(\bar{s}_1,\dots,\bar{s}_{n-2},\bar{s}_{n-1},\bar{s}_{n-1})\cap\Red_T(c)$. Then, up to the Hurwitz action of $\llangle\sigma_{n-1}\rrangle\subset B_n$, we have $(t_1,\dots,t_n)=(s_1,\dots,s_n)$.
\end{enumerate}

For our proof, we will use that the Hurwitz transitivity of reduced reflection factorizations of Coxeter elements in Coxeter groups as well as some generalizations are very well understood. Recall that a standard parabolic Coxeter element in $(\overline{W},\overline{S})$ is an element of the form $\bar{s}_{i_1}\cdots\bar{s}_{i_k}$ where $1\leq i_1<\dots<i_k\leq n-1$.
\begin{lemma}\label{parabolic}
    Let $\bar{s}_{i_1}\cdots\bar{s}_{i_k}$ be a standard parabolic Coxeter element in $(\overline{W},\overline{S})$ and let $(\bar{t}_1,\dots,\bar{t}_{k+2})\in\overline{T}^{k+2}$ such that $\bar{t}_1\cdots\bar{t}_{k+2}=\bar{s}_{i_1}\cdots\bar{s}_{i_k}$. Then there exist $\bar{t}\in\overline{T}$ and $\tau\in B_{k+2}$ such that $\tau(\bar{t}_1,\dots,\bar{t}_{k+2})=(\bar{s}_{i_1},\dots,\bar{s}_{i_k},\bar{t},\bar{t})$.
\end{lemma}
\begin{proof}
First, we know that $l_{\overline{S}}(\bar{s}_{i_1}\cdots\bar{s}_{i_k})=k=l_{\overline{T}}(\bar{s}_{i_1}\cdots\bar{s}_{i_k})$. Thus we can apply \cite[Lemma 2.3]{WegenerYahiatene} to obtain a braid $\tau_1\in B_{k+2}$ and reflections $r_1,\dots,r_k,r_{k+1}\in \overline{T}$ such that
\[\tau_1(t_1,\dots , t_{k+2}) = (r_1,\dots , r_k, r_{k+1}, r_{k+1}).\]
Now $(r_1,\dots , r_k)$ is a reduced reflection factorization of the parabolic Coxeter element $\bar{s}_{i_1}\cdots\bar{s}_{i_k}$. By \cite[Theorem 1.3]{BaumeisterDyerStumpWegener}, the Hurwitz action on this set is transitive. Therefore we can find a braid $\tau_2'\in B_k$ such that $\tau_2'(r_1,\dots , r_k)=(\bar{s}_{i_1},\dots,\bar{s}_{i_k})$. Interpreting $\tau_2'$ as a braid $\tau_2\in B_{k+2}$ via the standard embedding, we conclude
\begin{align*}
    \tau_2\tau_1(t_1,\dots , t_{k+2}) &= \tau_2(r_1,\dots , r_k, r_{k+1}, r_{k+1})\\
    &=(\bar{s}_{i_1},\dots,\bar{s}_{i_k}, r_{k+1}, r_{k+1}).
\end{align*}   
\end{proof}
Note that this factorization ends with two copies of the same reflection. We need one more simple result on the Hurwitz action in arbitrary groups dealing with such factorizations.
\begin{lemma}\label{sameSame}
Let $G$ be a group, $T \subseteq G$ be a subset closed under conjugation and $t_1, \dots, t_m, t \in T$ be some elements, where we additionally assume that $t^2 = \mathbbm{1}$. Then for any $x \in \llangle t_1, \dots, t_m\rrangle$, there exists a braid $\sigma\in B_{m+2}$ such that
\[\sigma(t_1,\dots , t_m, t, t) = (t_1,\dots , t_m, xtx^{-1}, xtx^{-1}).\]
\end{lemma}
\begin{proof}
For any $1 \le i \le m$, we put
$$\tau_i = \sigma_i \dots \sigma_{m-1} \sigma_m^{-1} \sigma_{m+1}^{-1} \sigma_{m+1}^{-1} \sigma_m^{-1} \sigma_{m-1}^{-1} \dots \sigma_i^{-1}.$$
One can check that  
$$\tau_i(t_1, \dots, t_m, t, t) = (t_1, \dots,  t_m, t_i t t_i^{-1}, t_i t t_i^{-1}),$$
which implies the statement.
\end{proof}

We are now ready to prove our main transitivity theorem.
\begin{theorem}\label{newMainTransitivity}
Let $(W,S)$ be a generalized Coxeter datum such that $\ell_T(c)=|S|$ and there exists a Coxeter system $(\overline{W},\overline{S})$ and a group epimorphism $p:W\to\overline{W}$ satisfying (T1), (T2) and (T3). Then the Hurwitz action on $\Red_T(c)$ is transitive.
\end{theorem}
\begin{proof}
    First, note that the Hurwitz action and $p:W^n\to\overline{W}^n$ commute. Now, let $n=\ell_T(c)$ and let $(t_1,\dots,t_n)\in\Red_T(c)$ be any reduced reflection factorization of $c$. By assumption (T1), we infer that
    \[p(c)=\bar{s}_1\cdots\bar{s}_{n-2}\bar{s}_{n-1}\bar{s}_{n-1}=\bar{s}_1\cdots\bar{s}_{n-2}\]
    is a standard parabolic Coxeter element in $(\overline{W},\overline{S})$. Therefore $p(t_1,\dots,t_n)$ is a factorization of $\bar{s}_1\cdots\bar{s}_{n-2}$ of length $n$. By Lemma \ref{parabolic} there exists a braid $\tau_1\in B_n$ and a reflection $\bar{t}\in\overline{T}$ such that
    \[p(\tau_1(t_1,\dots,t_n))=\tau_1 p(t_1,\dots,t_n)=(\bar{s}_1,\dots,\bar{s}_{n-2},\bar{t},\bar{t}).\]
    By assumption (T2), we know that $\bar{t}$ is conjugate to $\bar{s}_{n-1}$ under $\llangle\bar{s}_1,\dots,\bar{s}_{n-2}\rrangle$. Applying Lemma \ref{sameSame}, we can find a braid $\tau_2$ such that
    \[\tau_2(\bar{s}_1,\dots,\bar{s}_{n-2},\bar{t},\bar{t})=(\bar{s}_1,\dots,\bar{s}_{n-2},\bar{s}_{n-1},\bar{s}_{n-1}).\]
    Therefore $p(\tau_2\tau_1(t_1,\dots,t_n))=(\bar{s}_1,\dots,\bar{s}_{n-2},\bar{s}_{n-1},\bar{s}_{n-1})$ and by assumption (T3), we know that actually
    \[\tau_2\tau_1(t_1,\dots,t_n)=(s_1,\dots,s_n)\]
    up to the Hurwitz action of $\llangle\sigma_{n-1}\rrangle$. Thus any $(t_1,\dots,t_n)\in\Red_T(c)$ is in the same Hurwitz orbit as $(s_1,\dots,s_n)$. This concludes the proof.
\end{proof}

\begin{corollary}\label{C:HurwitzTransitivityNonTub}
    Let $(\Gamma, K)$ be a non-tubular canonical bilinear lattice. Then the Hurwitz action on $\Red_T(c)$ is transitive.
\end{corollary}
\begin{proof}
    We have already shown all necessary conditions to apply Theorem \ref{newMainTransitivity}. The reflection length is given by Corollary \ref{C:reflectionLength} combined with Proposition \ref{P:TypeDistinction}. The epimorphism $p:W\to\overline W:=W_\circ$ is given in Lemma \ref{epi}. It clearly satisfies condition (T1). The fact that $\overline W$ is a Coxeter group, is given in Proposition \ref{P:CoxeterPositiveReduced}. Conditions (T2) and (T3) are reformulated in terms of roots in Lemmas \ref{L:condition1} and \ref{L:condition2}, respectively.
\end{proof}

\begin{corollary}\label{C:HurwitzTransitivityTub}
    Let $(\Gamma, K)$ be a tubular canonical bilinear lattice. Then the Hurwitz action on $\Red_{\widetilde T}(\tilde c)$ is transitive.
\end{corollary}
\begin{proof}
    We have already shown all necessary conditions to apply Theorem \ref{newMainTransitivity}. The reflection length is given by Corollary \ref{C:reflectionLengthHyp}. The epimorphism $p:\widetilde W\to\overline W:=W_\circ$ is given by the concatenation of the epimorphisms in Lemma \ref{hypSES} and Lemma \ref{epi}. The fact that $\overline W$ is a Coxeter group, is given in Proposition \ref{P:CoxeterPositiveReduced}. Conditions (T2) and (T3) are reformulated in terms of roots in Lemmas \ref{L:condition1hyp} and \ref{L:condition2hyp}, respectively.
\end{proof}

\section{Order preserving bijections}\label{S:OrderPreservingBijections}

In this section, we prove the main categorification result. The underlying reduction process has long been familiar to experts and can be found, for example, in \cite{IngallsThomas}, \cite{IgusaSchiffler}, \cite{RingelCatalanCombinatorics}, \cite{HuberyKrause}, and \cite{BaumeisterWegenerYahiateneI}. However, to the best of our knowledge, it has never been presented as clearly and abstractly as we do here.

To establish a poset bijection between non-crossing partitions and exceptional subcategories, one first observes that both are generated by certain sequences: subsequences of reduced reflection factorizations and exceptional sequences, respectively. Since any (sub)sequence can be extended to a sequence of maximal length, it suffices to consider these maximal sequences, on which the braid group acts. The problem of establishing the bijections then reduces to classifying the orbits in these maximal length sequences. The transitivity of the braid group action on complete exceptional sequences in 
$\Coh(\XX)$ was already known from \cite{KussinMeltzer}, whereas the transitivity for reduced reflection factorizations is the main result we established in Section \ref{S:HurwitzTransitivity}.

We adopt this abstract approach to highlight the conceptual simplicity of the proof. All arguments in this section are elementary. To complete the proof of the main result, it remains only to show that exceptional hereditary curves and their associated non-crossing partitions fit naturally into this abstract framework.

\begin{definition}\label{collectionOfMaps}
    Let $\E$ be a set, $n$ an integer and $\E_n\subseteq\E^n$ a set of distinguished $n$-tuples with elements in $\E$. Then we set
    \begin{equation}\label{categorification1}
        \E_r:=\left\{(e_1,\dots,e_r)\,|\,(e_1,\dots,e_n)\in\E_n\right\}. \tag{C1}
    \end{equation}
    Let $\{A_r,\mu_r\}_{1\leq r\leq n}$ be a collection of sets $A_r$ and surjective maps $\mu_r:\E_r\twoheadrightarrow A_r$ such that for any $1\leq r\leq n$, $(e_1,\dots,e_n)\in\E_n$ and $(e_1',\dots,e_r')\in\E_r$ we have
    \begin{equation}\label{categorification2}
        (e_1',\dots,e_r',e_{r+1},\dots,e_n)\in\E_n\Leftrightarrow\mu_r(e_1,\dots,e_r)=\mu_r(e_1',\dots,e_r'). \tag{C2}
    \end{equation}
    Then we call $\{\E_r,A_r,\mu_r\}_{1\leq r\leq n}$ \emph{exceptional datum}.
\end{definition}

\begin{example}\label{Ex:exceptionalDatum}
We give two examples of such exceptional data.
\begin{enumerate}
    \item[(a)] Let $\Coh(\XX)$ be the category of coherent sheaves on an exceptional hereditary curve $\XX$. Let $\E$ be the set of isomorphism classes of exceptional objects, $n=\rk\bigl(K_0(\XX)\bigr)$ and $\E_n$ be the set of complete exceptional sequences in $\Coh(\XX)$ (up to isomorphism). Then by Lemma \ref{L:ExtendabilityCohX} the set $\E_r$ is simply the set of exceptional sequences of length $r$.\\
    We want to find maps $\mu_r$ that govern the simultaneous extendability of exceptional sequences. Let us choose $A_r$ to be the set of exceptional subcategories generated by exceptional sequences of length $r$ and let $\mu_r$ be the map that sends an exceptional sequence to the thick subcategory generated by it. These maps are surjective by definition. Moreover, they fulfill assumption (\ref{categorification2}) by the perpendicular calculus from Lemma ~\ref{L:PerpendicularCalculus}: Let $(E_1,\dots,E_n)$ be a complete exceptional sequence and let $(E_1',\dots,E_r')$ be an exceptional sequence. If $(E_1',\dots,E_r',E_{r+1},\dots,E_n)$ is a complete exceptional sequence, then
    $$\llangle E_1,\dots,E_r\rrangle=\llangle E_{r+1},\dots,E_n\rrangle^\perp=\llangle E_1',\dots,E_r'\rrangle.$$
    On the other hand, if $\llangle E_1,\dots,E_r\rrangle=\llangle E_1',\dots,E_r'\rrangle$, then $$E_1',\dots,E_r' \in\llangle E_1,\dots,E_r\rrangle=\llangle E_{r+1},\dots,E_n\rrangle^\perp.$$ Thus $(E_1',\dots,E_r',E_{r+1},\dots,E_n)$ is an exceptional sequence and therefore a complete exceptional sequence.
    \item[(b)] Let $(W,T,c)$ be a generalized dual Coxeter datum in the sense of Definition \ref{D:Coxeter datum}. Let $\F=T$, $n=\ell_T(c)$ and $\F_n=\Red_T(c)$. For $1\leq r\leq n$, let $B_r$ be the set of non-crossing partitions $w\in\NC_T(W,c)$ of reflection length $\ell_T(w)=r$ and $\nu_r:\F_r\to B_r$ the map that sends a sequence of reflections $(t_1,\dots,t_r)$ to its product $t_1\cdots t_r$.\\
    To see that these maps are surjective, choose any reduced reflection factorization $(t_1,\dots,t_r)$ of $w\in B_r\subset\NC_T(W,c)$. Now $\ell_T(w^{-1}c)=n-r$. So for any reduced reflection factorization $(t_{r+1},\dots,t_n)$ of $w^{-1}c$ we conclude $(t_1,\dots,t_n)\in\Red_T(c)$. The fact that they fulfill assumption (\ref{categorification2}) follows from the elementary fact that $t_1\cdots t_r=c(t_{r+1}\cdots t_n)^{-1}$ for any $(t_1,\dots,t_n)\in\Red_T(c)$.
\end{enumerate}
\end{example}

Let $\{\E_r,A_r,\mu_r\}_{1\leq r\leq n},\{\F_r,B_r,\nu_r\}_{1\leq r\leq n}$ be two exceptional data and let $\rho:\E\hookrightarrow\F$ be an injective map between the respective base sets. Then we denote by $\rho_r:\E_r\rightarrow\F^r$ the maps defined by $\rho_r(e_1,\dots,e_r)=(\rho(e_1),\dots,\rho(e_r))$ for $1\leq r\leq n$. Assume that we have an element $x\in\E_n$ and a group $G$ such that the following holds.
\begin{enumerate}
    \item[(C3)] We have $\rho_n(x)\in\F_n$.
    \item[(C4)] The group $G$ acts on $\E_n$ and $\F^n$ such that $\rho_n:\E_n\to\F^n$ is equivariant under these actions.
\end{enumerate}

\begin{lemma}\label{bijectionCompleteSequences}
    If $G$ acts transitively on $\E_n$ and $\F_n$, then $\Image(\rho_n)=\F_n$. In other words, $\rho_n:\E_n\xrightarrow{\sim}\F_n$ is an isomorphism.
\end{lemma}
\begin{proof}
    By assumption (C3), there is an $x\in\E_n$ such that $\rho_n(x)\in\F_n$. Let $y\in\F^n$. Then $y\in\Image(\rho_n)$ if and only if there exists $x'\in\E_n$ such that $\rho_n(x')=y$. By the transitivity on $\E_n$, this is equivalent to the existence of $g\in G$ such that $\rho_n(g\cdot x)=y$. By assumption (C4), this can be reformulated as the existence of $g\in G$ such that $g\cdot\rho_n(x)=y$. Finally, this is equivalent to $y\in\F_n$ by the transitivity on $\F_n$.
\end{proof}
\begin{lemma}\label{bijectionSequences}
    If $\rho_n:\E_n\xrightarrow{\sim}\F_n$ is an isomorphism, then $\rho_r:\E_r\xrightarrow{\sim}\F_r$ is an isomorphism for any $1\leq r\leq n$.
\end{lemma}
\begin{proof}
    Let $1\leq r\leq n$ and $(f_1,\dots,f_r)\in\F^r$. Then $(f_1,\dots,f_r)\in\Image(\rho_r)$ if and only if there exists $(e_1,\dots,e_r)\in\E_r$ such that $\rho_r(e_1,\dots,e_r)=(f_1,\dots,f_r)$. By the definition (\ref{categorification1}) of $\E_r$, this is the case if and only if there exist $(e_1,\dots,e_n)\in\E_n$ and $f_{r+1},\dots,f_n\in\F$ such that $\rho_n(e_1,\dots,e_n)=(f_1,\dots,f_n)$. By the assumption $\Image(\rho_n)=\F_n$, this is equivalent to the existence of $f_{r+1},\dots,f_n\in\F$ such that $(f_1,\dots,f_n)\in\F_n$. Finally, this is the case if and only if $(f_1,\dots,f_r)\in\F_r$ by the definition (\ref{categorification1}) of $\F_r$.
\end{proof}

\begin{lemma}\label{bijectionSets}
    If $\rho_r:\E_r\xrightarrow{\sim}\F_r$ is an isomorphism for any $1\leq r\leq n$, then there exist isomorphisms $\theta_r:A_r\xrightarrow{\sim}B_r$ for $1\leq r\leq n$ completing the following commutative square.
    \begin{equation}\label{E:commutativeSquare}
        \begin{tikzcd}
    \E_r \arrow[r, "\rho_r"', "\sim"] \arrow[d, two heads, "\mu_r"'] & \F_r \arrow[d, two heads, "\nu_r"] \\
    A_r \arrow[r, "\theta_r"', "\sim"] & B_r
    \end{tikzcd}
    \end{equation}
\end{lemma}
\begin{proof}
    Let $1\leq r\leq n$. We define $\theta_r$ by
    \[\theta_r:A_r\rightarrow B_r,\quad\theta_r(a)=\nu_r\circ\rho_r(e_1,\dots,e_r)\text{ for }(e_1,\dots,e_r)\in\mu_r^{-1}(a).\]
    It is clear from the diagram (\ref{E:commutativeSquare}) that $\theta_r$ is surjective. Let us prove that $\theta_r$ is well-defined and injective at once.
    
    Let $a,a'\in A_r$ with $(e_1,\dots,e_r)\in\mu_r^{-1}(a)$, $(e_1',\dots,e_r')\in\mu_r^{-1}(a')$. By the assumptions (\ref{categorification1}) and (\ref{categorification2}) on $\mu_r$, we know that $a=a'$ if and only if there exist $e_{r+1},\dots,e_n\in\E$ such that $(e_1,\dots,e_n)$, $(e_1',\dots,e_r',e_{r+1},\dots,e_n)\in\E_n$. Applying the bijectivity of $\rho_n$, this is equivalent to the existence of $f_{r+1},\dots,f_n\in\F$ such that $(\rho(e_1),\dots,\rho(e_r),f_{r+1},\dots,f_n)$, $(\rho(e_1'),\dots,\rho(e_r'),f_{r+1},\dots,f_n)\in\F_n$. By the assumptions (\ref{categorification1}) and (\ref{categorification2}) on $\nu_r$, this is equivalent to $\nu_r(\rho(e_1),\dots,\rho(e_r))=\nu_r(\rho(e_1'),\dots,\rho(e_r'))$, i.e. $\theta_r(a)=\theta_r(a')$.
\end{proof}

In conclusion, we have shown a bijection between the sets $A_r$ and $B_r$ for any $1\leq r\leq n$ under the assumption of transitivity. We will now show that an exceptional datum even defines a partial order on $\sqcup_{r=1}^nA_r$ making $\sqcup_{r=1}^n\theta_r$ a bijection of posets.
\begin{lemma}\label{partialOrderGeneral}
    Let $\{\E_r,A_r,\mu_r\}_{1\leq r\leq n}$ be an exceptional datum and let $A:=\sqcup_{r=1}^nA_r$. Then
    \begin{equation*}
        a'\leq_\mu a\,:\Leftrightarrow\,\exists 1\leq r\leq s\leq n, (e_1,\dots,e_n)\in \E_n\text{ s.t. }\mu_r(e_1,\dots,e_r)=a',\,\mu_s(e_1,\dots,e_s)=a
    \end{equation*}
    defines a partial order on $A$.
\end{lemma}
\begin{proof}
    Reflexivity and antisymmetry are clear. We only show transitivity. Let $a,a',a''\in A$ such that $a''\leq_\mu a'$ and $a'\leq_\mu a$. Then there exist $1\leq q\leq r\leq s\leq n$, $(e_1,\dots,e_n)\in\E_n$ and $(e_1',\dots,e_n')\in\E_n$ such that $\mu_s(e_1,\dots,e_s)=a$, $\mu_r(e_1,\dots,e_r)=a'=\mu_r(e_1',\dots,e_r')$ and $\mu_q(e_1',\dots,e_q')=a''$. By assumption (\ref{categorification2}) we have $(e_1',\dots,e_r',e_{r+1},\dots,e_n)\in\E_n$. Once again, apply assumption (\ref{categorification2}) to see that not only $\mu_q(e_1',\dots,e_q')=a''$ but also $\mu_s(e_1',\dots,e_r',e_{r+1},\dots,e_s)=\mu_s(e_1,\dots,e_s)=a$. This shows $a''\leq_\mu a$.
\end{proof}

\begin{example}\label{Ex:partialOrder}
Let us apply Lemma \ref{partialOrderGeneral} to the exceptional data seen in Example \ref{Ex:exceptionalDatum}.
\begin{enumerate}
    \item[(a)] Let $\Coh(\XX)$ be the category of coherent sheaves on an exceptional hereditary curve $\XX$. Let $\{\E_r,A_r,\mu_r\}$ be the exceptional datum defined by exceptional sequences and the thick subcategories generated by them as seen in Example \ref{Ex:exceptionalDatum} ~(a).
    
    Let $\mathcal{H}', \mathcal{H}''$ be two exceptional subcategories of $\Coh(\XX)$. If $\mathcal{H}''\subseteq\mathcal{H}'$ then let $(E_1,\dots,E_r)$ be any exceptional sequence that generates $\mathcal{H}''$. By Lemma \ref{L:ExtendabilityCohX} there exist exceptional objects $E_{r+1},\dots,E_s$ such that $(E_1,\dots,E_s)$ is an exceptional sequence that generates $\mathcal{H}'$, i.e. $\mathcal{H}''\leq_\mu\mathcal{H}'$. On the other hand, let $\mathcal{H}''\leq_\mu\mathcal{H}'$. Then there exists $1\leq r\leq s\leq n$ and a complete exceptional sequence $(E_1,\dots,E_n)$ such that $(E_1,\dots,E_s)$ generates $\mathcal{H}'$ and $(E_1,\dots,E_r)$ generates $\mathcal{H}''$. In particular, $E_1,\dots,E_r\in\mathcal{H}'$ so $\mathcal{H}''\subseteq\mathcal{H}'$. Thus $\leq_\mu$ is the inclusion.
    \item[(b)] Let $(W,T,c)$ be a generalized dual Coxeter datum. Let $\{\F_r,B_r,\nu_r\}$ be the exceptional datum defined by reduced reflection factorizations and non-crossing partitions as seen in Example \ref{Ex:exceptionalDatum} (b).
    
    Let $v,w\in\NC_T(W,c)$ be two non-crossing partitions with $\ell_T(v)=r$, $\ell_T(w)=s$. If $v\leq_T w$, then let $(t_1,\dots,t_r)$ be any reduced reflection factorization of $v$. Now $\ell_T(v^{-1}w)=s-r$. So for any reduced reflection factorization $(t_{r+1},\dots,t_s)$ of $v^{-1}w$ we conclude that $(t_1,\dots,t_s)$ is a reduced reflection factorization of $w$ and thus an initial subsequence of a reduced reflection factorization of $c$, i.e. $v\leq_\nu w$. On the other hand, let $v\leq_\nu w$. Then $r\leq s$ and there exists $(t_1,\dots,t_n)\in\Red_T(c)$ such that $t_1\cdots t_s=w$ and $t_1\cdots t_r=v$. Then $v^{-1}w=t_{r+1}\cdots t_s$. In particular $v\leq_T w$. Thus $\leq_\nu$ is the absolute order $\leq_T$.
\end{enumerate}
\end{example}

\begin{proposition}\label{P:posetBijectionAbstract}
    Let $\{\E_r,A_r,\mu_r\}_{1\leq r\leq n},\{\F_r,B_r,\nu_r\}_{1\leq r\leq n}$ be two exceptional data and let $\rho:\E\hookrightarrow\F$ an injective map, $G$ a group satisfying the assumptions (C3) and (C4). Assume further that $G$ acts transitively on $\E_n$ and $\F_n$. Set $A=\sqcup_{r=1}^nA_r$, $B=\sqcup_{r=1}^nB_r$. Then there exists an order preserving bijection $\theta:A\to B$ with respect to the partial orders $\leq_\mu$ and $\leq_\nu$ as defined in Lemma \ref{partialOrderGeneral}.
\end{proposition}
\begin{proof}
    By Lemmas \ref{bijectionCompleteSequences}, \ref{bijectionSequences} and \ref{bijectionSets}, it only remains to check that $\theta:=\sqcup_{r=1}^n\theta_r$ defined by (\ref{E:commutativeSquare}) is order preserving.
    
    Let $a,a'\in A$. Then $a'\leq_\mu a$ if and only if there exist $1\leq r\leq s\leq n$ and $(e_1,\dots,e_n)\in\E_n$ such that $\mu_r(e_1,\dots,e_r)=a'$ and $\mu_s(e_1,\dots,e_s)=a$. By the definition of $\theta_r$ and $\theta_s$ via (\ref{E:commutativeSquare}), this is equivalent to the existence of $1\leq r\leq s\leq n$ and $(e_1,\dots,e_n)\in\E_n$ such that $\nu_r\circ\rho_r(e_1,\dots,e_r)=\theta(a')$ and $\nu_s\circ\rho_s(e_1,\dots,e_s)=\theta(a)$. By the bijectivity of $\rho_n:\E_n\to\F_n$, this is equivalent to the existence of $1\leq r\leq s\leq n$ and $(f_1,\dots,f_n)\in\F_n$ such that $\nu_r(f_1,\dots,f_r)=\theta(a')$ and $\nu_s(f_1,\dots,f_s)=\theta(a)$. This is the case if and only if $\theta(a')\leq_\nu\theta(a)$ by definition of $\leq_\nu$.
\end{proof}

\noindent
We have thus seen the actual proof of the categorification result. We have seen how exceptional subcategories of $\Coh(\XX)$ and non-crossing partitions individually fit into this framework. It remains to associate a reflection group (of canonical type) to a category $\Coh(\XX)$ and note that the map sending an exceptional object to its reflection satisfies the assumptions (C3) and (C4).
\begin{definition}
    Let $\Coh(\XX)$ be the category of coherent sheaves on an exceptional hereditary curve $\XX$. Let $(\Gamma,K)$ be its Grothendieck group (of rank $n$) equipped with the Euler form. Let $(E_1,\dots,E_n)$ be any complete exceptional sequence in $\Coh(\XX)$ and let $R=([E_1],\dots,[E_n])$. Then the generalized dual Coxeter datum $(W,T,c)$ defined by $R$ via Definition \ref{D:Coxeter datum} is the \emph{generalized dual Coxeter datum associated to $\Coh(\XX)$}.
    
We denote by $\mathsf{Ex}\bigl(\Coh(\XX)\bigr)$ the set of thick exact subcategories of $\Coh(\XX)$ generated by an exceptional sequence. It  is a partially ordered set with respect to the inclusion of subcategories.
\end{definition}

\begin{proposition}\label{P:ReflectionGroupForCategory}
    Let $\Coh(\XX)$ be the category of coherent sheaves on an exceptional hereditary curve $\XX$. The generalized dual Coxeter datum associated to $\Coh(\XX)$ does not depend on the choice of complete exceptional sequence. Moreover, $W$ is a reflection group of canonical type in the sense of Definition \ref{D:CanonicalLattice}.
\end{proposition}
\begin{proof}
    By Theorem ~\ref{T:KussinMeltzer}, the Hurwitz action on complete exceptional sequences in $\Coh(\XX)$ is transitive. Using Lemma \ref{L:Mutate}, we can interpret this braid group action as the braid group action on complete exceptional sequences in $(\Gamma,K)$ given in Proposition \ref{P:braidGroupActionRoots}. Now Remark ~\ref{R:CoxData1} (d) shows that $(W,T,c)$ does not depend on the choice of complete exceptional sequence.
    
    By Theorem ~\ref{T:ExceptionalMain} and Proposition \ref{P:categoryInterpretation} it is clear that, choosing the standard exceptional sequence (\ref{E:StandardCollection}), the associated reflection group $W$ is a reflection group of canonical type. 
\end{proof}

\begin{theorem}\label{T:MainCategorification}
    Let $\Coh(\XX)$ be the category of coherent sheaves on a non-tubular exceptional hereditary curve $\XX$ and let $(W,T,c)$ be the associated generalized dual Coxeter datum. Then the map 
\begin{equation}\label{E:CoxMap1}
\mathsf{cox}: \mathsf{Ex}\bigl(\Coh(\XX)\bigr) \lar \mathsf{NC}_{T}(W, c), \quad \llangle F_1, \dots, F_r\rrangle \mapsto s_{[F_1]} \dots s_{[F_r]}
\end{equation}
is an isomorphism of posets. 
Here, $\llangle F_1, \dots, F_r\rrangle$ denotes the thick exact (in fact, abelian) subcategory of $\Coh(\XX)$  generated by an exceptional sequence $\bigl(F_1, \dots, F_r\bigr)$.
\end{theorem}
\begin{proof}
    Let $\{\E_r,A_r,\mu_r\}$ be the exceptional datum from Example \ref{Ex:exceptionalDatum} (a). Let $\{\F_r,B_r,\nu_r\}$ be the exceptional datum from Example \ref{Ex:exceptionalDatum} (b) for the generalized dual Coxeter datum $(W,T,c)$. Then $(\sqcup_{r=1}^nA_r,\leq_\mu)$ is the set $\mathsf{Ex}\bigl(\Coh(\XX)\bigr)$ of exceptional subcategories of $\Coh(\XX)$, ordered by inclusion, and $(\sqcup_{r=1}^nB_r,\leq_\nu)$ is $\mathsf{NC}_T(W,c)$, ordered by the absolute order $\leq_T$ by Example \ref{Ex:partialOrder}.

    Let $\rho:\E\to\F$ be the map that sends an exceptional object $E$ to its associated reflection $s_{[E]}$. Indeed, we have $s_{[E]}\in T$. Consider the following argument. Let $E$ be an exceptional object. Then by Lemma \ref{L:ExtendabilityCohX} there exists a complete exceptional sequence $(E_1,\dots,E_n)$ such that $E_1=E$. If we choose this complete exceptional sequence, it is clear that $s_{[E]}\in S\subset T$. As argued above, $T$ does not depend on the choice of complete exceptional sequence by Remark ~\ref{R:CoxData1} (d) and Theorem ~\ref{T:KussinMeltzer}. The map $\rho$ is injective by \ref{T:HereditaryKey} (b) and Proposition \ref{P:HKLemmas}, i.e. two exceptional objects are isomorphic if and only if they define the same reflection. Moreover, there exists $x\in\mathbf{E}_n$ such that $\rho_n(x)\in\F_n$ by construction. The reader who prefers a specific complete exceptional sequence may find it in Theorem \ref{T:ExceptionalMain}. In conclusion, $\rho:\E\hookrightarrow\F$ is a map satisfying assumption (C3).
    
    Finally, let $G=B_n$ be the braid group on $n$ strands. The braid group equivariance of $\rho_n$ follows from Lemma \ref{L:Mutate}, i.e. assumption (C4) is satisfied. The transitivity of the $B_n$ action on $\E_n$ is Theorem \ref{T:KussinMeltzer}. The transitivity of the $B_n$ action on $\F_n$ is our first main result Corollary \ref{C:HurwitzTransitivityNonTub}. The assertion thus follows from Proposition \ref{P:posetBijectionAbstract}.
\end{proof}

\begin{theorem}\label{T:MainCategorificationII}
    Let $\Coh(\XX)$ be the category of coherent sheaves on a tubular exceptional hereditary curve $\XX$ and let $(\widetilde W,\widetilde T,\tilde c)$ be the hyperbolic extension of the associated generalized dual Coxeter datum $(W,T,c)$. Then the map 
\begin{equation}\label{E:CoxMap2}
\widetilde{\mathsf{cox}}: \mathsf{Ex}\bigl(\Coh(\XX)\bigr) \lar \mathsf{NC}_{T}(W, c), \quad \llangle F_1, \dots, F_r\rrangle \mapsto \tilde{s}_{[F_1]} \dots \tilde{s}_{[F_r]}
\end{equation}
is an isomorphism of posets.     
\end{theorem}
\begin{proof}
    Let $\{\E_r,A_r,\mu_r\}$ be the exceptional datum from Example \ref{Ex:exceptionalDatum} (a). Let $\{\F_r,B_r,\nu_r\}$ be the exceptional datum from Example \ref{Ex:exceptionalDatum} (b) for the generalized dual Coxeter datum $(\widetilde W,\widetilde T,\tilde c)$. Then $(\sqcup_{r=1}^nA_r,\leq_\mu)$ is the set $\mathsf{Ex}\bigl(\Coh(\XX)\bigr)$ of exceptional subcategories of $\Coh(\XX)$, ordered by inclusion, and $(\sqcup_{r=1}^nB_r,\leq_\nu)$ is $\mathsf{NC}_{\widetilde T}(\widetilde W,\tilde c)$, ordered by the absolute order $\leq_{\widetilde T}$ by Example ~\ref{Ex:partialOrder}.

    Recall from Lemma \ref{hypSES}, that there is a group epimorphism $\pi:\widetilde W\to W$. It is not hard to see that $\pi$ restricts to an isomorphism of sets between $\widetilde T$ and $T$; see Corollary \ref{hyperbolicEpimorphism}. Let $\rho:\E\to\F$ be the map that sends an exceptional object $E$ to its associated reflection $\tilde s_{[E]}$ in the hyperbolic extension. Indeed, $\tilde s_{[E]}\in\widetilde T$ by the same argument as in the proof of Theorem \ref{T:MainCategorification}. Note that $E\mapsto s_{[E]}$ is an injective map from $\E$ to $T$ and $s_{[E]}\mapsto\tilde s_{[E]}$ is a restriction of the isomorphism between $T$ and $\widetilde T=\F$. Thus $\rho:\E\to\F$ is injective. Moreover, by Theorem \ref{T:ExceptionalMain} there exists a complete exceptional sequence such that the corresponding product of reflections is $c\in W$. By Equation (\ref{E:DefinitionHyperbolicCoxeter}), this factorization lifts to a factorization of $\tilde c\in\widetilde W$, i.e. there exists $x\in\mathbf{E}_n$ such that $\rho_n(x)\in\F_n$. In conclusion, $\rho:\E\hookrightarrow\F$ is a map satisfying assumption (C3).
    
    Finally, let $G=B_n$ be the braid group on $n$ strands. The braid group equivariance of $\rho_n$ follows from Lemma \ref{L:Mutate} and the fact that $\pi:\widetilde W\to W$ is a morphism of groups. The transitivity of the $B_n$ action on $\E_n$ is Theorem \ref{T:KussinMeltzer}. The transitivity of the $B_n$ action on $\F_n$ is Corollary \ref{C:HurwitzTransitivityTub}. The assertion thus follows from Proposition \ref{P:posetBijectionAbstract}.
\end{proof}

\begin{corollary} Let $\XX$ be an exceptional hereditary curve and $\gamma \in  \Gamma$ be a real root.
\begin{itemize}
\item[(a)] Assume that $\XX$ is non-tubular. Then there exists an exceptional object 
$E \in \Coh(\XX)$ such that $[E] \in \bigl\{+\gamma, -\gamma\}$ if and only if 
$s_\gamma \in \mathsf{NC}_T(W, c)$ is a non-crossing partition.
\item[(b)] Similarly, if $\XX$ is tubular then there exists an exceptional object 
$E \in \Coh(\XX)$ such that $[E] \in \bigl\{+\gamma, -\gamma\}$ if and only if 
$\tilde{s}_\gamma \in \mathsf{NC}_{\widetilde{T}}(\widetilde{W}, \tilde{c})$ is a non-crossing partition.
\end{itemize}
\end{corollary}

\begin{appendix}
\section{The domestic and tubular types}\label{S:Comparison}
In this appendix, we compare the reflection groups of canonical type to reflection groups already known in the literature. More precisely, we discuss that domestic reflection groups of canonical type are precisely the affine Coxeter groups. Moreover, tubular reflection groups of canonical type are elliptic Weyl groups introduced by Saito \cite{SaitoI}.

Classical (resp. elliptic) Dynkin diagrams are used to classify finite and affine Coxeter groups (resp. elliptic Weyl groups). In order to compare the reflection groups of canonical type with these groups, we will introduce the notion of a Dynkin diagram of canonical type. The conventions we are using are inspired from \cite{Bourbaki} and \cite{SaitoI}.

Consider the setting of Section \ref{S:ReflectionGroups}. Let $\sigma$ be a symbol as defined in (\ref{E:Symbol}); see Definition \ref{D:Symbol}. Let $W$ be the reflection group of canonical type obtained from the canonical bilinear lattice of Definition \ref{D:CanonicalLattice} with $$R = \bigl(\alpha_{(t, p_t-1)}, \dots, \alpha_{(t, 1)},\dots,\alpha_{(1, p_1-1)}, \dots, \alpha_{(1,1)}, \alpha_0, \alpha_{0^\ast}\bigr).$$ Recall also that the elements of $R$ are the simple roots in $(V,B)$, where $V$ is a real vector space with basis $R$, equipped with the symmetric bilinear form $B=(-,-)$, which is the symmetrization of the form $K$. Recall that the Coxeter element $c$ is the product of the reflections corresponding to the elements of $R$ in the same order.

\begin{definition}\label{D:DynkinCanonical}
    A \emph{Dynkin diagram of canonical type} is a diagram with set of vertices in bijection with the set of simple roots $R$ and the edges between two simple roots illustrate the symmetric bilinear form following the conventions in Figure \ref{rules}. We obtain the Dynkin diagrams of canonical type given in Figures \ref{epsilon1} and \ref{epsilon2} for the cases $\varepsilon=1$ and $\varepsilon=2$, respectively.
\end{definition}

\begin{figure}[h]
\noindent
\begin{tikzpicture}[
roundnode/.style={circle, draw=black, fill=black!50, inner sep=0pt, minimum width=4pt}
]
\node[roundnode]      (1)                            {};
\node[roundnode]      (1*)        [above=1cm of 1] {};
\node				  (d)         [above=0.3cm of 1] {};
\node				  (dots)      [below=2.5cm of 1] {$\dots$};
\node				  (1name)	  [below=0cm of 1] {\scriptsize $\alpha_0$};
\node				  (1*name)	  [above=0cm of 1*] {\scriptsize $\alpha_{0^*}$};

\node[roundnode]      (11)       [left=2cm of d] {};
\node[roundnode]      (12)       [left=1cm of 11] {};
\node			      (1dots)    [left=0.2cm of 12] {$\dots$};
\node[roundnode]      (1-1)      [left=1.2cm of 12] {};
\node[roundnode]      (1-2)      [left=1cm of 1-1] {};
\node			      (1dtext1)  [left=1.2cm of d] {};
\node			      (1ddtext1) [above=0.1cm of 1dtext1] {};
\node[rotate=345]     (1text1)   [below=0.35cm of 1ddtext1] {\scriptsize $(f_1,e_1)$};
\node			      (1dtext1*) [left=1cm of d] {};
\node[rotate=15]      (1text1*)  [above=0.1cm of 1dtext1*] {\scriptsize $(f_1,e_1)$};
\node				  (11name)	  [above=0cm of 11] {\scriptsize $\alpha_{(1,1)}$};
\node				  (12name)	  [above=0cm of 12] {\scriptsize $\alpha_{(1,2)}$};
\node				  (1-1dname)  [below=0.1cm of 1-1] {};
\node				  (1-1ddname) [right=1cm of 1-1dname] {};
\node				  (1-1name)	  [left=0.25cm of 1-1ddname] {\tiny $\alpha_{(1,p_1-2)}$};
\node				  (1-2name)	  [above=0cm of 1-2] {\tiny $\alpha_{(1,p_1-1)}$};

\node			      (dummy21)  [left=1.2cm of 1] {};
\node[roundnode]      (21)       [below=1.2cm of dummy21] {};
\node			      (dummy22)  [left=0.8cm of 21] {};
\node[roundnode]      (22)       [below=0.8cm of dummy22] {};
\node			      (2dummydots) [left=0.4cm of 22] {};
\node[rotate=45]      (2dots)    [below=0.4cm of 2dummydots] {$\dots$};
\node			      (dummy2-1) [left=1cm of 22] {};
\node[roundnode]      (2-1)      [below=1cm of dummy2-1] {};
\node			      (dummy2-2) [left=0.8cm of 2-1] {};
\node[roundnode]      (2-2)      [below=0.8cm of dummy2-2] {};
\node			      (2dtext1*) [left=1.2cm of d] {};
\node[rotate=60]      (2text1*)  [below=1.0cm of 2dtext1*] {\scriptsize $(f_2,e_2)$};
\node			      (2dtext1)  [left=0.7cm of d] {};
\node[rotate=45]	  (2text1)   [below=1.25cm of 2dtext1] {\scriptsize $(f_2,e_2)$};
\node				  (21dname)	  [right=0.1cm of 21] {};
\node				  (21name)	  [below=0cm of 21dname] {\scriptsize $\alpha_{(2,1)}$};
\node				  (2-2name)	  [below=0cm of 2-2] {\scriptsize $\alpha_{(2,p_2-1)}$};

\node[roundnode]      (-11)      [right=2cm of d] {};
\node[roundnode]      (-12)      [right=1cm of -11] {};
\node			      (-1dots)    [right=0.2cm of -12] {$\dots$};
\node[roundnode]      (-1-1)     [right=1.2cm of -12] {};
\node[roundnode]      (-1-2)     [right=1cm of -1-1] {};
\node			      (-1dtext1)  [right=1.3cm of d] {};
\node			      (-1ddtext1) [above=0.1cm of -1dtext1] {};
\node[rotate=15]      (-1text1)   [below=0.3cm of -1ddtext1] {\scriptsize $(e_{t},f_{t})$};
\node			      (-1dtext1*) [right=1cm of d] {};
\node[rotate=345]     (-1text1*)  [above=0.1cm of -1dtext1*] {\scriptsize $(e_{t},f_{t})$};
\node				  (-11name)	  [above=0cm of -11] {\scriptsize $\alpha_{(t,1)}$};
\node				  (-1-2dname) [left=0cm of -1-2] {};
\node				  (-1-2name)  [above=0cm of -1-2dname] {\tiny $\alpha_{(t,p_t-1)}$};

\node			      (dummy-21)  [right=1.2cm of 1] {};
\node[roundnode]      (-21)       [below=1.2cm of dummy-21] {};
\node			      (dummy-22)  [right=0.8cm of -21] {};
\node[roundnode]      (-22)       [below=0.8cm of dummy-22] {};
\node			      (-2dummydots) [right=0.4cm of -22] {};
\node[rotate=315]     (-2dots)    [below=0.4cm of -2dummydots] {$\dots$};
\node			      (dummy-2-1) [right=1cm of -22] {};
\node[roundnode]      (-2-1)      [below=1cm of dummy-2-1] {};
\node			      (dummy-2-2) [right=0.8cm of -2-1] {};
\node[roundnode]      (-2-2)      [below=0.8cm of dummy-2-2] {};
\node			      (-2dtext1*) [right=1.15cm of d] {};
\node[rotate=300]	  (-2text1*)  [below=0.9cm of -2dtext1*] {\tiny $(e_{t-1},f_{t-1})$};
\node			      (-2dtext1)  [right=0.6cm of d] {};
\node[rotate=315]	  (-2text1)   [below=1.15cm of -2dtext1] {\tiny $(e_{t-1},f_{t-1})$};
\node				  (-21dname)  [above=0cm of -21] {};
\node				  (-21ddname) [below=0.05cm of -21dname] {};
\node				  (-21name)	  [right=0cm of -21ddname] {\scriptsize $\alpha_{(t-1,1)}$};
\node				  (-2-2dname)  [right=0.1cm of -2-2] {};

\node				  (-2-2name)  [below=0cm of -2-2dname] {\scriptsize $\alpha_{(t-1,p_{t-1}-1)}$};

\begin{scope}[decoration={
    markings,
    mark=at position 0.6 with {\arrow{Straight Barb}}}
    ]
\draw (1) -- (11);
\draw (1) -- (21);
\draw (1) -- (-11);
\draw (1) -- (-21);
\draw (1*) -- (11);
\draw (1*) -- (21);
\draw (1*)  -- (-11);
\draw (1*) -- (-21);
\draw (11) -- (12);
\draw (1-1) -- (1-2);
\draw (21) -- (22);
\draw (2-1) -- (2-2);
\draw (-11) -- (-12);
\draw (-1-1) -- (-1-2);
\draw (-21) -- (-22);
\draw (-2-1) -- (-2-2);
\end{scope}
\begin{scope}[decoration={
    markings,
    mark=at position 0.75 with {\arrow{Straight Barb[length=3mm, width=3mm]}}}
    ] 
\draw[densely dotted, double] (1*) -- (1);
\end{scope}
\end{tikzpicture}
\caption{Dynkin diagram for a reflection group of canonical type in the case $\varepsilon=1$.}\label{epsilon1}
\end{figure}
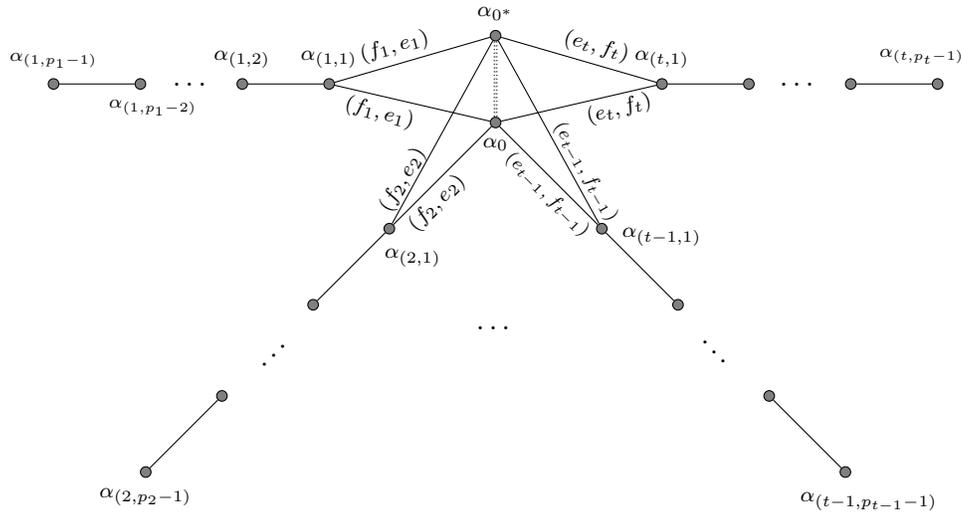
 
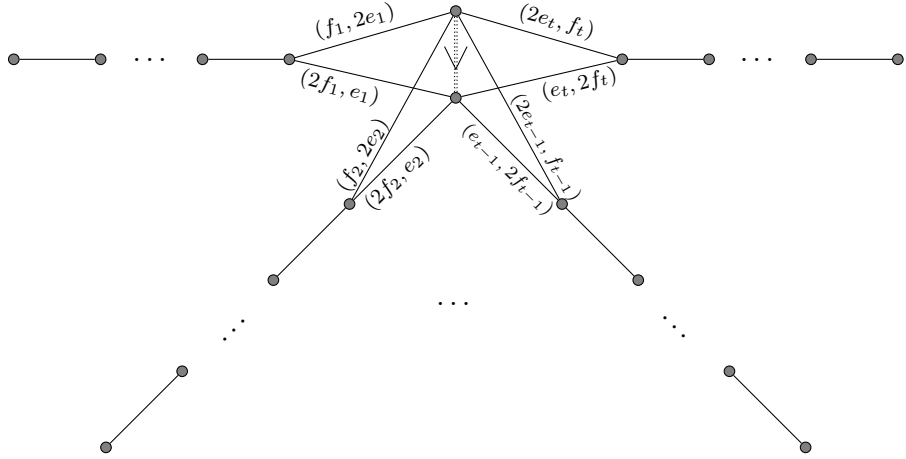
\begin{figure}[h]
\noindent
\begin{tikzpicture}[
roundnode/.style={circle, draw=black, fill=black!50, inner sep=0pt, minimum width=4pt}
]
\node[roundnode]      (1v2)                            {};
\node[roundnode]      (1*v2)        [above=1cm of 1v2] {};
\node				  (dv2)         [above=0.3cm of 1v2] {};
\node				  (dotsv2)      [below=2.5cm of 1v2] {$\dots$};

\node[roundnode]      (11v2)       [left=2cm of dv2] {};
\node[roundnode]      (12v2)       [left=1cm of 11v2] {};
\node			      (1dotsv2)    [left=0.2cm of 12v2] {$\dots$};
\node[roundnode]      (1-1v2)      [left=1.2cm of 12v2] {};
\node[roundnode]      (1-2v2)      [left=1cm of 1-1v2] {};
\node			      (1dtext1v2)  [left=1.2cm of dv2] {};
\node			      (1ddtext1v2) [above=0.1cm of 1dtext1v2] {};
\node[rotate=345]     (1text1v2)   [below=0.35cm of 1ddtext1v2] {\scriptsize $(2f_1,e_1)$};
\node			      (1dtext1*v2) [left=1cm of d] {};
\node[rotate=15]      (1text1*v2)  [above=0.1cm of 1dtext1*] {\scriptsize $(f_1,2e_1)$};

\node			      (dummy21v2)  [left=1.2cm of 1v2] {};
\node[roundnode]      (21v2)       [below=1.2cm of dummy21v2] {};
\node			      (dummy22v2)  [left=0.8cm of 21v2] {};
\node[roundnode]      (22v2)       [below=0.8cm of dummy22v2] {};
\node			      (2dummydotsv2) [left=0.4cm of 22v2] {};
\node[rotate=45]      (2dotsv2)    [below=0.4cm of 2dummydotsv2] {$\dots$};
\node			      (dummy2-1v2) [left=1cm of 22v2] {};
\node[roundnode]      (2-1v2)      [below=1cm of dummy2-1v2] {};
\node			      (dummy2-2v2) [left=0.8cm of 2-1v2] {};
\node[roundnode]      (2-2v2)      [below=0.8cm of dummy2-2v2] {};
\node			      (2dtext1*v2) [left=1.2cm of dv2] {};
\node[rotate=60]      (2text1*v2)  [below=1.0cm of 2dtext1*v2] {\scriptsize $(f_2,2e_2)$};
\node			      (2dtext1v2)  [left=0.7cm of dv2] {};
\node[rotate=45]	  (2text1v2)   [below=1.25cm of 2dtext1v2] {\scriptsize $(2f_2,e_2)$};

\node[roundnode]      (-11v2)      [right=2cm of dv2] {};
\node[roundnode]      (-12v2)      [right=1cm of -11v2] {};
\node			      (-1dotsv2)    [right=0.2cm of -12v2] {$\dots$};
\node[roundnode]      (-1-1v2)     [right=1.2cm of -12v2] {};
\node[roundnode]      (-1-2v2)     [right=1cm of -1-1v2] {};
\node			      (-1dtext1v2)  [right=1.3cm of dv2] {};
\node			      (-1ddtext1v2) [above=0.1cm of -1dtext1v2] {};
\node[rotate=15]      (-1text1v2)   [below=0.3cm of -1ddtext1v2] {\scriptsize $(e_{t},2f_{t})$};
\node			      (-1dtext1*v2) [right=1cm of dv2] {};
\node[rotate=345]     (-1text1*v2)  [above=0.1cm of -1dtext1*v2] {\scriptsize $(2e_{t},f_{t})$};

\node			      (dummy-21v2)  [right=1.2cm of 1v2] {};
\node[roundnode]      (-21v2)       [below=1.2cm of dummy-21v2] {};
\node			      (dummy-22v2)  [right=0.8cm of -21v2] {};
\node[roundnode]      (-22v2)       [below=0.8cm of dummy-22v2] {};
\node			      (-2dummydotsv2) [right=0.4cm of -22v2] {};
\node[rotate=315]     (-2dotsv2)    [below=0.4cm of -2dummydotsv2] {$\dots$};
\node			      (dummy-2-1v2) [right=1cm of -22v2] {};
\node[roundnode]      (-2-1v2)      [below=1cm of dummy-2-1v2] {};
\node			      (dummy-2-2v2) [right=0.8cm of -2-1v2] {};
\node[roundnode]      (-2-2v2)      [below=0.8cm of dummy-2-2v2] {};
\node			      (-2dtext1*v2) [right=1.15cm of dv2] {};
\node[rotate=300]	  (-2text1*v2)  [below=0.9cm of -2dtext1*v2] {\tiny $(2e_{t-1},f_{t-1})$};
\node			      (-2dtext1v2)  [right=0.6cm of dv2] {};
\node[rotate=315]	  (-2text1v2)   [below=1.15cm of -2dtext1v2] {\tiny $(e_{t-1},2f_{t-1})$};

\begin{scope}[decoration={
    markings,
    mark=at position 0.6 with {\arrow{Straight Barb}}}
    ]
\draw (1v2) -- (11v2);
\draw (1v2) -- (21v2);
\draw (1v2) -- (-11v2);
\draw (1v2) -- (-21v2);
\draw (1*v2) -- (11v2);
\draw (1*v2) -- (21v2);
\draw (1*v2) -- (-11v2);
\draw (1*v2) -- (-21v2);
\draw (11v2) -- (12v2);
\draw (1-1v2) -- (1-2v2);
\draw (21v2) -- (22v2);
\draw (2-1v2) -- (2-2v2);
\draw (-11v2) -- (-12v2);
\draw (-1-1v2) -- (-1-2v2);
\draw (-21v2) -- (-22v2);
\draw (-2-1v2) -- (-2-2v2);
\end{scope}
\begin{scope}[decoration={
    markings,
    mark=at position 0.75 with {\arrow{Straight Barb[length=3mm, width=3mm]}}}
    ] 
\draw[densely dotted, double, postaction={decorate}] (1*v2) -- (1v2);
\end{scope}
\end{tikzpicture}
\caption{Dynkin diagram for a  reflection group of canonical type in the case $\varepsilon=2$.}\label{epsilon2}
\end{figure}

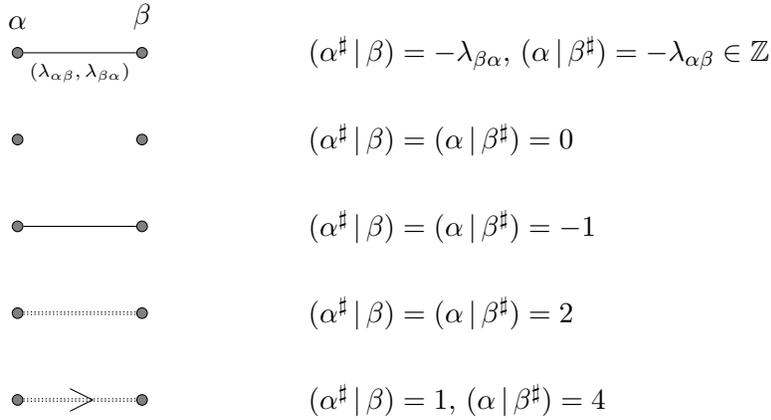
\begin{figure}[H]
\begin{tikzpicture}[
roundnode/.style={circle, draw=black, fill=black!50, inner sep=0pt, minimum width=4pt}
]
\node[roundnode]      (e0a)                              {};
\node[roundnode]        (e0b)       [right=1.5cm of e0a] {};
\node        (e0text)       [right=2cm of e0b] {$(\alpha^\sharp\,\vert\,\beta)=(\alpha\,\vert\,\beta^\sharp)=0$};

\node[roundnode]      (e1a)       [below=1cm of e0a]     {};
\node[roundnode]        (e1b)       [right=1.5cm of e1a] {};
\node       (e1text)       [right=2cm of e1b] {$(\alpha^\sharp \,\vert\,\beta)=(\alpha\,\vert\,\beta^\sharp)=-1$};

\node[roundnode]      (e*a)         [below=1cm of e1a]   {};
\node[roundnode]        (e*b)       [right=1.5cm of e*a] {};
\node        (e*text)       [right=2cm of e*b] {$(\alpha^\sharp \,\vert\,\beta)=(\alpha\,\vert\,\beta^\sharp) =2$};

\node[roundnode]      (e*2a)         [below=1cm of e*a]   {};
\node[roundnode]        (e*2b)       [right=1.5cm of e*2a] {};
\node        (e*2text)       [right=2cm of e*2b] {$(\alpha^\sharp\,\vert\,\beta)=1,\,(\alpha\,\vert\,\beta^\sharp)=4$};

\node[roundnode]      (exa)         [above=1cm of e0a]     {};
\node[roundnode]        (exb)       [right=1.5cm of exa] {};
\node       (extext)       [right=2cm of exb] {$(\alpha^\sharp\,\vert\,\beta)=-\lambda_{\beta\alpha},\,(\alpha\,\vert\, \beta^\sharp)=-\lambda_{\alpha\beta} \in \mathbb{Z}$};

\node  			   (ea)       [above=0.1cm of exa] {$\alpha$};
\node  			   (eb)       [above=0.1cm of exb] {$\beta$};

\begin{scope}[decoration={
    markings,
    mark=at position 0.6 with {\arrow{Straight Barb}}}
    ] 
\draw (e1a) -- (e1b);
\draw (exa) to node[below] {\tiny$(\lambda_{\alpha\beta},\lambda_{\beta\alpha})$} (exb);
\draw[densely dotted, double] (e*a.east) -- (e*b.west);
\end{scope}
\begin{scope}[decoration={
    markings,
    mark=at position 0.65 with {\arrow{Straight Barb[length=3mm, width=3mm]}}}
    ] 
\draw[densely dotted, double, postaction={decorate}] (e*2a.east) -- (e*2b.west);
\end{scope}
\end{tikzpicture}
\caption{The conventions giving the symmetric bilinear form between two simple roots 
$\alpha,\beta$ in a canonical lattice.}\label{rules}
\end{figure}

The integers $t$, $e_i$ and $f_i$, $1 \leq i \leq t$ that appear in the Dynkin diagrams of canonical type correspond to the combinatorial data of the symbol $\sigma$. The middle vertices of these diagrams are $\alpha_0$ and $\alpha
_{0^*}$. The edge relating them is the only double-dotted edge in the diagram. Moreover, we have $t$ arms of the form $\bigl(\alpha_0,\alpha_{(i,1)}, \dots, \alpha_{(i,p_i-1)}\bigr)$ and $t$ arms of the form $\bigl(\alpha_{0^*},\alpha_{(i,1)}, \dots, \alpha_{(i,p_i-1)}\bigr)$  for $1\leq i \leq t$. Each of these arms has $p_i$ vertices, where $1 \leq i \leq t$. Each vertex $\alpha_0$ and $\alpha_{0^*}$ is related to $\alpha_{(1,1)}$, $\alpha_{(2,1)}, \dots,$ and $\alpha_{(t,1)}$ by a decorated edge. All other edges of the diagram are not decorated.

\begin{definition}\label{D:DynkinDiagramQuotient}
    Consider the generalized Coxeter datum $(W_\circ,S_\circ)$ introduced in Section \ref{SubsectionQuotientCoxGrp}. We define the \emph{Dynkin diagram} associated with $(W_\circ,S_\circ)$ as the diagram with vertices in bijection with the simple roots $R_\circ$ and the edges between two simple roots illustrate the symmetric bilinear from following the conventions in Figure \ref{rules}. It has the Dynkin diagram of canonical type given in Figure \ref{rootDiagramQuotient}. We obtain the Dynkin diagram given in Figure \ref{rootDiagramQuotient}. Observe that it has the shape of a star. It is obtained by removing $\alpha_{0^*}$ from the diagrams of Figures \ref{epsilon1} and \ref{epsilon2}. It has $t$ arms of the form $\bigl(\alpha_0,\alpha_{(i,1)}, \dots, \alpha_{(i,p_i-1)}\bigr)$ for $1\leq i \leq t$. The only decorated edges are the one having $\alpha_0$ as a vertex.
\end{definition}

\begin{figure}[H]
\noindent
\begin{tikzpicture}[
roundnode/.style={circle, draw=black, fill=black!50, inner sep=0pt, minimum width=4pt}
]
\node[roundnode]      (1q)                            {};
\node				  (dotsq)      [below=2.5cm of 1q] {$\dots$};
\node				  (1nameq)	  [above=0cm of 1q] {\scriptsize $\alpha_0$};

\node[roundnode]      (11q)       [left=2cm of 1q] {};
\node[roundnode]      (12q)       [left=1cm of 11q] {};
\node			      (1dotsq)    [left=0.2cm of 12q] {$\dots$};
\node[roundnode]      (1-1q)      [left=1.2cm of 12q] {};
\node[roundnode]      (1-2q)      [left=1cm of 1-1q] {};
\node			      (1dtext1q)  [left=1.0cm of 1q] {};
\node			      (1ddtext1q) [above=0.1cm of 1dtext1q] {};
\node				  (11nameq)	  [above=0cm of 11q] {\scriptsize $\alpha_{(1,1)}$};
\node				  (1-2nameq)  [above=0cm of 1-2q] {\tiny $\alpha_{(1,p_1-1)}$};

\node			      (dummy21q)  [left=1.2cm of 1q] {};
\node[roundnode]      (21q)       [below=1.2cm of dummy21q] {};
\node			      (dummy22q)  [left=0.8cm of 21q] {};
\node[roundnode]      (22q)       [below=0.8cm of dummy22q] {};
\node			      (2dummydotsq) [left=0.4cm of 22q] {};
\node[rotate=45]      (2dotsq)    [below=0.4cm of 2dummydotsq] {$\dots$};
\node			      (dummy2-1q) [left=1cm of 22q] {};
\node[roundnode]      (2-1q)      [below=1cm of dummy2-1q] {};
\node			      (dummy2-2q) [left=0.8cm of 2-1q] {};
\node[roundnode]      (2-2q)      [below=0.8cm of dummy2-2q] {};
\node			      (2dtext1q)  [left=0.5cm of 1q] {};
\node[rotate=45]	  (2text1q)   [below=0.65cm of 2dtext1q] {\scriptsize $(\varepsilon f_2,e_2)$};
\node				  (21dnameq)  [right=0.1cm of 21q] {};
\node				  (21nameq)	  [below=0cm of 21dnameq] {\scriptsize $\alpha_{(2,1)}$};
\node				  (2-2nameq)  [below=0cm of 2-2q] {\scriptsize $\alpha_{(2,p_2-1)}$};

\node[roundnode]      (-11q)      [right=2cm of 1q] {};
\node[roundnode]      (-12q)      [right=1cm of -11q] {};
\node			      (-1dotsq)    [right=0.2cm of -12q] {$\dots$};
\node[roundnode]      (-1-1q)     [right=1.2cm of -12q] {};
\node[roundnode]      (-1-2q)     [right=1cm of -1-1q] {};
\node				  (-11nameq)  [above=0cm of -11q] {\scriptsize $\alpha_{(t,1)}$};
\node				  (-1-2dnameq) [left=0cm of -1-2q] {};
\node				  (-1-2nameq)  [above=0cm of -1-2dnameq] {\tiny $\alpha_{(t,p_t-1)}$};

\node			      (dummy-21q)  [right=1.2cm of 1q] {};
\node[roundnode]      (-21q)       [below=1.2cm of dummy-21q] {};
\node			      (dummy-22q)  [right=0.8cm of -21q] {};
\node[roundnode]      (-22q)       [below=0.8cm of dummy-22q] {};
\node			      (-2dummydotsq) [right=0.4cm of -22q] {};
\node[rotate=315]     (-2dotsq)    [below=0.4cm of -2dummydotsq] {$\dots$};
\node			      (dummy-2-1q) [right=1cm of -22q] {};
\node[roundnode]      (-2-1q)      [below=1cm of dummy-2-1q] {};
\node			      (dummy-2-2q) [right=0.8cm of -2-1q] {};
\node[roundnode]      (-2-2q)      [below=0.8cm of dummy-2-2q] {};
\node			      (-2dtext1q)  [right=0.5cm of 1q] {};
\node[rotate=315]	  (-2text1q)   [below=0.65cm of -2dtext1q] {\tiny $(e_{t-1},\varepsilon f_{t-1})$};
\node				  (-21dnameq)  [above=0cm of -21q] {};
\node				  (-21ddnameq) [below=0.05cm of -21dnameq] {};
\node				  (-21nameq)   [right=0cm of -21ddnameq] {\scriptsize $\alpha_{(t-1,1)}$};
\node				  (-2-2dnameq) [right=0.1cm of -2-2q] {};

\node				  (-2-2nameq)  [below=0cm of -2-2dnameq] {\scriptsize $\alpha_{(t-1,p_{t-1}-1)}$};

\begin{scope}[decoration={
    markings,
    mark=at position 0.6 with {\arrow{Straight Barb}}}
    ]
\draw (1q) to node[below] {\scriptsize $(\varepsilon f_1,e_1)$} (11q);
\draw (1q) -- (21q);
\draw (1q) to node[below] {\scriptsize $(e_t,\varepsilon f_t)$} (-11q);
\draw (1q) -- (-21q);
\draw (11q) -- (12q);
\draw (1-1q) -- (1-2q);
\draw (21q) -- (22q);
\draw (2-1q) -- (2-2q);
\draw (-11q) -- (-12q);
\draw (-1-1q) -- (-1-2q);
\draw (-21q) -- (-22q);
\draw (-2-1q) -- (-2-2q);
\end{scope}
\end{tikzpicture}
    \caption{The star-like Dynkin diagram associated with $W_\circ$.}
    \label{rootDiagramQuotient}
\end{figure}
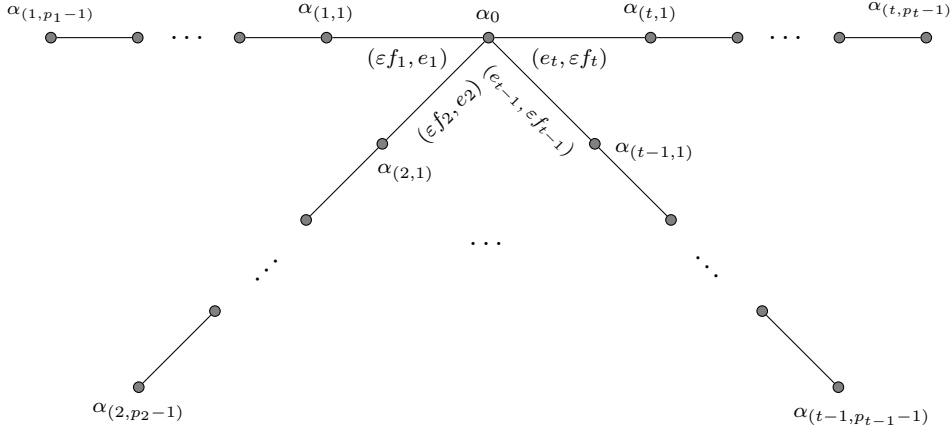

Let us first discuss the domestic case. Assume that the exceptional hereditary curve $\XX$ is domestic, i.e. $\delta(\XX) < 0$. It is well-known that in this case there exists a tame hereditary algebra $\Theta$ such that
$$D^b\bigl(\Coh(\XX)\bigr) \simeq D^b\bigl(\Theta\mbox{--}\mathsf{mod}\bigr).$$
Conversely, for any tame hereditary algebra $\Theta$, there exists a derived equivalent domestic exceptional hereditary curve $\XX$; see for instance \cite{KussinMemoirs}. It follows that reflection groups of domestic canonical type are precisely affine Weyl groups. A comparison of the affine root systems with the corresponding symbols can be found in \cite{LenzingKTheory}. In Table \ref{tab:domestic}, we provide a comparison between reduced symbols and affine Coxeter groups. We use the notation of \cite{Bourbaki} for affine Coxeter groups.

\begin{table}[h]
    \centering
\[\arraycolsep=1.5pt
\begin{array}{c|c|c|c|c|c|c|c|c|c}
\widetilde{A}_{p} & \widetilde{A}_{p_1+p_2-1} & \widetilde{B}_{p+1} & \widetilde{C}_{p} & \widetilde{D}_{p+2} & \widetilde{E}_{6} & \widetilde{E}_{7} & \widetilde{E}_{8} & \widetilde{F}_{4} & \widetilde{G}_{2}\\ \hline
\footnotesize{\left(\begin{array}{@{}ccc@{}}
     p
\end{array}\right)} &
\footnotesize{\left(\begin{array}{@{}cc@{}}
     p_1&p_2
\end{array}\right)} &
\footnotesize{\left(\begin{array}{@{}cc@{}}
     2&p\\2&1
\end{array}\right)} &
\footnotesize{\left(\begin{array}{@{}ccc@{}}
    p\\2
\end{array}\right)} &
\footnotesize{\left(\begin{array}{@{}cccc@{}}
     2 & 2 & p
\end{array}\right)} &
\footnotesize{\left(\begin{array}{@{}ccc@{}}
     2 & 3 & 3
\end{array}\right)}&
\footnotesize{\left(\begin{array}{@{}ccc@{}}
     2 & 3 & 4
\end{array}\right)}&
\footnotesize{\left(\begin{array}{@{}ccc@{}}
     2 & 3 & 5
\end{array}\right)} &
\footnotesize{\left(\begin{array}{@{}cc@{}}
     2&3\\1&2
\end{array}\right)} &
\footnotesize{\left(\begin{array}{@{}cc@{}}
     2\\3
\end{array}\right)}
	\end{array}
    \]
    
\caption{A dictionary between affine Coxeter groups and reduced symbols.}
    \label{tab:domestic}
\end{table}

\begin{remark}
As we observe from Table \ref{tab:domestic}, there are different ways to realize the affine Coxeter group $\widetilde{A}$ as a reflection group of canonical type. For example, let $p\geq2$ be any fixed integer. For any partition $p=p_1+p_2$ with $2\leq p_1\leq p_2$, we get a different generalized Coxeter datum $(W,S)$ for the same affine Coxeter group $W \cong \widetilde{A}_{p-1}$ (see the 2nd column of Table \ref{tab:domestic}). 

\end{remark}

\begin{example}
    We provide an example showing how to obtain the affine Coxeter system $\widetilde{A}_p$ in the classical sense from a reflection group of canonical type $W$, whose symbol is $(p)$ for any integer $p \geq 2$ (see the 1st column of Table \ref{tab:domestic}). Let $R$ be the corresponding complete exceptional sequence as in (\ref{E:SimpleRoots}):
    \[R=(\alpha_{(1,p-1)},\dots,\alpha_{(1,1)},\alpha_0,\alpha_{0^*}).\]
    Let $(W,S)$ be the associated generalized Coxeter datum.
    Apply the braid $\sigma_1^{-1}\cdots\sigma_p^{-1}$ to $R$ to obtain a new complete exceptional sequence
    \[R'=\left(\alpha_{0^*}-2\alpha_0+\sum_{i=1}^{p-1}\alpha_{(1,i)},\,\,\alpha_{(1,p-1)},\dots,\alpha_{(1,1)},\alpha_0\right)=:(e_1,\dots,e_{p+1}).\]
    Then we have $(e_i,e_i)=2$, $(e_i,e_j)=-1$ if and only if $i=j\pm1$ in $\ZZ_{p+1}$ and $(e_i,e_j)=0$ else. We get that the obtained generalized Coxeter datum $(W,S')$ is an affine Coxeter system in the classical sense. The corresponding Coxeter element is
    \[c=s_{e_1}\cdots s_{e_{p+1}}.\]
    This is precisely the Coxeter element that gives rise to a non-crossing partition \emph{lattice} and thus a \emph{Garside structure}; see \cite{Digne, McCammondSulway, PaoliniSalvetti}. The generalized Coxeter datum $(W,S)$ has also been used in the literature; see for example \cite{Shi} and \cite{NeaimeGarside}. 
\end{example}

Now, we discuss the tubular case. Let $\XX$ be tubular, i.e. $\delta(\XX) =  0$. Then the root system associated with $\XX$ is an elliptic root system of Saito of \emph{codimension one} and the associated reflection group is an elliptic Weyl group. Conversely, for any such elliptic root system, there exists a field $\kk$ such that the elliptic root system arises from an appropriate exceptional curve of tubular type over $\kk$; see \cite{LenzingExceptionalCurve, KussinMemoirs}.

In \cite{SaitoI}, elliptic root systems - also called marked extended affine root systems (abbreviated mEARS) - are classified via elliptic root diagrams. We provide a dictionary between those mEARS and tubular symbols. This is done in Tables \ref{tab:eps1} and \ref{tab:eps2} and can easily be checked by writing down the Dynkin diagram of canonical type explicitly. The first row of these tables indicates the notation used by Saito \cite{SaitoI} of the elliptic root diagram corresponding to a mEARS, and the second row indicates the corresponding symbol.

\begin{table}[h]
    \centering
\[\arraycolsep=1.5pt
\begin{array}{c|c|c|c|c|c|c|c|c|c}
BC_1^{(2,1)} & A_1^{(1,1)*} & BC_1^{(2,4)} & B_2^{(2,1)} & BC_2^{(2,2)}(1) & C_2^{(1,2)} & G_2^{(3,1)} & G_2^{(1,3)} & G_2^{(1,1)} & G_2^{(3,3)} \\ \hline
\footnotesize{\left(\begin{array}{@{}c@{}}
     2\\
     4 
\end{array}\right)} & 
\footnotesize{\left(\begin{array}{@{}c@{}}
     2\\
     4\\
     2
\end{array}\right)} & 
\footnotesize{\left(\begin{array}{@{}c@{}}
     2\\
     4\\
     4
\end{array}\right)} &
\footnotesize{\left(\begin{array}{@{}cc@{}}
     2 & 2\\
     2 & 2
\end{array}\right)} &
\footnotesize{\left(\begin{array}{@{}cc@{}}
     2 & 2\\
     2 & 2\\
     2 & 1
\end{array}\right)} &
\footnotesize{\left(\begin{array}{@{}cc@{}}
     2 & 2\\
     2 & 2\\
     2 & 2
\end{array}\right)} &
\footnotesize{\left(\begin{array}{@{}c@{}}
     3\\
     3
\end{array}\right)} &
\footnotesize{\left(\begin{array}{@{}c@{}}
     3\\
     3\\
     3
\end{array}\right)} &
\footnotesize{\left(\begin{array}{@{}cc@{}}
     2 & 2\\
     1 & 3
\end{array}\right)} & 
\footnotesize{\left(\begin{array}{@{}cc@{}}
     2 & 2\\
     1 & 3\\
     1 & 3
\end{array}\right)}
	\end{array}
    \]
\[\arraycolsep=1.5pt
\begin{array}{c|c|c|c|c|c|c|c|c|c}
B_3^{(1,1)} & C_3^{(2,2)} & F_4^{(2,1)} & F_4^{(1,2)} & F_4^{(1,1)} & F_4^{(2,2)} & D_4^{(1,1)} & E_6^{(1,1)} & E_7^{(1,1)} & E_8^{(1,1)} \\ \hline
\footnotesize{\left(\begin{array}{@{}ccc@{}}
     2 & 2 & 2\\
     1 & 1 & 2
\end{array}\right)} &
\footnotesize{\left(\begin{array}{@{}ccc@{}}
     2 & 2 & 2\\
     1 & 1 & 2 \\
     1 & 1 & 2
\end{array}\right)} &
\footnotesize{\left(\begin{array}{@{}cc@{}}
     4 & 2\\
     2 & 1
\end{array}\right)} &
\footnotesize{\left(\begin{array}{@{}cc@{}}
     4 & 2\\
     2 & 1\\
     2 & 1
\end{array}\right)} &
\footnotesize{\left(\begin{array}{@{}cc@{}}
     3 & 3\\
     1 & 2
\end{array}\right)} &
\footnotesize{\left(\begin{array}{@{}cc@{}}
     3 & 3\\
     1 & 2\\
     1 & 2
\end{array}\right)} &
\footnotesize{\left(\begin{array}{@{}cccc@{}}
     2 & 2 & 2 & 2
\end{array}\right)} &
\footnotesize{\left(\begin{array}{@{}ccc@{}}
     3 & 3 & 3
\end{array}\right)} &
\footnotesize{\left(\begin{array}{@{}ccc@{}}
     4 & 4 & 2
\end{array}\right)} &
\footnotesize{\left(\begin{array}{@{}ccc@{}}
     6 & 3 & 2
\end{array}\right)}
	\end{array}
    \]
\caption{Case $\varepsilon=1$: A dictionary between mEARS and symbols.}
    \label{tab:eps1}
\end{table}

\begin{table}[h]
    \centering
\[\arraycolsep=3.5pt
\begin{array}{c|c|c}
BC_1^{(2,1)} & BC_1^{(2,4)} & BC_2^{(2,2)}(1) \\ \hline
\footnotesize{\left(\begin{array}{@{}c|c@{}}
     2 & \\
     2 & 2\\
     2 & 
\end{array}\right)} &
\footnotesize{\left(\begin{array}{@{}c|c@{}}
     \begin{array}{@{}c@{}}
     2\\2
\end{array} & 2
\end{array}\right)} &
\footnotesize{\left(\begin{array}{@{}cc|c@{}}
     2 & 2 & 2
\end{array}\right)}
	\end{array}
    \]
    
\caption{Case $\varepsilon=2$: A dictionary between mEARS and symbols.}
    \label{tab:eps2}
\end{table}

\begin{remark}
    Many of the mEARS of Tables \ref{tab:eps1} and \ref{tab:eps2} yield the same reflection group. Actually, there are only $11$ cases of such groups. For instance, the cases $BC_1^{(2,1)}$, $A_1^{(1,1)*}$, and $BC_1^{(2,4)}$ give the same reflection group of canonical type; see Proposition \ref{P:ReductionCaseEpsilon1}.
\end{remark}

\section{Generalities on hyperbolic extensions}\label{S:Appendix}
Following \cite{SaitoI} and \cite{BaumeisterWegener}, we recall the notions of hyperbolic extensions, first for a pair $(V,B)$ of a finite-dimensional real vector space $V$ and a symmetric bilinear form $B$ on $V$, then for reflection groups. We will subsequently prove some easy but notable results on isomorphisms of these new groups. Note that the hyperbolic extensions are called hyperbolic covers in \cite{BaumeisterWegener}.

\begin{definition}
Let $(V,B)$ be a pair consisting of a finite dimensional real vector space $V$ and a symmetric bilinear form $B$ on $V$. Let $G$ be a subspace of $\Rad(B)$. A \emph{hyperbolic extension} of $(V,B)$ with respect to $G$ is a triple $(\widetilde{V},\widetilde{B},\iota)$ consisting of a real vector space $\widetilde{V}$ of dimension 
$\dim_{\RR}(\widetilde{V})=\dim_{\RR}(V)+\dim_{\RR}\left(\Rad(B)/G\right)$
a symmetric bilinear form $\widetilde{B}$ on $\widetilde{V}$ and an inclusion $\iota:V\to\widetilde{V}$ such that $\widetilde{B}\circ (\iota \times \iota)=B$ and $\Rad(\widetilde{B})=\iota(G)$.
\end{definition}
From this definition of hyperbolic extension, neither existence nor uniqueness is clear. However, the next two results give the desired properties via an explicit construction.
\begin{lemma}
Let $(V,B)$ be a real vector space with symmetric bilinear form and let $G$ be a subspace of $\Rad(B)$. Then there exists a hyperbolic extension $(\widetilde{V},\widetilde{B},\iota)$ of $(V,B)$ with respect to $G$.
\end{lemma}
\begin{proof}
Let $n=\dim_{\RR}(V)$ and $m=\dim_{\RR}\left(\Rad(B)/G\right)$. Let $(v_1,\dots,v_n)$ be a basis of $V$ such that $(v_{n-m+1},\dots,v_n)$ is a basis of $H\subseteq\Rad(B)$ with $G\oplus H=\Rad(B)$ and $(v_1,\dots,v_{n-m})$ is a basis of $H'\subseteq V$ such that $H'\oplus H=V$ and $G\subseteq H'$. Then we define $\widetilde{V}$ to be the real vector space generated by $(v_1,\dots,v_n,v_{n-m+1}',\dots,v_n')$. Further, we define a symmetric bilinear form $\widetilde{B}$ on $\widetilde{V}$ via
\[\widetilde{B}(v_i,v_j)=B(v_i,v_j),\quad\widetilde{B}(v_i',v_j')=0,\quad\widetilde{B}(v_i,v_j')=\begin{cases}1&\text{if }i=j,\\0&\text{if }i\neq j.\end{cases}\]
Finally, we define $\iota:V\to\widetilde{V}$ to be the obvious inclusion given by $\iota(v_i)=v_i$. Then clearly $B=\widetilde{B}\circ\iota$. More precisely, the Gram matrix of $\widetilde{B}$ with respect to the basis $(v_1,\dots,v_n,v_{n-m+1}',\dots,v_n')$ is given by
\[[\widetilde{B}]=\begin{pmatrix}[B|_{H'\times H'}]&0_{n-m,m}&0_{n-m,m}\\0_{m,n-m}&0_{m,m}&I_m\\0_{m,n-m}&I_m&0_{m,m}\end{pmatrix}.\]
Thus its rank is given by $\rk[\widetilde{B}]=\rk[B\big|_{H'\times H'}]+2m=\dim_{\RR}(\widetilde{V})-\dim_{\RR}(G)$. As $\iota(G)\subseteq\Rad(\widetilde{B})$ is clear, the equality follows by dimension reasons.
\end{proof}

\begin{lemma}\label{hypRealisation}
Let $(\widetilde{V},\widetilde{B},\iota)$ be a hyperbolic extension of $(V,B)$ with respect to $G\subseteq\Rad(B)$ with $n=\dim_{\RR}(V)$ and $m=\dim_{\RR}\left(\Rad(B)/G\right)$. For any basis $v_1,\dots,v_n$ of $V$ such that $G\oplus\langle v_{n-m+1},\dots,v_n\rangle=\Rad(B)$ and $G\subseteq H'$ there exist $v_{n-m+1}',\dots,v_n'\in\widetilde{V}$ such that $\iota(v_1),\dots,\iota(v_n),v_{n-m+1}',\dots,v_n'$ is a basis of $\widetilde{V}$ and
\begin{equation}\label{E:hypExtMatrix}[\widetilde{B}]=\begin{pmatrix}[B\big|_{H'\times H'}]&0_{n-m,m}&0_{n-m,m}\\0_{m,n-m}&0_{m,m}&I_m\\0_{m,n-m}&I_m&0_{m,m}\end{pmatrix}\end{equation}
with respect to this basis, where $H'=\langle v_1,\dots,v_{n-m}\rangle$.
\end{lemma}
\begin{proof}
Let us first reformulate the statement: Let $V\subseteq\widetilde{V}$ be a $n$-dimensional subspace of a $(n+m)$-dimensional real vector space. Let $\widetilde{B}$ be a symmetric bilinear form on $\widetilde{V}$ and $B=\widetilde{B}\bigl|_{V\times V}$. Assume that $\Rad(\widetilde{B})\subseteq\Rad(B)$ such that $\dim_{\RR}\bigl(\Rad(B)\bigr)=
\dim_{\RR}\bigl(\Rad(\widetilde{B})\bigr)+m$. Then for any basis $v_1,\dots,v_n$ of $V$ with $\langle v_{n-m+1},\dots,v_n\rangle\oplus\Rad(\widetilde{B})=\Rad(B)$ and $\Rad(\widetilde B)\subseteq\langle v_1,\dots,v_{n-m}\rangle$, there exist $v'_{n-m+1},\dots,v'_n$ such that $v_{1},\dots,v_n,v'_{n-m+1},\dots,v'_n$ is a basis of $\widetilde{V}$ and $[\widetilde{B}]$ is as in (\ref{E:hypExtMatrix}).

From this formulation, it is clear that we can work inductively. Thus the only interesting case is $m=1$. We will compute the needed $v_n'$ step by step. Start with any $v_n^{(4)}\in\widetilde{V}$ such that $v_1,\dots,v_n,v_n^{(4)}$ is a basis of $\widetilde{V}$. Now replace $v_1,\dots,v_{n-1}$ by a basis $\tilde{v}_1,\dots,\tilde{v}_{n-1}$ of $H'=\langle v_1,\dots,v_{n-1}\rangle$ such that $\bigl[B\bigl|_{H'\times H'}\bigr]=I_p\oplus-I_q\oplus0_r$ for some $p,q,r\in\mathbb{N}_0$. Then $\Rad(\widetilde{B})=\langle\tilde{v}_{p+q+1},\dots,\tilde{v}_{n-1}\rangle$. We have $\widetilde{B}(\tilde{v}_i,v_n^{(4)})=0$ for any $p+q+1\leq i\leq n-1$. Now set
\[v_n^{(3)}=v_n^{(4)}-\sum_{i=1}^p\widetilde{B}(\tilde{v}_i,v_n^{(4)})\tilde{v}_i+\sum_{i=p+1}^{p+q}\widetilde{B}(\tilde{v}_i,v_n^{(4)})\tilde{v}_i.\]
Now $\widetilde{B}(\tilde{v}_i,v_n^{(3)})=\widetilde{B}(\tilde{v}_i,v_n^{(4)})-\widetilde{B}(\tilde{v}_i,v_n^{(4)})=0$ for any $1\leq i\leq p+q$ and thus any $1\leq i\leq n-1$. However, $\widetilde{B}(v_n,v_n^{(3)})\neq0$ as $v_n\in\Rad(B)\setminus\Rad(\widetilde{B})$. Therefore set
\[v_n^{(2)}=v_n^{(3)}-\frac{\widetilde{B}(v_n^{(3)},v_n^{(3)})}{2\widetilde{B}(v_n,v_n^{(3)})}v_n.\]
Then $\widetilde{B}(\tilde{v}_i,v_n^{(2)})=\widetilde{B}(\tilde{v}_i,v_n^{(3)})=0$ for $1\leq i\leq n-1$ as $v_n\in\Rad(B)$, and $$\widetilde{B}(v_n^{(2)},v_n^{(2)})=\widetilde{B}(v_n^{(3)},v_n^{(3)})-2\frac{\widetilde{B}(v_n^{(3)},v_n^{(3)})}{2\widetilde{B}(v_n,v_n^{(3)})}\widetilde{B}(v_n,v_n^{(3)})+\widetilde{B}(v_n,v_n)=0.$$
Finally, set $v_n'=\dfrac{1}{\widetilde{B}(v_n,v_n^{(2)})}v_n^{(2)}$. This finishes the proof.
\end{proof}

\begin{corollary}\label{hypMono}
Let $(\widetilde{V}_G,\widetilde{B}_G,\iota_G)$ and $(\widetilde{V}_H,\widetilde{B}_H,\iota_H)$ be hyperbolic extensions of $(V,B)$ with respect to subspaces $G$ and $H$ such that $H\subseteq G\subseteq\Rad(B)$. Then there exists a monomorphism $\varphi:\widetilde{V}_G\to\widetilde{V}_H$ with $\widetilde{B}_G=\widetilde{B}_H\circ\varphi$ and $\varphi\circ\iota_G=\iota_H$.
\end{corollary}
\begin{proof}
    Let $v_1,\dots,v_n$ be a basis of $V$ such that $\langle v_{n-s+1},\dots,v_n\rangle\oplus G=\Rad(B)$ and $\langle v_{n-t+1},\dots,v_n\rangle\oplus H=\Rad(B)$ (with $s\leq t$). Then we can find $v^{(G)}_{n-s+1},\dots,v^{(G)}_n\in\widetilde{V}_G$, $v^{(H)}_{n-t+1},\dots,v^{(H)}_n\in\widetilde{V}_H$ as in Lemma \ref{hypRealisation}. Define $\varphi:\widetilde{V}_G\to\widetilde{V}_H$ by $\varphi(\iota_G(v_i))=\iota_H(v_i)$ for any $1\leq i\leq n$ and $\varphi(v^{(G)}_i)=v^{(H)}_i$ for any $n-s+1\leq i\leq n$.
\end{proof}

Let us now define hyperbolic extensions of reflection groups.
\begin{definition}
Let $R=(\alpha_1,\dots,\alpha_n)$ be a set of non-isotropic vectors in $(V,B)$. Let $(W,T,c)$ be the corresponding generalized dual Coxeter datum. Further, let $G\subseteq\Rad(B)$ and $(\widetilde{V},\widetilde{B},\iota)$ a hyperbolic extension of $(V,B)$ with respect to $G$. The generalized dual Coxeter datum $(\widetilde{W},\widetilde{T},\tilde{c})$ associated to $\widetilde{R}=(\iota(\alpha_1),\dots,\iota(\alpha_n))$ in $(\widetilde{V},\widetilde{B})$ is called a \emph{hyperbolic ~extension} of $(W,T,c)$ with respect to $G$.
\end{definition}
Let us once again consider uniqueness of hyperbolic extensions.
\begin{lemma}
Let $(W,T,c)$ be a generalized dual Coxeter datum in $(V,B)$ and let $H\subseteq G\subseteq\Rad(B)$. Let $(\widetilde{W}_G,\widetilde{T}_G,\tilde{c}_G)$ and $(\widetilde{W}_H,\widetilde{T}_H,\tilde{c}_H)$ be hyperbolic extensions of $(W,T,c)$ with respect to $G$ and $H$ respectively. Then there exists a group epimorphism $\pi:\widetilde{W}_H\to \widetilde{W}_G$ such that $\pi(\tilde{c}_H)=\tilde{c}_G$ and $\pi$ restricts to a set isomorphism $\pi:\widetilde{T}_H\to\widetilde{T}_G$.
\end{lemma}
\begin{proof}
Let $R=(\alpha_1,\dots,\alpha_n)$ be the simple roots of $(W,T,c)$ and let $(\widetilde{V}_G,\widetilde{B}_G,\iota_G)$ and $(\widetilde{V}_H,\widetilde{B}_H,\iota_H)$ be the hyperbolic extensions of $(V,B)$ corresponding to $(\widetilde{W}_G,\widetilde{T}_G,\tilde{c}_G)$ and $(\widetilde{W}_H,\widetilde{T}_H,\tilde{c}_H)$. Let $\varphi:\widetilde{V}_G\to\widetilde{V}_H$ be the monomorphism from Corollary \ref{hypMono}. More precisely, $\varphi$ is split, i.e. there exists an epimorphism $\psi:\widetilde{V}_H\to\widetilde{V}_G$ with $\psi\circ\varphi=\id_{\widetilde{V}_G}$. We compute for any $x\in\widetilde{V}_G$
\begin{align*}
    s_{\iota_G(\alpha_i)}(x)&=x-\frac{2\widetilde{B}_G(x,\iota_G(\alpha_i))}{\widetilde{B}_G(\iota_G(\alpha_i),\iota_G(\alpha_i))}\iota_G(\alpha_i)\\
    &=\psi\left(\varphi(x)-\frac{2\widetilde{B}_G(x,\iota_G(\alpha_i))}{\widetilde{B}_G(\iota_G(\alpha_i),\iota_G(\alpha_i))}\varphi(\iota_G(\alpha_i))\right)\\
    &=\psi\left(\varphi(x)-\frac{2\widetilde{B}_H(\varphi(x),\varphi\circ\iota_G(\alpha_i))}{\widetilde{B}_H(\varphi\circ\iota_G(\alpha_i),\varphi\circ\iota_G(\alpha_i))}\varphi\circ\iota_G(\alpha_i)\right)\\
    &=\psi\left(s_{\varphi\circ\iota_G(\alpha_i)}(\varphi(x))\right)\\
    &=\psi\left(s_{\iota_H(\alpha_i)}(\varphi(x))\right).
\end{align*}
Therefore $s_{\iota_G(\alpha_i)}=\psi\circ s_{\iota_H(\alpha_i)}\circ\varphi$. From the explicit construction in Corollary \ref{hypMono} we know that $\Image(\varphi)$ is a $s_{\iota_H(\alpha_i)}$ invariant subspace for any $1\leq i\leq n$. Therefore $\varphi\circ\psi|_{\Image(\varphi)}=\id_{\Image(\varphi)}$ implies
\[s_{\iota_G(\alpha_i)}s_{\iota_G(\alpha_j)}=\psi\circ s_{\iota_H(\alpha_i)}\circ\varphi\psi\circ s_{\iota_H(\alpha_j)}\circ\varphi=\psi\circ s_{\iota_H(\alpha_i)}\circ s_{\iota_H(\alpha_j)}\circ\varphi.\]
Thus $\pi:\widetilde{W}_H\to\widetilde{W}_G$, $\pi(w)=\psi w\varphi$ is a morphism of groups with $\pi(s_{\iota_H(\alpha_i)})=s_{\iota_G(\alpha_i)}$. The surjectivity as well as $\pi(\widetilde{T}_H)=\widetilde{T}_G$ and $\pi(\tilde{c}_H)=\tilde{c}_G$ are obvious.
\end{proof}

In particular, two hyperbolic extensions $(\widetilde{W},\widetilde{T},\tilde{c})$ and $(\widetilde{W}',\widetilde{T}',\tilde{c}')$ with respect to the same $G\subseteq\Rad(B)$ are isomorphic, and the isomorphism identifies $\widetilde{T}$ with $\widetilde{T}'$ and $\tilde{c}$ with $\tilde{c}'$. Furthermore, this result allows us to clarify the relation between reflection groups and their hyperbolic extensions.
\begin{corollary}\label{hyperbolicEpimorphism}
Let $(W,T,c)$ be a generalized dual Coxeter datum in $(V,B)$ and let $G\subseteq\Rad(B)$. Let $(\widetilde{W},\widetilde{T},\tilde{c})$ be a hyperbolic extension of $(W,T,c)$ with respect to $G$. Then there exists a group epimorphism $\pi:\widetilde{W}\to W$ such that $\pi(\tilde{c})=c$ and $\pi$ restricts to a set isomorphism $\pi:\widetilde{T}\to T$.
\end{corollary}

For low-dimensional hyperbolic extensions, we can even prove a stronger isomorphism result that is not true in general.
\begin{lemma}\label{L:lowDimExt1}
Let $(W,T,c)$ be a generalized dual Coxeter datum in $(V,B)$ such that  $\dim_{\RR}(\Rad(B))=1$ and let $(\widetilde{W},\widetilde{T},\tilde{c})$ be a hyperbolic extension of $(W,T,c)$ with respect to $\Rad(B)$. Then the epimorphism from Corollary \ref{hyperbolicEpimorphism} is an isomorphism.
\end{lemma}
\begin{proof}
The claim can be shown using explicit matrix calculations. Choose a basis $(v_1,\dots,\linebreak v_{n+1})$ of $V'$ such that $\langle v_1,\dots,v_n\rangle=V$ and
\[[\widetilde{B}]=\begin{pmatrix}
    I_p&0&0&0\\0&-I_q&0&0\\0&0&0&1\\0&0&1&0
\end{pmatrix}\]
for some $p,q\in\mathbb{N}_0$. Assume that $\tilde{w}\in\ker(\pi)\subset\widetilde{W}$. Then $\tilde{w}(v)=v$ for any $v\in V$, i.e.
\[[\tilde{w}]=\begin{pmatrix}
    I_n&\vec{\lambda}\\0&\lambda_{n+1}
\end{pmatrix}\]
for some $\vec{\lambda}\in\mathbb{R}^n$ with entries $\lambda_1,\dots,\lambda_n$. Now $\widetilde{W}\subset O(\widetilde{V},\widetilde{B})$ implies
\[[\widetilde{B}]=[\tilde{w}]^T[\widetilde{B}][\tilde{w}].\]
In other words, $\lambda_1=\dots=\lambda_p=0$, $-\lambda_{p+1}=\dots=-\lambda_{p+q}=0$, $\lambda_{n+1}=1$ and $$\sum_{i=1}^p\lambda_i^2-\sum_{i=p+1}^{p+q}\lambda_i^2+2\lambda_n\lambda_{n+1}=0$$ which shows $\lambda_1=\dots=\lambda_n=0$, $\lambda_{n+1}=1$, i.e. $\tilde{w}=\id_{\widetilde{V}}$.
\end{proof}
Analogously, we have the following result.
\begin{corollary}\label{C:lowDimExt2}
Let $(W,T,c)$ be a generalized dual Coxeter datum in $(V,B)$ and let $G\subseteq\Rad(B)$ with $\dim(G)=1$. Let $(\widetilde{W}_G,\widetilde{T}_G,\tilde{c}_G)$ and $(\widetilde{W}_0,\widetilde{T}_0,\tilde{c}_0)$ be hyperbolic extensions of $(W,T,c)$ with respect to $G$ and 0 respectively. Then there exists a group isomorphism $\varphi:\widetilde{W}_0\to\widetilde{W}_G$ with $\varphi(\widetilde{T}_0)=\widetilde{T}_G$ and $\varphi(\tilde{c}_0)=\tilde{c}_G$.
\end{corollary}
\end{appendix}

\bibliographystyle{alpha}
\bibliography{Bibliography}

\vspace*{1em}

\end{document}